\documentclass[11pt]{article}
\usepackage{authblk}
\usepackage[margin=1in]{geometry}
\usepackage{amsmath,amssymb,amsthm,mathtools}
\usepackage{booktabs,tabularx,array}
\usepackage{enumitem}
\usepackage{placeins}
\usepackage{float}
\usepackage{hyperref}
\usepackage[nameinlink,capitalise]{cleveref}
\newcolumntype{Y}{>{\raggedright\arraybackslash}X}

\numberwithin{equation}{section}
\hypersetup{
  colorlinks=true,
  linkcolor=blue,
  urlcolor=blue,
  citecolor=blue
}
\usepackage{bm}
\usepackage{xcolor}

\theoremstyle{plain}
\newtheorem{theorem}{Theorem}[section]
\newtheorem{proposition}[theorem]{Proposition}
\newtheorem{lemma}[theorem]{Lemma}
\newtheorem{corollary}[theorem]{Corollary}

\theoremstyle{definition}
\newtheorem{definition}[theorem]{Definition}

\newtheorem{assumption}[theorem]{Assumption}

\theoremstyle{remark}
\newtheorem{remark}[theorem]{Remark}

\newcommand{\R}{\mathbb{R}}

\newcommand{\dd}{\,\mathrm{d}}
\newcommand{\Vol}{\mathrm{Vol}}
\newcommand{\E}{\mathbb E}

\newcommand{\KL}{\mathrm{KL}}

\newcommand{\Vspec}{\mathcal V^{\mathrm{spec}}}
\newcommand{\Pcal}{\mathcal P}
\newcommand{\Ucal}{\mathcal U}

\newcommand{\Acal}{\mathcal A}

\newcommand{\Hcal}{\mathcal H}

\newcommand{\Bcal}{\mathcal B}
\newcommand{\norm}[1]{\left\|#1\right\|}

\newcommand{\ip}[2]{\left\langle #1,#2\right\rangle}
\newcommand{\one}{\mathbf 1}

\newcommand{\argmin}{\operatorname*{arg\,min}}

\newcommand{\eps}{\varepsilon}
\newcommand{\cD}{\mathcal D}
\newcommand{\cJ}{\mathcal J}
\newcommand{\cP}{\mathcal P}
\newcommand{\bbone}{\mathbf 1}

\newcommand{\PP}{\mathbb P}

\title{Density Estimation on Compact Manifolds under Intrinsic Spectral Block Variation}

\author{Olga Klopp \\\vskip -0.4 cm ESSEC Business School\\
	Fedor Noskov \\ HSE University 
}

\date{}

\begin{document}
\maketitle

\begin{abstract}
We introduce an intrinsic spectral sparsity model for nonparametric density estimation on compact connected Riemannian manifolds. Instead of penalizing coefficients in an arbitrarily chosen Laplace--Beltrami eigenbasis, we group each complete eigenspace and measure the Hilbert norm of its spectral component. The resulting block-variation space is basis independent and isometry invariant. We establish its structural, atomic, and nonlinear approximation properties and clarify its relation to Sobolev, Besov, and coefficientwise spectral $\ell^1$ classes. We then construct a coordinate-free block-shrinkage estimator and prove a nonasymptotic signal-dependent $L^2$-oracle inequality that adapts to the unknown set of detectable eigenspaces. Under polynomial spectral growth, the risk theory separates the number of spectral blocks from their multiplicities and exhibits two regimes: one driven by a single high-dimensional eigenspace and the other by cumulative spectral complexity. Under matching spectral-growth and nondegeneracy assumptions, corresponding minimax lower bounds show that this multiplicity dependence is intrinsic, with sharp consequences for spheres and the rotation group $SO(3)$. Finally, we develop a positive, normalized, block-penalized exponential spectral sieve for log-densities and derive likelihood oracle inequalities together with expected Kullback--Leibler, Hellinger, and $L^2$ risk bounds. The resulting framework provides a geometry-respecting theory of sparse density estimation that remains invariant under changes of eigenbasis.
\end{abstract}



\section{Introduction}
\label{sec:introduction}

Spectral methods are natural tools for nonparametric inference on a
Riemannian manifold because the Laplace--Beltrami operator is determined by
the geometry rather than by a coordinate chart. Smoothness can therefore be
expressed intrinsically through spectral weights. Sparsity is more delicate:
when an eigenvalue has multiplicity greater than one, its individual
eigenfunctions are not canonical, and coefficientwise sparsity changes after
an orthogonal rotation of the basis inside the eigenspace.

The sphere gives the simplest illustration. On \(\mathbb S^2\), a rotation
mixes the \(2\ell+1\) spherical-harmonic coefficients of degree \(\ell\), but
it leaves their Euclidean norm unchanged. This suggests imposing sparsity
across complete eigenspaces while leaving the direction inside an active
eigenspace unrestricted. For a fixed Riemannian metric, and in the absence of
additional invariant structure within an eigenspace, these complete
Laplace eigenspaces  are the spectral blocks determined by the operator without introducing any auxiliary basis or decomposition.
The resulting model
captures both the gain from ignoring inactive blocks and the unavoidable cost
of estimating an unknown direction inside an active high-dimensional one. When the spectrum is simple, the blocks are one-dimensional and the model reduces to weighted coefficientwise spectral sparsity; the distinctive multiplicity effects arise on manifolds with nontrivial eigenspaces.

Probability distributions on compact manifolds arise when observations are directions, orientations, or periodic configurations, with spheres, rotation groups, and tori providing canonical examples. In such problems, the notion of a simple or sparse density should not depend on a coordinate system, a physical reference frame, or an arbitrary choice of eigenfunctions inside a repeated Laplace eigenspace. The block-variation model is intended for densities that are spectrally compressible at the level of complete eigenspaces: only a relatively small number of harmonic degrees or spectral levels carry substantial energy, while the direction within an active eigenspace may remain unrestricted. From a statistical viewpoint, this is an infinite-dimensional group-sparsity problem in which the groups are determined by the geometry rather than selected by the analyst. A dense spectral estimator pays for every eigenspace below a cutoff, whereas coefficientwise thresholding can produce basis-dependent selections. The central question is therefore whether one can adapt to the unknown active eigenspaces while preserving geometric invariance, and what unavoidable price is imposed by their multiplicities. The oracle inequalities and matching lower bounds below quantify precisely this trade-off. The framework is designed for global spectral compressibility and is complementary to needlet and wavelet procedures aimed at spatially localized or multiscale structure.

Let \((M,g)\) be a compact connected smooth Riemannian manifold without
boundary, let \(\nu\) be normalized Riemannian volume, and suppose that
\[
        X_1,\ldots,X_n \sim p_0\,d\nu
\]
independently, where \(p_0\) is an unknown density. Let
\(L=-\Delta_M\) be the nonnegative Laplace--Beltrami operator, with distinct
eigenvalues, eigenspaces, multiplicities, and orthogonal projectors
\[
        0=\mu_0<\mu_1<\mu_2<\cdots,
        \qquad
        E_j=\ker(L-\mu_j I),
        \qquad
        d_j=\dim(E_j),
        \qquad
        P_j:L^2(M)\to E_j.
\]
For \(s\geq0\), define the intrinsic spectral block-variation norm
\begin{equation}\label{eq:intro-block-variation-norm}
        \|f\|_{\Vspec_s(M)}
        :=
        \sum_{j\geq0}
        (1+\mu_j)^{s/2}\|P_jf\|_{L^2(M)}
\end{equation}
and the associated density class
\[
        \Pcal_s(R,b)
        :=
        \left\{
        p\geq0:
        \int_M p\,d\nu=1,\quad
        \|p\|_{\Vspec_s(M)}\leq R,\quad
        \|p\|_{L^\infty(M)}\leq b
        \right\}.
\]
Because \(M\) is connected and \(\nu(M)=1\), the zero eigenspace is
\(E_0=\operatorname{span}\{\one\}\) and
\[
        P_0f=\left(\int_M f\,d\nu\right)\one.
\]
Consequently, every density satisfies \(P_0p=\one\) and
\begin{equation}\label{eq:density-constant-block}
        \|p\|_{\Vspec_s(M)}
        =
        1+\sum_{j\geq1}(1+\mu_j)^{s/2}\|P_jp\|_{L^2(M)}.
\end{equation}
Thus \(\Pcal_s(R,b)\) can be nonempty only when \(R\geq1\) and \(b\geq1\),
and \(R-1\) is the variation budget available for nonconstant perturbations.
Accordingly, \(R-1\), rather than \(R\), is the effective radius governing
both the upper and lower statistical bounds.

\paragraph{Intrinsic sparsity and statistical adaptation}
If
\[
        a_j:=(1+\mu_j)^{s/2}\|P_jf\|_{L^2(M)},
\]
then
\[
        \|f\|_{H^s(M)}
        =
        \left(\sum_{j\geq0}a_j^2\right)^{1/2},
        \qquad
        \|f\|_{\Vspec_s(M)}
        =
        \sum_{j\geq0}a_j.
\]
Sobolev regularity therefore controls weighted block energy in \(\ell^2\),
whereas spectral block variation imposes an \(\ell^1\) condition across the
exact Laplace eigenspaces. It penalizes dispersion across eigenspaces without
selecting coordinates within an active block. Since the norm uses only the
spectral projectors \(P_j\) and the Hilbert norms
\(\|P_jf\|_{L^2(M)}\), it is independent of the basis chosen inside each
eigenspace and invariant under isometries. Section~\ref{sec:spectral-regularity}
develops the corresponding structural, atomic, and nonlinear-approximation
theory.

The statistical analogue is the oracle expression
\begin{equation}\label{eq:intro-oracle-expression}
        \sum_{j\geq1}
        \min\left\{
        \|P_jp_0\|_{L^2(M)}^2,\tau_j^2
        \right\}.
\end{equation}
Here \(\tau_j\) is the detection threshold, or noise scale, for the empirical
projection onto \(E_j\). We write \(\delta_n\) for an effective squared-noise
level per spectral degree of freedom, meaning that
\(\tau_j^2\lesssim\delta_n d_j\) on the retained blocks; its precise definition
and calibration are given in Section~\ref{sec:block-shrinkage}. A weak block may
be discarded at the cost of its signal energy, while a detectable block is
estimated at its noise level. Thus detectable blocks in a set \(S\) cost of
order \(\delta_n\sum_{j\in S}d_j\), rather than the dimension of every block
below a frequency cutoff. Invariant block sparsity can therefore remove the
cost of inactive eigenspaces, but it cannot remove the dimension cost within
an active one.

\subsection{Contributions}
\label{subsec:contributions}

\paragraph{A basis-independent oracle inequality}
We estimate each empirical eigenspace projection and apply Euclidean soft
thresholding to the whole block. 
Arbitrary coordinates may be used for computation, but the shrinkage factor depends only on the block norm, and the resulting estimator is independent of the chosen eigenbasis. We prove
a nonasymptotic \(L^2\)-oracle inequality with the blockwise minimum in
\eqref{eq:intro-oracle-expression}, a spectral truncation term, and a
deviation remainder; the same bound holds after projection onto the closed
convex set of densities. The estimator therefore adapts to the unknown set of
detectable eigenspaces rather than paying automatically for every block below
a cutoff. A Monte Carlo study on $\mathbb S^2$ in
Section~\ref{sec:experiments} illustrates this signal-dependent
behavior: as the sample size increases, the active harmonic degrees
are progressively retained, while inactive degrees are largely
suppressed.

\paragraph{Multiplicity-sensitive rates and matching lower bounds}
Assume that, for exponents \(\alpha,\gamma\geq0\),
\[
        d_j\lesssim(1+\mu_j)^\gamma,
        \qquad
        \sum_{1\leq j:\mu_j\leq\Lambda}d_j
        \lesssim(1+\Lambda)^{\alpha+\gamma}.
\]
The exponent \(\gamma\) controls individual eigenspace growth, while
\(\alpha+\gamma\) controls cumulative spectral dimension. With the effective
squared-noise level \(\delta_n\) introduced above, the resulting rate over
\(\Pcal_s(R,b)\) is
\[
        \rho_{\alpha,\gamma,s}(R-1,\delta_n)
        =
        \begin{cases}
        (R-1)^{\frac{2\gamma}{s+\gamma}}
        \delta_n^{\frac{s}{s+\gamma}},
        &0<s<\gamma,\\[0.8em]
        (R-1)^{\frac{2(\alpha+\gamma)}{s+\gamma+2\alpha}}
        \delta_n^{\frac{s+\alpha}{s+\gamma+2\alpha}},
        &s\geq\gamma.
        \end{cases}
\]
Under matching lower-growth and nondegeneracy assumptions, the minimax lower
bounds are stated directly in terms of the sample size. For fixed density
envelopes, they match the two branches above after the substitution
\(\delta_n\asymp n^{-1}\), and hence recover the same powers of \(R-1\) and
\(n\). On \(S^m\), \(m\geq2\), and on \(SO(3)\), index-dependent coding levels
exploit multiplicity growth and make the upper estimator attain
\(\delta_n\asymp b/n\), which is equivalent to the \(n^{-1}\) scale when \(b\)
is fixed.

\paragraph{Intrinsic regularity and approximation}
We prove that \(\Vspec_s(M)\) is a Banach space with dense finite spectral
sums, and establish basis independence, isometry invariance, monotone
smoothness embeddings, an atomic variation-norm characterization, and an
attained basis-envelope identity for coefficientwise spectral
\(\ell^1\)-norms. We also prove
\(
\Vspec_s(M)\hookrightarrow B^s_{2,1}(M)\hookrightarrow H^s(M)
\)
and nonlinear approximation bounds in both \(H^s\) and \(L^2\). The
nonlinear result counts retained eigenspaces, whereas the linear truncation
result counts total spectral degrees of freedom. Together they show that the
class records how weighted energy is distributed among exact eigenspaces, not
merely its total Sobolev size.

\paragraph{A positive likelihood counterpart}
We also consider
\[
        p_u(x)=\exp\{u(x)-A(u)\},
        \qquad
        A(u)=\log\int_M e^{u(y)}\,d\nu(y),
\]
with a group penalty on the complete eigenspace components of the centered
log-density. The estimator is positive and normalized by construction. On
bounded log-density sieves, we prove a Kullback--Leibler oracle inequality and
derive expected Kullback--Leibler, Hellinger, and \(L^2\) risk bounds. Writing \(q>0\) for the spectral block-variation order of the centered
log-density, and assuming that, for some \(\beta\geq0\),
\[
        \sup_{x\in M}K_j(x,x)\lesssim(1+\mu_j)^\beta,
\]
we obtain explicit multiplicity-sensitive rates whenever \(q\geq\beta\). On homogeneous manifolds one may take
\(\beta=\gamma\). This positive theory is complementary to the linear theory,
because regularity is imposed on the centered log-density rather than directly
on the density. Taken together, these results provide a density-estimation theory for the
isometry-invariant block geometry determined by a fixed Laplace operator.

\subsection{Relation to prior work and scope of novelty}
\label{subsec:intro-related-work}

\paragraph{Invariant spectral sparsity on the sphere}
The closest geometric precursor is the rotation-invariant spherical-harmonic
regularization of Le Gia, Sloan, Womersley, and
Wang~\cite{le_gia_sloan_womersley_wang_2020}. Their mixed
\(\ell^1/\ell^2\) penalty groups all coefficients of a fixed harmonic degree.
With weights \((1+\ell(\ell+1))^{s/2}\), it coincides with
\eqref{eq:intro-block-variation-norm} on \(\mathbb S^2\). This construction
addresses the basis dependence of coefficientwise \(\ell^1\) regularization
identified by Cammarota and Marinucci~\cite{cammarota_marinucci_2015}.
Related invariant whole-degree penalties and sparse representations on the
sphere appear in Li and Chen~\cite{li_chen_2024_group_sparse} and Greco and
Marinucci~\cite{greco_marinucci_2026_sparsity}. The present work replaces
spherical-harmonic degrees by complete Laplace eigenspaces on a general compact
manifold and develops the corresponding density-estimation theory, including
oracle inequalities, upper rates, and minimax lower bounds.

\paragraph{Block thresholding, needlets, and manifold density estimation}
Classical spectral density estimation on closed Riemannian manifolds using
Laplace--Beltrami Fourier expansions was studied by
Hendriks~\cite{hendriks1990nonparametric}. Our focus is different: we impose
a basis-independent block-sparsity geometry on complete eigenspaces and
develop signal-dependent oracle inequalities and multiplicity-sensitive
minimax theory. Block thresholding and its oracle theory are classical in
wavelet statistics; see, for example, Cai~\cite{Cai1999}. On geometric domains,
Durastanti~\cite{durastanti_2015_block} studied needlet block thresholding on
the sphere, Kerkyacharian, Nickl, and
Picard~\cite{KerkyacharianNicklPicard2012} developed adaptive needlet density
estimators and confidence bands on compact homogeneous manifolds, and
Cleanthous, Georgiadis, Kerkyacharian, Petrushev, and
Picard~\cite{cleanthous_et_al_2020_density} treated kernel and wavelet density
estimation in a broader spectral metric-space framework. Those methods use
localized multiresolution systems and are naturally linked to Besov
regularity. Our blocks are instead the exact eigenspaces of the fixed
Laplace--Beltrami operator. This sacrifices spatial localization but gives an
intrinsic decomposition on every compact Riemannian manifold and separates
block selection from within-block dimension. On manifolds carrying additional
symmetry, finer invariant substructures may be available; the present paper
deliberately works at the level of complete Laplace eigenspaces.

 \paragraph{Spectral Barron spaces}
Weighted coefficientwise spectral $\ell^{1}$-spaces have been used in PDE
and neural-network approximation theory; see Lu, Lu, and
Wang~\cite{lu2021apriori}, Liao and Ming~\cite{liao2025spectral}, and, for
compact groups, Mensah and Aremua~\cite{mensah2025spectral}. The norm~\eqref{eq:intro-block-variation-norm}
is related to these spaces both as a basis-independent grouped envelope and
as an atomic variation norm over unit directions contained in single Laplace
eigenspaces; see~\eqref{eq:basis-envelope-characterization}. Abstractly, it may also be viewed as a weighted
decomposition-space norm associated with the exact spectral resolution. Our
emphasis, however, is on its basis-free eigenspace geometry and the resulting
density-estimation theory. We therefore use \emph{spectral block variation}
as the primary terminology. An active block may point in an arbitrary
direction in a $d_j$-dimensional eigenspace, so the model does not by itself
impose ridge-function structure or imply dimension-independent statistical
rates.

\paragraph{Positive spectral models}
The positive estimator lies in the tradition of exponential-series and sieve
likelihood density estimation, including Barron and
Sheu~\cite{barron_sheu_1991}. Harmonic exponential families on spheres, tori,
and rotation groups were developed by Cohen and
Welling~\cite{cohen_welling_2015}. Our construction combines this likelihood
framework with a group penalty indexed by complete Laplace eigenspaces,
thereby preserving basis independence and retaining the
multiplicity-sensitive oracle geometry.

\subsection{Organization}

Section~\ref{sec:spectral-regularity} develops the block-variation spaces.
Section~\ref{sec:block-shrinkage} establishes the block-shrinkage oracle
inequality and upper bounds, and Section~\ref{sec:lower_bound} proves matching
minimax lower bounds. Sections~\ref{sec:positive-spectral-sieves} and
\ref{sec:experiments} present the positive likelihood extension and the study
on $\mathbb S^2$, respectively. Section~\ref{sec:discussion} concludes, and
proofs and technical details are collected in the appendix; see
Appendix~\ref{app:structural-proofs}--\ref{app:fano-lower-bound}.


\section{Intrinsic spectral block-variation regularity}\label{sec:spectral-regularity}

Throughout this section \((M,g)\) is compact, connected, smooth, and
boundaryless, and \(\nu\) denotes normalized Riemannian volume. We work in
real \(L^2(M)\) unless explicitly stated otherwise. The complex-valued case
is identical after inserting complex conjugates in inner products and
kernels.

\subsection{Spectral preliminaries}

We use the standard spectral calculus for the nonnegative Laplace--Beltrami
operator on a closed manifold; see, for example,
\cite{rosenberg1997laplacian,chavel1984eigenvalues}.  Its spectrum is discrete,
each eigenspace is finite-dimensional, and
\[
        L^2(M)=\bigoplus_{j\geq0}E_j,
        \qquad
        E_j=\ker(L-\mu_jI),
        \qquad
        0=\mu_0<\mu_1<\cdots.
\]
Because \(M\) is connected and \(\nu(M)=1\),
\begin{equation}\label{eq:constant-projector}
        E_0=\operatorname{span}\{\one\},
        \qquad
        P_0f=\left(\int_M f\,d\nu\right)\one.
\end{equation}
If \(P_j\) is the orthogonal projector onto \(E_j\), then
\[
        f=\sum_{j\geq0}P_jf\quad\text{in }L^2(M),
        \qquad
        \|f\|_{L^2(M)}^2=\sum_{j\geq0}\|P_jf\|_{L^2(M)}^2.
\]
For \(\Lambda\geq0\), set
\[
        E_{\leq\Lambda}:=\bigoplus_{\mu_j\leq\Lambda}E_j,
        \qquad
        \Pi_{\leq\Lambda}:=\sum_{\mu_j\leq\Lambda}P_j,
        \qquad
        D(\Lambda):=\sum_{\mu_j\leq\Lambda}d_j.
\]
Weyl's law yields a constant \(C_W=C_W(M,g)\) such that
\begin{equation}\label{eq:weyl}
        D(\Lambda)\leq C_W(1+\Lambda)^{m/2},
        \qquad \Lambda\geq0.
\end{equation}
For \(r\in\mathbb R\), the spectral Sobolev norm is
\[
        \|f\|_{H^r(M)}^2
        :=\sum_{j\geq0}(1+\mu_j)^r\|P_jf\|_{L^2(M)}^2.
\]
This gives the usual Sobolev space, with an equivalent derivative norm when
\(r\) is a nonnegative integer; these standard facts follow from elliptic
regularity and the functional calculus of \(I+L\)
\cite{rosenberg1997laplacian}.

\subsection{Definition and basic properties}

\begin{definition}[Intrinsic spectral block-variation space]\label{def:block-variation}
For \(s\geq0\), define
\[
        \|f\|_{\Vspec_s(M)}
        :=
        \sum_{j\geq0}(1+\mu_j)^{s/2}\|P_jf\|_{L^2(M)}.
\]
The intrinsic spectral block-variation space of order \(s\) is
\[
        \Vspec_s(M)
        :=
        \left\{f\in L^2(M):\|f\|_{\Vspec_s(M)}<\infty\right\}.
\]
For \(R>0\), write
\[
        \Vspec_s(R)
        :=
        \left\{f\in L^2(M):\|f\|_{\Vspec_s(M)}\leq R\right\}.
\]
\end{definition}

If \(\{\varphi_{j,\ell}\}_{\ell=1}^{d_j}\) is any orthonormal basis of
\(E_j\), then
\[
        \|P_jf\|_{L^2(M)}
        =
        \left(
        \sum_{\ell=1}^{d_j}
        |\langle f,\varphi_{j,\ell}\rangle_{L^2(M)}|^2
        \right)^{1/2}.
\]
Thus the definition is independent of the basis chosen inside each
eigenspace and  it imposes sparsity across the exact spectral subspaces determined by $L$, while leaving the direction within an active block unrestricted.

\begin{proposition}[Structural properties]\label{prop:structural}
For every \(s\geq0\), \(\Vspec_s(M)\) is a Banach space, finite spectral
sums are dense in \(\Vspec_s(M)\), and
\[
        \|f\|_{L^2(M)}\leq\|f\|_{\Vspec_s(M)}.
\]
If \(t\geq s\), then \(\Vspec_t(M)\hookrightarrow\Vspec_s(M)\)
continuously and
\[
        \|f\|_{\Vspec_s(M)}\leq\|f\|_{\Vspec_t(M)}.
\]
Moreover, if \(T:M\to M\) is an isometry, then
\[
        \|f\circ T\|_{\Vspec_s(M)}=\|f\|_{\Vspec_s(M)}.
\]
\end{proposition}

The proof is given in Appendix~\ref{app:structural-proofs}.

\begin{proposition}[Intrinsic block compressibility]
\label{prop:block-compressibility}
Let \(s\geq0\) and \(f\in\Vspec_s(M)\). For \(K\geq1\), define
\[
        \sigma_K^{\mathrm{block}}(f;H^s)
        :=
        \inf_{\substack{S\subset\mathbb N_0\\ |S|\leq K}}
        \left\|
        f-\sum_{j\in S}P_jf
        \right\|_{H^s(M)}.
\]
Then
\[
        \sigma_K^{\mathrm{block}}(f;H^s)
        \leq
        \frac{\|f\|_{\Vspec_s(M)}}{\sqrt{K+1}}.
\]
In particular,
\[
        \inf_{\substack{S\subset\mathbb N_0\\ |S|\leq K}}
        \left\|
        f-\sum_{j\in S}P_jf
        \right\|_{L^2(M)}
        \leq
        \frac{\|f\|_{\Vspec_s(M)}}{\sqrt{K+1}}.
\]
\end{proposition}

This is the classical Stechkin best-\(K\)-term estimate applied to the
weighted block magnitudes
\((1+\mu_j)^{s/2}\|P_jf\|_2\); see
\cite{devore1998nonlinear}.  The short specialization is recorded in
Appendix~\ref{proof_prop:block-compressibility}.  It is nonlinear because the selected set depends
on \(f\), and its deterministic complexity is the number of retained
eigenspaces rather than their aggregate dimension.  The block dimensions
reappear in the statistical estimation cost.

The next result gives the principal structural interpretation of the norm: it is both an atomic variation norm and the smallest coefficientwise spectral $\ell^1$-norm obtainable by rotating bases within eigenspaces.
\begin{proposition}[Atomic and basis-envelope characterizations]
\label{prop:atomic-basis-envelope}
Let
\[
        \Acal_s
        :=
        \bigcup_{j\geq0}
        \left\{
        (1+\mu_j)^{-s/2}u:
        u\in E_j,\ \|u\|_{L^2(M)}=1
        \right\}.
\]
Then the closed unit ball of \(\Vspec_s(M)\) is
\[
        \left\{f:\|f\|_{\Vspec_s(M)}\leq1\right\}
        =
        \overline{\operatorname{aconv}(\Acal_s)}^{\,L^2(M)},
\]
where \(\operatorname{aconv}\) denotes the absolutely convex hull.
Let \(\mathfrak E(L)\) denote the collection of eigenbases
\(  \Bcal=\{\phi_{j,\ell}:j\geq0,\ 1\leq\ell\leq d_j\} \)
obtained by choosing an orthonormal basis separately in each eigenspace, and
define
\[
        \|f\|_{\mathrm{SB},s;\Bcal}
        :=
        \sum_{j\geq0}(1+\mu_j)^{s/2}
        \sum_{\ell=1}^{d_j}
        |\langle f,\phi_{j,\ell}\rangle_{L^2(M)}|.
\]
Then, with extended values on \(L^2(M)\),
\begin{equation}\label{eq:basis-envelope-characterization}
        \|f\|_{\Vspec_s(M)}
        =
        \min_{\Bcal\in\mathfrak E(L)}
        \|f\|_{\mathrm{SB},s;\Bcal}.
\end{equation}
\end{proposition}

The proof is given in Appendix~\ref{app:structural-proofs}. The first identity
justifies the word ``variation'': the atoms are weighted unit directions
supported in single eigenspaces. It also makes the statistical limitation
transparent, because the atom family at level \(j\) is the unit sphere of a
\(d_j\)-dimensional Hilbert space. The minimum in
\eqref{eq:basis-envelope-characterization} is attained by aligning one basis
vector in each nonzero block with \(P_jf\). This minimizing eigenbasis depends
on \(f\); the identity does not assert that one common eigenbasis makes the
entire class coefficientwise sparse. When the spectrum is simple, the norm
reduces to a weighted coefficientwise spectral \(\ell^1\)-norm. With
multiplicities, it is the basis-independent grouped envelope of those norms.

\FloatBarrier
\subsection{Concrete forms of the intrinsic spectral blocks}\label{sec:examples}

For each \(\lambda\in\operatorname{Spec}(L)\), let \(P_\lambda\) denote the
orthogonal projector onto the full \(\lambda\)-eigenspace and set
\[
        b_\lambda(f):=\|P_\lambda f\|_{L^2(M)}.
\]
Then
\[
        \|f\|_{\Vspec_s(M)}
        =
        \sum_{\lambda\in\operatorname{Spec}(L)}
        (1+\lambda)^{s/2}b_\lambda(f).
\]
Table~\ref{tab:concrete-blocks} identifies these block energies in the
principal examples.  The spectral formulas for spheres and compact Lie groups
are standard; see \cite{chavel1984eigenvalues,helgason2000groups}.  In the torus
row we use \(\mathbb T^d=(\mathbb R/2\pi\mathbb Z)^d\) with normalized Haar
measure, so the eigenvalues are \(|k|^2\).

\begin{table}[!htbp]
\centering
\small
\begin{tabularx}{\textwidth}{@{}l l Y@{}}
\toprule
Space & Laplace block & Block energy \(b_\lambda(f)\) \\
\midrule

\(\mathbb T^d\)
&
\(r\in\mathcal R_d:=\{|k|^2:k\in\mathbb Z^d\}\)
&
\(\displaystyle
b_r(f)=
\left(\sum_{|k|^2=r}|\widehat f_k|^2\right)^{1/2}\).
For \(d=1\) and \(r=n^2\), this groups the modes \(n\) and \(-n\).
\\[1.2ex]

\(\mathbb S^m\)
&
\(\mu_\ell=\ell(\ell+m-1)\)
&
\(\displaystyle
b_{\mu_\ell}(f)=
\left(\sum_{q=1}^{d_\ell}|\widehat f_{\ell,q}|^2\right)^{1/2}\).
\\[1.2ex]

Compact Lie group \(G\)
&
\(\lambda\) a Casimir eigenvalue
&
\(\displaystyle
b_\lambda(f)=
\left(
\sum_{\pi:\kappa_\pi=\lambda}
 d_\pi\|\widehat f(\pi)\|_{\mathrm{HS}}^2
\right)^{1/2}\).
For \(SO(3)\),
\(b_{\ell(\ell+1)}(f)=
\sqrt{2\ell+1}\|\widehat f(\ell)\|_{\mathrm{HS}}\).
\\[1.2ex]

\(M_1\times M_2\)
&
\(\lambda=\mu_a^{(1)}+\mu_b^{(2)}\)
&
\(\displaystyle
b_\lambda(f)=
\left\|
\sum_{\mu_a^{(1)}+\mu_b^{(2)}=\lambda}
(P_a^{(1)}\otimes P_b^{(2)})f
\right\|_{L^2(M_1\times M_2)}\).
\\

\bottomrule
\end{tabularx}
\caption{Intrinsic spectral blocks in standard geometries.}
\label{tab:concrete-blocks}
\end{table}

Thus the intrinsic blocks are, respectively, lattice shells,
spherical-harmonic degrees, Casimir eigenspaces, and collision classes of
product eigenvalues. In each case, coordinates within a block are aggregated
in \(\ell^2\), making the norm independent of a change of orthonormal basis.
The spectral growth calculations are given in Appendix~\ref{sec:specializations}, and the resulting matching minimax rates are collected in Corollary~\ref{cor:minimax-standard-homogeneous}.
\FloatBarrier

\subsection{Comparison with Sobolev and Besov spaces}

The spectral block-variation condition is stronger than Sobolev regularity of
the same order. Indeed,
\[
        \|f\|_{H^s(M)}
        =
        \left(
        \sum_{j\geq0}
        \left[(1+\mu_j)^{s/2}\|P_jf\|_{L^2(M)}\right]^2
        \right)^{1/2}
        \leq
        \|f\|_{\Vspec_s(M)}.
\]
More precisely, the block \(\ell^1\) geometry implies a Besov embedding.

\begin{proposition}[Sobolev--Besov comparison]\label{prop:besov-embedding}
Let \(s\geq0\). With the standard compact-manifold Littlewood--Paley
characterization of Besov spaces
\cite{geller_mayeli_2009_besov},
\[
        \Vspec_s(M)
        \hookrightarrow
        B^s_{2,1}(M)
        \hookrightarrow
        B^s_{2,2}(M)=H^s(M)
\]
continuously.
\end{proposition}

The proof is given in Appendix~\ref{app:structural-proofs}. The interpretation is
simple: \(\Vspec_s(M)\) imposes an \(\ell^1\) penalty across individual
eigenspaces, \(B^s_{2,1}(M)\) imposes an \(\ell^1\) penalty across dyadic
spectral bands, and \(H^s(M)\) imposes an \(\ell^2\) energy condition. When
\(s>m/2\), Sobolev embedding also gives
\begin{equation}\label{eq:variation-to-continuous}
        \|f\|_{L^\infty(M)}
        \leq C_{M,g,s}\|f\|_{H^s(M)}
        \leq C_{M,g,s}\|f\|_{\Vspec_s(M)}.
\end{equation}
Thus, in this smooth regime, an \(L^\infty\) envelope follows from the
variation radius on a fixed manifold. We retain the explicit envelope \(b\)
in the statistical classes so that the theory also covers \(s\leq m/2\) and
so that concentration constants remain visible.

The block-variation norm refines the usual spectral Besov scale by
summing the energies of individual eigenspaces before aggregation into
dyadic bands. Consequently, the embedding in
Proposition~\ref{prop:besov-embedding} can be strict. A spectral-shell
criterion and explicit constructions in the round spheres and standard flat
tori are given in Appendix~\ref{app:structural-proofs}.

\subsection{Spectral truncation}

\begin{proposition}[Spectral truncation]\label{prop:truncation}
Let \(s>0\). For every \(f\in\Vspec_s(M)\) and every \(\Lambda\geq0\),
\[
        \|f-\Pi_{\leq\Lambda}f\|_{L^2(M)}
        \leq
        \|f\|_{\Vspec_s(M)}(1+\Lambda)^{-s/2}.
\]
Consequently, for every \(f\in\Vspec_s(R)\),
\[
        \|f-\Pi_{\leq\Lambda}f\|_{L^2(M)}^2
        \leq
        R^2(1+\Lambda)^{-s}.
\]
For \(0<\eps<R\), choosing
\(\Lambda_\eps=(R/\eps)^{2/s}-1\) gives error at most \(\eps\) using at most
\[
        D(\Lambda_\eps)
        \leq
        C_W(R/\eps)^{m/s}
\]
Laplace--Beltrami degrees of freedom.
\end{proposition}

Proposition~\ref{prop:block-compressibility} and
Proposition~\ref{prop:truncation} describe two different approximation
geometries. The former is nonlinear and measures complexity by the number of
retained eigenspaces, without charging their dimensions. The latter is linear
and measures complexity by the total number \(D(\Lambda)\) of spectral degrees
of freedom below a cutoff. Statistical estimation combines the two: sparsity
can reduce the number of active blocks, but estimating an active block costs
its dimension.

\section{Block-shrinkage density estimation: oracle inequality and rates}
\label{sec:block-shrinkage}

\subsection{Coordinate-free estimator, equivariance, and oracle inequality}

Let $K_j$ be the integral kernel of the orthogonal projector $P_j$. Thus, for
any orthonormal basis $\{\phi_{j,\ell}\}_{\ell=1}^{d_j}$ of $E_j$,
\[
        K_j(x,y)
        =
        \sum_{\ell=1}^{d_j}\phi_{j,\ell}(x)\phi_{j,\ell}(y),
        \qquad
        P_jf(x)=\int_M K_j(x,y)f(y)\,d\nu(y).
\]
The kernel is independent of the chosen basis. Given independent observations
$X_1,\ldots,X_n\sim p_0\,d\nu$, define the empirical element of $E_j$
\begin{equation}
\label{eq:empirical-intrinsic-block}
        \widehat g_j
        :=
        \frac1n\sum_{i=1}^n K_j(\cdot,X_i)
        \in E_j,
        \qquad
        g_j:=\E_{p_0}\widehat g_j=P_jp_0.
\end{equation}
For $\tau>0$, define radial soft thresholding on the Hilbert space $E_j$ by
\[
        \mathsf S_\tau(g)
        :=
        \left(1-\frac{\tau}{\norm{g}_{L^2(M)}}\right)_+g,
        \qquad g\in E_j,
\]
with $\mathsf S_\tau(0)=0$. Let $\cJ\subset\{1,2,\ldots\}$ be finite. For
thresholds $\{\tau_j:j\in\cJ\}$, set
\begin{equation}
\label{eq:intrinsic-block-estimator}
        \widetilde p_{\cJ}
        :=
        1+\sum_{j\in\cJ}\mathsf S_{\tau_j}(\widehat g_j).
\end{equation}
The zero-frequency component is known exactly because $P_0p_0=1$.

Formula \eqref{eq:intrinsic-block-estimator} is coordinate free. Indeed, in an
arbitrary orthonormal basis of $E_j$,
\[
        \widehat g_j
        =
        \sum_{\ell=1}^{d_j}\widehat\theta_{j,\ell}\phi_{j,\ell},
        \qquad
        \widehat\theta_{j,\ell}
        =
        \frac1n\sum_{i=1}^n\phi_{j,\ell}(X_i),
\]
and radial soft thresholding becomes the usual Euclidean group soft-thresholding
map
\[
        S_\tau(y)
        =
        \left(1-\frac{\tau}{\norm y_2}\right)_+y.
\]
Thus coordinates may be used for computation, but neither the estimator nor
its block norms depend on the basis chosen inside an eigenspace.
On standard homogeneous manifolds the kernels and block coordinates are explicit; on a general manifold, implementation requires numerical approximation of the relevant Laplace eigenspaces and spectral projectors.

The estimator is also equivariant under isometries. If $T$ is an isometry,
write $T\mathbf X=(TX_i)_{i=1}^n$. Invariance of the projection kernels gives
\[
        \widehat g_j^{\,T\mathbf X}
        =
        \widehat g_j^{\,\mathbf X}\circ T^{-1},
        \qquad
        \widetilde p_{\cJ}^{\,T\mathbf X}
        =
        \widetilde p_{\cJ}^{\,\mathbf X}\circ T^{-1}.
\]
Hence the procedure respects the same isometry action as the regularity class.

The finite spectral sum in \eqref{eq:intrinsic-block-estimator} has unit
integral but need not be nonnegative. Let
\[
        \cD
        =
        \left\{
        q\in L^2(M):q\geq0\ \text{a.e.},\ \int_Mq\,d\nu=1
        \right\}
\]
and define
\[
        \widetilde p_{\cJ}^+
        :=
        \Pi_{\cD}\widetilde p_{\cJ}.
\]
Since $\cD$ is closed and convex and $p_0\in\cD$,
\[
        \norm{\widetilde p_{\cJ}^+-p_0}_{L^2(M)}
        \leq
        \norm{\widetilde p_{\cJ}-p_0}_{L^2(M)}.
\]
The projection has the familiar threshold-and-shift form
\[
        \Pi_{\cD}f=(f-\lambda_f)_+,
        \qquad
        \int_M(f-\lambda_f)_+\,d\nu=1.
\]
It is therefore a well-defined theoretical post-processing step, although it
is generally no longer a finite spectral sum.

We first state the probabilistic input abstractly.

\begin{assumption}[Block deviation]
\label{ass:block-deviation}
For every $j\in\cJ$, let
\[
        \varepsilon_j=\widehat g_j-g_j\in E_j.
\]
There are deterministic thresholds $\tau_j>0$ and remainders $\eta_j\geq0$
such that
\[
        \E_{p_0}\!\left[
        \norm{\varepsilon_j}_{L^2(M)}^2
        \bbone\!\left\{\norm{\varepsilon_j}_{L^2(M)}>\tau_j/2\right\}
        \right]
        +
        \tau_j^2\PP_{p_0}\!\left(
        \norm{\varepsilon_j}_{L^2(M)}>\tau_j/2
        \right)
        \leq \eta_j.
\]
The bound is required uniformly over the density class under consideration.
\end{assumption}

\begin{theorem}[Intrinsic block-shrinkage oracle inequality]
\label{thm:block-oracle}
Let $p_0$ be a density and let $\widetilde p_{\cJ}$ be defined by
\eqref{eq:intrinsic-block-estimator}. If
Assumption~\ref{ass:block-deviation} holds for every $j\in\cJ$, then
\begin{align*}
        \E_{p_0}\norm{\widetilde p_{\cJ}-p_0}_{L^2(M)}^2
        &\leq
        9\sum_{j\in\cJ}
        \min\left\{
        \norm{P_jp_0}_{L^2(M)}^2,\tau_j^2
        \right\}\\
        &\quad+
        \sum_{j\notin\cJ,\,j\geq1}
        \norm{P_jp_0}_{L^2(M)}^2
        +2\sum_{j\in\cJ}\eta_j.
\end{align*}
The same inequality holds with $\widetilde p_{\cJ}$ replaced by
$\widetilde p_{\cJ}^+$.
\end{theorem}

The first term is signal dependent: a weak block is discarded at the cost of
its signal energy, while a detectable block is estimated at its stochastic
threshold. The second term is the deterministic price of omitting blocks
outside $\cJ$.

\subsection{Bernstein thresholds}

Fix an arbitrary orthonormal basis of $E_j$ and write
\[
        \Phi_j(x)
        =
        (\phi_{j,1}(x),\ldots,\phi_{j,d_j}(x))\in\R^{d_j}.
\]
Then
\[
        \norm{\Phi_j(x)}_2^2=K_j(x,x),
        \qquad
        \int_MK_j(x,x)\,d\nu(x)=d_j.
\]
Set
\[
        \kappa_j:=\sup_{x\in M}K_j(x,x).
\]

\begin{proposition}[Vector Bernstein block threshold]
\label{prop:bernstein-threshold}
Assume $\norm{p_0}_{L^\infty(M)}\leq b$. There is a universal constant $C>0$
such that, for every $j\geq1$ and every $x\geq1$,
\[
        \PP_{p_0}\left(
        \norm{\widehat g_j-g_j}_{L^2(M)}
        >
        C\left[
        \sqrt{\frac{b(d_j+x)}{n}}
        +(1+\sqrt b)\frac{\sqrt{\kappa_j}\,x}{n}
        \right]
        \right)
        \leq e^{-x}.
\]
Consequently, for any coding level $\xi_j\geq1$, Assumption~\ref{ass:block-deviation}
holds with
\begin{equation}
\label{eq:bernstein-threshold}
        \tau_j
        =
        2C\left[
        \sqrt{\frac{b(d_j+\xi_j)}{n}}
        +(1+\sqrt b)\frac{\sqrt{\kappa_j}\,\xi_j}{n}
        \right]
\end{equation}
and
\begin{equation}
\label{eq:bernstein-remainder}
        \eta_j
        \leq
        C'\tau_j^2e^{-c\xi_j}
\end{equation}
for universal constants $C',c>0$, after increasing $C$ if necessary.
\end{proposition}

\begin{definition}[Homogeneous Riemannian manifold]
A Riemannian manifold $(M,g)$ is homogeneous if a subgroup of
$\operatorname{Isom}(M,g)$ acts transitively on $M$. Equivalently, $M$ may be
identified with a homogeneous space $G/H$ for a transitive isometry group $G$
and a point stabilizer $H$.
\end{definition}

\begin{corollary}[Homogeneous projector diagonal]
\label{cor:homogeneous-threshold}
Suppose $M$ is homogeneous and $\nu$ is normalized Riemannian volume. Then
\[
        K_j(x,x)=d_j,
        \qquad
        \kappa_j=d_j.
\]
For the thresholds in \eqref{eq:bernstein-threshold},
\begin{equation}
\label{eq:homogeneous-threshold-square}
        \tau_j^2
        \leq
        Cb\left
        \{
        \frac{d_j+\xi_j}{n}
        +\frac{d_j\xi_j^2}{n^2}
        \right\},
\end{equation}
where we used $b\geq1$, which is necessary for a nonempty density class.
\end{corollary}

\medskip

The proof is given in Appendix~\ref{app:calibration-corollary-proofs}.

\begin{corollary}[Summable expected-risk calibration]
\label{cor:summable-coding}
Suppose $M$ is homogeneous. Let $\cJ_n$ be finite and choose coding levels
$\xi_j\geq1$ such that, for every $j\in\cJ_n$,
\begin{equation}
\label{eq:summable-code-local}
        \xi_j\leq C_\xi d_j,
        \qquad
        \xi_j^2\leq n.
\end{equation}
Then the Bernstein thresholds satisfy
\( \tau_j^2\leq C\frac{b}{n}d_j\),
        \(j\in\cJ_n,\)
and
\begin{equation}
\label{eq:summable-code-remainder}
        \sum_{j\in\cJ_n}\eta_j
        \leq
        C\frac{b}{n}
        \sum_{j\in\cJ_n}d_je^{-c\xi_j}.
\end{equation}
In particular, if \(
        \sup_n\sum_{j\in\cJ_n}d_je^{-c\xi_j}<\infty\),
then $\delta_n=b/n$ is an effective squared-noise level in the sense of
Definition~\ref{def:effective-noise}, and the total remainder is $O(b/n)$.
\end{corollary}

\medskip

The proof is given in Appendix~\ref{app:calibration-corollary-proofs}.
\paragraph{From blockwise concentration to global risk rates}
The preceding results provide block thresholds \(\tau_j\) together
with expected-risk remainders \(\eta_j\). We summarize the stochastic
cost of estimating a block through the effective squared-noise level
\(\delta_n\), formalized in Definition~\ref{def:effective-noise} below and characterized by
\[
\tau_j^2 \lesssim \delta_n d_j
\]
over the retained eigenspaces. The next section optimizes the resulting
signal-dependent oracle bound over the block-variation ball and balances
it with the error from spectral truncation.

\subsection{Multiplicity-sensitive upper bound}

For a density $p_0\in\Pcal_s(R,b)$, the constant block consumes exactly one
unit of the variation norm. Define the nonconstant variation budget
\begin{equation}
\label{eq:nonconstant-budget}
        A_R:=R-1.
\end{equation}
Then
\begin{equation}
\label{eq:nonconstant-budget-constraint}
        \sum_{j\geq1}(1+\mu_j)^{s/2}
        \norm{P_jp_0}_{L^2(M)}
        \leq A_R.
\end{equation}
If $R=1$, the class contains only the constant density $p_0\equiv1$, and the
estimator $\widehat p\equiv1$ has zero risk. We henceforth assume $R>1$.

\begin{assumption}[Polynomial block growth]
\label{ass:poly-growth}
There are exponents $\alpha,\gamma\geq0$ and constants $C_d,C_D$ such that,
for every $j\geq1$ and every $\Lambda\geq1$,
\[
        d_j\leq C_d(1+\mu_j)^\gamma,
        \qquad
        \sum_{1\leq j:\,\mu_j\leq\Lambda}d_j
        \leq C_D(1+\Lambda)^{\alpha+\gamma}.
\]
The exponent $\gamma$ controls individual eigenspace growth, while
$\alpha+\gamma$ controls cumulative spectral dimension. In standard regularly
varying examples, $\alpha$ is the distinct-block counting exponent.
\end{assumption}

The pair $(\alpha,\gamma)$ in an upper-growth assumption need not be unique.
All results below hold for every admissible pair. The sharpest displayed bound
is obtained by minimizing over admissible pairs; in the geometric examples we
use the exact regularly varying exponents.

\begin{definition}[Effective squared-noise level]
\label{def:effective-noise}
Let $\cJ\subset\{1,2,\ldots\}$ be a finite set of retained spectral blocks,
and let $\{\tau_j:j\in\cJ\}$ be the corresponding thresholds. A number
$\delta_n>0$ is called an \emph{effective squared-noise level} on $\cJ$ if
there exists a constant $C_v>0$, independent of $n$ and $j$, such that
\begin{equation}
\label{eq:effective-block-noise}
        \tau_j^2\leq C_v\delta_nd_j,
        \qquad j\in\cJ.
\end{equation}
Thus $\delta_n$ controls the squared threshold per spectral degree of freedom,
uniformly over the retained blocks.
\end{definition}

For the Bernstein thresholds in \eqref{eq:bernstein-threshold}, one may take
\begin{equation}
\label{eq:effective-block-noise-bernstein}
        \delta_n
        =
        \max_{j\in\cJ}
        \left\{
        \frac{b(d_j+\xi_j)}{nd_j}
        +(1+\sqrt b)^2
        \frac{\kappa_j\xi_j^2}{n^2d_j}
        \right\}.
\end{equation}
\paragraph{Geometry and the size of the effective noise level}
An effective squared-noise level is not unique: any larger value also
satisfies Definition~\ref{def:effective-noise}. Statements about its
order therefore refer to the smallest scale supplied by a given
threshold calibration.

For a retained set \(\mathcal J_n\), define
\[
a_n
:=
\max_{j\in\mathcal J_n}\frac{\xi_j}{d_j},
\qquad
q_n
:=
\max_{j\in\mathcal J_n}\frac{\kappa_j}{d_j},
\qquad
\Xi_n
:=
\max_{j\in\mathcal J_n}\xi_j.
\]
Since \(b\geq1\), equation~\eqref{eq:effective-block-noise-bernstein} gives
\[
\delta_n
\lesssim
\frac{b}{n}(1+a_n)
+
\frac{bq_n\Xi_n^2}{n^2}.
\]
The quantity \(a_n\) measures whether the coding levels can be
absorbed by the block dimensions, whereas \(q_n\) measures the
concentration of the projector diagonal relative to its average
\(d_j\). Thus there is no universal calibration on an
arbitrary compact manifold.
Under a common coding level \(\xi_j=L_n\), one obtains
\[
\delta_n
\lesssim
b\left\{
\frac{L_n}{n}
+
\frac{q_nL_n^2}{n^2}
\right\}.
\]
Hence \(L_n\asymp\log n\) gives
\(\delta_n\lesssim b(\log n)/n\) provided
\(q_n\lesssim n/\log n\); if \(q_n\) grows faster, the second term may
dominate and the effective noise level can be larger. On a homogeneous
manifold \(q_n=1\), and the summable calibration of
Corollary~\ref{cor:summable-coding} attains the sharper scale
\(b/n\) whenever the coding levels are absorbed by the block
dimensions.

For $A>0$ and $\delta>0$, define
\begin{equation}
\label{eq:multiplicity-sensitive-rate}
        \rho_{\alpha,\gamma,s}(A,\delta)
        :=
        \begin{cases}
        A^{\frac{2\gamma}{s+\gamma}}
        \delta^{\frac{s}{s+\gamma}},
        &0<s<\gamma,\\[0.9em]
        A^{\frac{2(\alpha+\gamma)}{s+\gamma+2\alpha}}
        \delta^{\frac{s+\alpha}{s+\gamma+2\alpha}},
        &s\geq\gamma.
        \end{cases}
\end{equation}
When $\gamma=0$, only the second branch is used. The two formulas agree at
$s=\gamma$. The two regimes reflect different least-favorable spectral geometries. When
$0<s<\gamma$, individual eigenspace dimensions grow faster than the
regularity weights suppress high frequencies, so the worst-case oracle cost
may be driven by a single high-multiplicity block. When $s\geq\gamma$, the
regularity weights control individual blocks strongly enough that cumulative
spectral dimension, and hence the residual exponent $\alpha$, also enters the
balance.

\begin{theorem}[Multiplicity-sensitive block-shrinkage upper bound]
\label{thm:multiplicity-sensitive-upper}
Suppose
Assumption~\ref{ass:poly-growth} holds
and choose a deterministic cutoff $\Lambda_{\max,n}\geq\mu_1$. Define
\[
        \cJ_n:=\{j\geq1:\mu_j\leq\Lambda_{\max,n}\}.
\]
Assume that the following conditions hold uniformly over
$p_0\in\Pcal_s(R,b)$:
\begin{enumerate}[label=\textup{(\roman*)},leftmargin=2.2em]
\item for every $j\in\cJ_n$, Assumption~\ref{ass:block-deviation} holds
with deterministic threshold $\tau_j$ and remainder $\eta_j$;
\item the thresholds satisfy
\[
        \tau_j^2\leq C_v\delta_n d_j,
        \qquad j\in\cJ_n,
\]
for a constant $C_v$ independent of $n$ and $j$;
\item the truncation error is at most the target oracle rate:
\[
        A_R^2(1+\Lambda_{\max,n})^{-s}
        \leq
        \rho_{\alpha,\gamma,s}(A_R,\delta_n).
\]
\end{enumerate}
Then
\begin{equation}
\label{eq:main-upper-bound}
        \sup_{p_0\in\Pcal_s(R,b)}
        \E_{p_0}
        \norm{\widetilde p_{\cJ_n}^+-p_0}_{L^2(M)}^2
        \leq
        C\rho_{\alpha,\gamma,s}(A_R,\delta_n)
        +C\sum_{j\in\cJ_n}\eta_j,
\end{equation}
where $C$ depends only on
$s,\alpha,\gamma,C_d,C_D,C_v$.

Condition~\textup{(iii)} is ensured by the explicit choice
\begin{equation}
\label{eq:explicit-cutoff}
        \Lambda_{\max,n}
        \geq
        \max\left\{
        \mu_1,
        \left(
        \frac{A_R^2}
        {\rho_{\alpha,\gamma,s}(A_R,\delta_n)}
        \right)^{1/s}-1
        \right\}.
\end{equation}
\end{theorem}

The sufficient cutoff in \eqref{eq:explicit-cutoff} is explicit once the
blockwise noise scale $\delta_n$ has been specified. The theorem adapts to the
unknown set of detectable blocks, but the displayed construction still uses
$s$, $R$, and the density envelope entering the thresholds; it is therefore
not, by itself, adaptive to unknown smoothness, radius, or envelope.

\begin{corollary}[Calibration on homogeneous manifolds]
\label{cor:homogeneous-no-log-general}
Suppose
Assumption~\ref{ass:poly-growth} holds, and assume that $M$ is homogeneous.
Set $A_R:=R-1$ and $\delta_n:=b/n$. Choose $\Lambda_{\max,n}$ according to
\eqref{eq:explicit-cutoff}, define
$\cJ_n:=\{j\geq1:\mu_j\leq\Lambda_{\max,n}\}$, and use the Bernstein
thresholds \eqref{eq:bernstein-threshold}. If the coding levels satisfy
\eqref{eq:summable-code-local} and
\[
        \sup_n\sum_{j\in\cJ_n}d_je^{-c\xi_j}<\infty,
\]
then
\[
        \sup_{p_0\in\Pcal_s(R,b)}
        \E_{p_0}
        \norm{\widetilde p_{\cJ_n}^+-p_0}_{L^2(M)}^2
        \leq
        C\rho_{\alpha,\gamma,s}(A_R,b/n)+C\frac bn.
\]
Whenever the exponent of $\delta$ in \eqref{eq:multiplicity-sensitive-rate}
is strictly less than one, the final $b/n$ term is negligible for fixed
$R>1$ and $b$.
\end{corollary}

\begin{corollary}[Logarithm-free expected-risk calibration on spheres and $SO(3)$]
\label{cor:explicit-no-log-examples}
Let either
$M=S^m$ with $m\geq2$, in which case
\(
        (\alpha,\gamma)=\left(\frac12,\frac{m-1}{2}\right),
\)
or let $M=SO(3)$ with a bi-invariant metric, in which case
\(
        (\alpha,\gamma)=\left(\frac12,1\right).
\)
Index the nonconstant eigenspaces by $\ell\geq1$. Let $c$ be the universal
constant in \eqref{eq:bernstein-remainder}, choose $A_0>0$ so that
\[
        cA_0>m \quad\text{on }S^m,
        \qquad
        cA_0>3 \quad\text{on }SO(3),
\]
and set
\(
        \xi_\ell:=A_0\log(e+\ell).
\)
For each $n$, use the Bernstein thresholds
\eqref{eq:bernstein-threshold}, set $\delta_n:=b/n$, choose
$\Lambda_{\max,n}$ according to \eqref{eq:explicit-cutoff}, and define
$\cJ_n:=\{j\geq1:\mu_j\leq\Lambda_{\max,n}\}$. Then there exist constants
$C>0$ and $n_0\geq1$ such that, for every $n\geq n_0$,
\begin{equation}
\label{eq:no-log-examples-upper}
        \sup_{p_0\in\Pcal_s(R,b)}
        \E_{p_0}
        \norm{\widetilde p_{\cJ_n}^+-p_0}_{L^2(M)}^2
        \leq
        C\left\{
        \rho_{\alpha,\gamma,s}(A_R,b/n)+\frac bn
        \right\}.
\end{equation}
Here $C$ may depend on the fixed manifold, $s$, and the chosen coding
constant $A_0$, but not on $n$, $R$, or $b$; the threshold $n_0$ may
additionally depend on $R$ and $b$.
Consequently, for fixed $M$, $s$, $R$, and $b$,
\[
        \sup_{p_0\in\Pcal_s(R,b)}
        \E_{p_0}
        \norm{\widetilde p_{\cJ_n}^+-p_0}_{L^2(M)}^2
        \lesssim
        \rho_{\alpha,\gamma,s}(A_R,b/n).
\]
\end{corollary}

\medskip

The proof is given in Appendix~\ref{app:calibration-corollary-proofs}.

\begin{remark}[Dense risk versus signal-dependent oracle risk]
\label{remark:comparison_dense_spectral}
A dense cutoff on the same class pays
\[
        A_R^2(1+\Lambda)^{-s}+\frac{bD(\Lambda)}n.
\]
When $D(\Lambda)\asymp\Lambda^{\alpha+\gamma}$, optimization gives
\[
        A_R^{\frac{2(\alpha+\gamma)}{s+\alpha+\gamma}}
        \left(\frac bn\right)^{\frac{s}{s+\alpha+\gamma}}.
\]
At the ideal variance scale $bd_j/n$, the block oracle never loses this
benchmark, because for every $\Lambda$,
\begin{align*}
        \sum_{j\geq1}
        \min\left\{
        \norm{P_jp_0}_{L^2(M)}^2,\frac{bd_j}{n}
        \right\}
        &\leq
        \frac{bD(\Lambda)}n
        +\sum_{\mu_j>\Lambda}\norm{P_jp_0}_{L^2(M)}^2\\
        &\leq
        \frac{bD(\Lambda)}n
        +A_R^2(1+\Lambda)^{-s}.
\end{align*}
If the detectable blocks form a set $S$, their variance contribution is
approximately $\delta_n\sum_{j\in S}d_j$, rather than
$\delta_nD(\Lambda)$. Thus the signal-specific saving can be much larger than
the difference between the two worst-case powers of $n$.
\end{remark}

\section{Minimax lower bounds and optimality}
\label{sec:lower_bound}

The oracle upper bound of Theorem~\ref{thm:multiplicity-sensitive-upper}
suggests two distinct least-favorable geometries. In the dense-block regime
\(s\geq\gamma\), the difficulty is created by a spectral window containing many
active eigenspaces, and both the number of blocks and their dimensions enter
the rate. In the rough high-multiplicity regime \(0<s<\gamma\), a single large
eigenspace already produces the dominant obstruction. We show that the two
branches of \(\rho_{\alpha,\gamma,s}\) are minimax sharp under matching spectral
growth assumptions.

The lower bounds are stated directly in terms of $n$, whereas the upper
bounds inherit the effective noise level established in
Section~\ref{sec:block-shrinkage}.  On homogeneous manifolds a common code
produces the generic scale $b\log(n)/n$, while summable index-dependent codes
attain $b/n$ on $S^m$, $m\geq2$, and on $SO(3)$; hence, for fixed density
envelopes, the upper and lower rates match at the ideal scale, with any
remaining calibration gap confined to the upper bound.

For lower bounds it is convenient to work with density classes bounded both
above and below. For \(0<a\leq b<\infty\), define
\begin{equation}
\label{eq:two-sided-density-class}
        \Pcal_s(R,a,b)
        :=
        \left\{
        p:
        a\leq p\leq b,\quad
        \int_M p\,\dd\nu=1,\quad
        \norm{p}_{\Vspec_s(M)}\leq R
        \right\}.
\end{equation}
 The
lower bound \(p\geq a\) keeps the testing alternatives uniformly away from
zero and makes the usual comparisons among \(L^2\), Hellinger, and
Kullback--Leibler losses nondegenerate.

For a class \(\mathcal F\) of densities, write
\begin{equation}
\label{eq:minimax-risk}
        \mathfrak R_n(\mathcal F)
        :=
        \inf_{\widehat p}
        \sup_{p_0\in\mathcal F}
        \E_{p_0}
        \norm{\widehat p-p_0}_{L^2(M)}^2,
\end{equation}
where the infimum is over all measurable estimators based on
\(X_1,\ldots,X_n\).
Since $\Pcal_s(R,a,b)\subseteq\Pcal_s(R,b)$,
Theorem~\ref{thm:multiplicity-sensitive-upper} immediately gives the
corresponding minimax upper bound.  We next formulate the lower-growth assumptions needed for the dense-block
construction.

\begin{assumption}[Lower polynomial block growth]
\label{ass:poly-lower-growth}
There exist constants \(c_d,c_D>0\) and a threshold \(\Lambda_0\geq1\) such
that
\begin{equation}
\label{eq:poly-lower-growth}
        d_j\geq c_d(1+\mu_j)^\gamma
        \quad\text{whenever }\mu_j\geq\Lambda_0,
\end{equation}
and
\begin{equation}
\label{eq:poly-lower-cumulative-growth}
        \sum_{1\leq j:\,\mu_j\leq\Lambda}d_j
        \geq c_D(1+\Lambda)^{\alpha+\gamma}
        \quad\text{for every }\Lambda\geq\Lambda_0.
\end{equation}
\end{assumption}

The two parts of Assumption~\ref{ass:poly-lower-growth} play different roles.
The cumulative lower bound guarantees a spectral window with many degrees of
freedom, while the individual lower bound, combined with the cumulative upper
bound in Assumption~\ref{ass:poly-growth}, controls the number of distinct
blocks in that window by order \((1+\Lambda)^\alpha\).

\begin{theorem}[Lower bound in the dense-block regime]
\label{thm:minimax-lower-dense}
Assume Assumptions~\ref{ass:poly-growth} and
\ref{ass:poly-lower-growth}, and let \(s>0\). There exists a constant
\(c>0\), independent of \(R\) and \(n\), such that for every \(R>1\) there
is an integer \(n_0=n_0(R)\geq1\) for which, for every \(n\geq n_0\),
\begin{equation}
\label{eq:dense-minimax-lower}
        \mathfrak R_n\bigl(\Pcal_s(R,1/2,3/2)\bigr)
        \geq
        c\,
        A_R^{\frac{2(\alpha+\gamma)}{s+\gamma+2\alpha}}
        n^{-\frac{s+\alpha}{s+\gamma+2\alpha}}.
\end{equation}
\end{theorem}

The lower bound in Theorem~\ref{thm:minimax-lower-dense} is valid for every
\(s>0\). In the dense-block regime \(s\geq\gamma\), it matches the second
branch of Theorem~\ref{thm:multiplicity-sensitive-upper} at the ideal squared-noise scale
\(\delta_n\asymp n^{-1}\), including the dependence on the
budget \(A_R\). The proof uses a packing in a fixed-width spectral window. A
randomization theorem of Burq and Lebeau~\cite{burq2013injections} supplies
directions in that window whose \(L^\infty\)-norms are
\(O(\sqrt{\log\Lambda})\), while the block-variation constraint charges the
square root of the number of distinct eigenspaces.
The dense construction is not sharp when \(0<s<\gamma\). In that regime the
hardest alternatives can be confined to a single high-dimensional eigenspace,
so we impose the following nondegeneracy condition.

\begin{assumption}[Large eigenspaces with controlled diagonal]
\label{ass:single-block-lower}
There exist constants \(c_d,C_K>0\) and a threshold \(\Lambda_0\geq1\) such
that, for every \(j\) with \(\mu_j\geq\Lambda_0\),
\begin{equation}
\label{eq:single-block-lower}
        d_j\geq c_d(1+\mu_j)^\gamma,
        \qquad
        \sup_{x\in M}K_j(x,x)\leq C_Kd_j.
\end{equation}
\end{assumption}

The projector-diagonal condition holds on every compact homogeneous manifold
with normalized volume, because \(K_j(x,x)=d_j\).

\begin{theorem}[Lower bound in the rough high-multiplicity regime]
\label{thm:minimax-lower-rough}
Assume Assumption~\ref{ass:single-block-lower}, and let \(0<s<\gamma\).
There exists a constant \(c>0\), independent of \(R\) and \(n\), such that
for every \(R>1\) there is an integer \(n_0=n_0(R)\geq1\) for which, for
every \(n\geq n_0\),
\begin{equation}
\label{eq:rough-minimax-lower}
        \mathfrak R_n\bigl(\Pcal_s(R,1/2,3/2)\bigr)
        \geq
        c\,
        A_R^{\frac{2\gamma}{s+\gamma}}
        n^{-\frac{s}{s+\gamma}}.
\end{equation}
\end{theorem}

At the ideal squared-noise scale,
Theorem~\ref{thm:minimax-lower-rough} matches the first branch of
Theorem~\ref{thm:multiplicity-sensitive-upper}. The proof constructs a packing inside one
eigenspace of dimension of order \((1+\mu_j)^\gamma\). The diagonal bound in
\eqref{eq:single-block-lower} yields a logarithmic sup-norm bound for a random
unit direction; because the testing amplitude decays polynomially, this
logarithm affects only the threshold \(n_0\), not the minimax rate.

Combining the two lower bounds with the calibrations of Section~\ref{sec:block-shrinkage}
gives the following sharp comparisons.

\begin{corollary}[Minimax rates on standard homogeneous manifolds]
\label{cor:minimax-standard-homogeneous}
Fix \(R>1\). Let
\[
        \mathfrak R_n^M(s,R)
        :=
        \inf_{\widehat p}
        \sup_{p_0\in\Pcal_s(R,1/2,3/2)}
        \E_{p_0}\norm{\widehat p-p_0}_{L^2(M)}^2.
\]
Suppose \(M\) is compact, connected, and homogeneous, and assume the matching
growth laws
\[
        d_j\asymp(1+\mu_j)^\gamma,
        \qquad
        \sum_{1\leq j:\,\mu_j\leq\Lambda}d_j
        \asymp(1+\Lambda)^{\alpha+\gamma}.
\]
For fixed \(M,s,R\), as \(n\to\infty\), the common-code expected-risk
calibration gives
\begin{equation}
\label{eq:minimax-common-code}
        c\rho_{\alpha,\gamma,s}(A_R,n^{-1})
        \leq
        \mathfrak R_n^M(s,R)
        \leq
        C\rho_{\alpha,\gamma,s}(A_R,L_n/n),
        \qquad L_n\asymp\log n.
\end{equation}
If, in addition, the summable-code conditions of
Corollary~\ref{cor:homogeneous-no-log-general} hold, then
\begin{equation}
\label{eq:minimax-summable-calibration}
        \mathfrak R_n^M(s,R)
        \leq
        C\left\{
        \rho_{\alpha,\gamma,s}(A_R,n^{-1})+n^{-1}
        \right\}.
\end{equation}
For the following examples, the summable expected-risk calibration is
available and the upper and lower rates match without logarithmic inflation.

\smallskip
\noindent\textnormal{(i)} \textbf{Spheres of dimension at least two.}
For \(M=S^m\), \(m\geq2\),
\(\alpha=\frac12\), \( \gamma=\frac{m-1}{2},\)
and
\[
        \mathfrak R_n^{S^m}(s,R)
        \asymp
        \begin{cases}
        (R-1)^{\frac{2(m-1)}{2s+m-1}}
        n^{-\frac{2s}{2s+m-1}},
        &0<s<\frac{m-1}{2},\\[0.9em]
        (R-1)^{\frac{2m}{2s+m+1}}
        n^{-\frac{2s+1}{2s+m+1}},
        &s\geq\frac{m-1}{2}.
        \end{cases}
\]

\[
\text{For }S^2,\text{ this becomes}\qquad
        \mathfrak R_n^{S^2}(s,R)
        \asymp
        \begin{cases}
        (R-1)^{\frac{2}{2s+1}}n^{-\frac{2s}{2s+1}},
        &0<s<\frac12,\\[0.8em]
        (R-1)^{\frac{4}{2s+3}}n^{-\frac{2s+1}{2s+3}},
        &s\geq\frac12.
        \end{cases}
\]

\smallskip
\noindent\textnormal{(ii)} \textbf{The rotation group.}
For \(M=SO(3)\) with a bi-invariant metric,
\(
        \alpha=\frac12\),
        \(\gamma=1,\)
and
\[
        \mathfrak R_n^{SO(3)}(s,R)
        \asymp
        \begin{cases}
        (R-1)^{\frac{2}{s+1}}n^{-\frac{s}{s+1}},
        &0<s<1,\\[0.8em]
        (R-1)^{\frac{3}{s+2}}n^{-\frac{s+1/2}{s+2}},
        &s\geq1.
        \end{cases}
\]
For \(S^1\), the multiplicities are bounded and the present summable-code
argument does not remove the common logarithmic factor.
\end{corollary}
The proof is given in Appendix~\ref{app:minimax-short-proofs}.


\section{A positive likelihood extension}
\label{sec:positive-spectral-sieves}

The projected block-shrinkage estimator is a genuine density, but it is not
likelihood based and does not directly control Kullback--Leibler or Hellinger
loss.  We therefore record a complementary exponential-family construction.
For a measurable function \(u\), set
\[
        A(u):=\log\int_M e^{u(y)}\,d\nu(y),
        \qquad
        p_u(x):=e^{u(x)-A(u)},
\]
and impose \(\int_M u\,d\nu=0\) for identifiability.  The regularity assumption
in this section is placed on the centered log-density, rather than directly on
the density.  Specifically, for \(q>0\),
\[
        \mathcal U_q(R,B)
        :=
        \left\{
        u:\ \int_Mu\,d\nu=0,\quad
        \|u\|_{\Vspec_q(M)}\le R,\quad
        \|u\|_{L^\infty(M)}\le B
        \right\}.
\]
Because \(u\mapsto p_u\) is nonlinear, this is a different parameter class
from \(\Pcal_s(R,b)\).

Let \(J\subset\{j\ge1\}\) be finite and
\(F_J:=\bigoplus_{j\in J}E_j\).  Write
\[
        \mathbb P_n f:=\frac1n\sum_{i=1}^nf(X_i),
        \qquad
        \mathbb P_0 f:=\int_Mfp_0\,d\nu,
        \qquad
        F_J(B):=\{u\in F_J:\|u\|_\infty\le B\}.
\]
For penalty levels \(\lambda_j>0\), define
\begin{equation}\label{eq:positive-sieve-estimator}
        \widehat u_J
        :=
        \argmin_{u\in F_J(B)}
        \left\{
        -\mathbb P_n u+A(u)
        +4\sum_{j\in J}\lambda_j\|P_ju\|_{L^2(M)}
        \right\},
        \qquad
        \widehat p_J:=p_{\widehat u_J}.
\end{equation}
The estimator is positive and normalized by construction.  Its penalty is
intrinsic because each block norm is unchanged by an orthogonal rotation of an
eigenbasis inside \(E_j\).  The feasible set \(F_J(B)\) is compact and convex.
Moreover, for every \(v\in F_J\) and every nonzero \(h\in F_J\),
\[
        D^2A(v)[h,h]=\operatorname{Var}_{p_v}(h)>0,
\]
because \(F_J\) contains no nonzero constant functions.  Hence the objective
in \eqref{eq:positive-sieve-estimator} is strictly convex on \(F_J(B)\), so the
minimizer exists and is unique.  

For an arbitrary orthonormal basis
\(\{\phi_{j,\ell}\}_{\ell=1}^{d_j}\) of \(E_j\), let
\(\Phi_j=(\phi_{j,1},\ldots,\phi_{j,d_j})\),
\(Z_j=(\mathbb P_n-\mathbb P_0)\Phi_j\), and
\begin{equation}\label{eq:positive-event}
        \mathcal E_J
        :=
        \left\{\|Z_j\|_2\le\lambda_j\ \text{for every }j\in J\right\}.
\end{equation}
An orthogonal change of basis in \(E_j\) rotates \(Z_j\) without changing its
Euclidean norm, so \(\mathcal E_J\) is basis independent.
The positive estimator is also equivariant under isometries.  For an isometry
\(T\), write \(TX=(TX_i)_{i=1}^n\).  Uniqueness of the minimizer and invariance
of \(A\), \(F_J(B)\), and the block penalty give
\[
        \widehat u_J^{\,TX}=\widehat u_J^{\,X}\circ T^{-1},
        \qquad
        \widehat p_J^{\,TX}=\widehat p_J^{\,X}\circ T^{-1}.
\]

\begin{theorem}[Block likelihood oracle inequality]
\label{thm:positive-sieve-oracle}
Assume that \(p_0=p_{u_0}\) for a mean-zero \(u_0\) satisfying
\(\|u_0\|_\infty\le B\).  On the event \(\mathcal E_J\), the following
inequality holds simultaneously for every \(u\in F_J(B)\):
\begin{equation}\label{eq:positive-kl-oracle}
        \KL(p_0,\widehat p_J)
        \le
        C_B\left[
        \|u-u_0\|_{L^2(M)}^2
        +\sum_{j\in S(u)}\lambda_j^2
        \right],
        \qquad
        S(u):=\{j\in J:P_ju\ne0\},
\end{equation}
where \(C_B\) depends only on \(B\).  Consequently, still on
\(\mathcal E_J\),
\begin{equation}\label{eq:positive-hellinger-oracle}
        h^2(\widehat p_J,p_0)
        \le
        C_B\inf_{u\in F_J(B)}
        \left[
        \|u-u_0\|_{L^2(M)}^2
        +\sum_{j\in S(u)}\lambda_j^2
        \right],
\end{equation}
where \(h^2(p,q)=\frac12\int_M(\sqrt p-\sqrt q)^2\,d\nu\).
\end{theorem}

The proof, including the required KL--\(L^2\) equivalence on bounded
log-density paths, is given in Appendix~\ref{app:positive-likelihood}.

\begin{corollary}[Multiplicity-sensitive rate for the positive sieve]
\label{cor:positive-sieve-rate}
Assume the polynomial block-growth condition of
Assumption~\ref{ass:poly-growth} and suppose that, for some \(\beta\ge0\),
\[
        \kappa_j:=\sup_{x\in M}K_j(x,x)
        \le C_\kappa(1+\mu_j)^\beta,
        \qquad j\ge1.
\]
Let \(q\ge\beta\), \(R>0\), \(B>0\), and
\(B_\star\ge\max\{B,C_\kappa^{1/2}R\}\).  Put
\(J_n=\{j\ge1:\mu_j\le\Lambda_{\max,n}\}\), and use
\eqref{eq:positive-sieve-estimator} with \(J=J_n\) and \(B=B_\star\).
Suppose that
\begin{equation}\label{eq:lambda-delta-condition}
        \lambda_j^2\le C_\lambda\delta_nd_j,
        \qquad j\in J_n,
\end{equation}
that
\begin{equation}\label{eq:event-complement-rate}
        \sup_{u_0\in\mathcal U_q(R,B)}
        \mathbb P_{p_{u_0}}(\mathcal E_{J_n}^c)
        =o\!\left(\rho_n^{\log}(R,\delta_n)\right),
\end{equation}
where
\begin{equation}\label{eq:positive-log-rate}
        \rho_n^{\log}(R,\delta_n)
        :=\rho_{\alpha,\gamma,q}(R,\delta_n),
\end{equation}
and that
\begin{equation}\label{eq:positive-tail-choice}
        R^2(1+\Lambda_{\max,n})^{-q}
        \le\rho_n^{\log}(R,\delta_n).
\end{equation}
Then there is a constant
\(C=C(B_\star,q,\alpha,\gamma,C_d,C_D,C_\lambda,C_\kappa)>0\),
independent of \(n\) and of \(u_0\), such that, for all sufficiently large
\(n\),
\[
\begin{aligned}
        &\sup_{u_0\in\mathcal U_q(R,B)}
        \mathbb E_{p_{u_0}}\KL(p_{u_0},\widehat p_{J_n})\vee
        \sup_{u_0\in\mathcal U_q(R,B)}
        \mathbb E_{p_{u_0}}h^2(\widehat p_{J_n},p_{u_0})\\
        &\qquad\vee
        \sup_{u_0\in\mathcal U_q(R,B)}
        \mathbb E_{p_{u_0}}
        \|\widehat p_{J_n}-p_{u_0}\|_{L^2(M)}^2
        \le C\rho_n^{\log}(R,\delta_n).
\end{aligned}
\]
\end{corollary}

Note that for the feasible choice
\(B_\star\ge C_\kappa^{1/2}R\), the constant may depend on \(R\)
through \(B_\star\).

\begin{corollary}[Positive likelihood rate on homogeneous manifolds]
\label{cor:positive-homogeneous-rate}
Assume Assumption~\ref{ass:poly-growth}, suppose that \(M\) is homogeneous,
and fix \(q>0\) with \(q\geq\gamma\), together with \(R,B>0\). Choose
\[
        B_\star\geq\max\{B,C_d^{1/2}R\},
        \qquad
        \vartheta_q:=\frac{q+\alpha}{q+\gamma+2\alpha}.
\]
Let
\[
        c_0>
        \vartheta_q\left(1+\frac{\alpha+\gamma}{q}\right),
        \qquad
        L_n:=c_0\log n,
\]
and define
\begin{equation}
\label{eq:positive-homogeneous-rate}
        \rho_n^{\mathrm{hom},+}(R,B)
        :=
        R^{\frac{2(\alpha+\gamma)}{q+\gamma+2\alpha}}
        \left(\frac{e^BL_n}{n}\right)^{\vartheta_q}.
\end{equation}
Choose the cutoff
\begin{equation}
\label{eq:positive-homogeneous-cutoff}
        \Lambda_{\max,n}
        :=
        \max\left\{
        \mu_1,
        \left(\frac{R^2}{\rho_n^{\mathrm{hom},+}(R,B)}\right)^{1/q}-1
        \right\},
        \qquad
        J_n:=\{j\geq1:\mu_j\leq\Lambda_{\max,n}\},
\end{equation}
and, for \(j\in J_n\), take
\begin{equation}
\label{eq:positive-homogeneous-penalty}
        \lambda_j
        :=
        K\left[
        \sqrt{\frac{e^B(d_j+L_n)}{n}}
        +(1+e^{B/2})\frac{\sqrt{d_j}\,L_n}{n}
        \right],
\end{equation}
where \(K\) is a sufficiently large universal constant. Use
\eqref{eq:positive-sieve-estimator} with \(J=J_n\) and sieve radius
\(B_\star\). Then, for all sufficiently large \(n\),
\[
\begin{aligned}
        &\sup_{u_0\in\mathcal U_q(R,B)}
        \mathbb E_{p_{u_0}}\KL(p_{u_0},\widehat p_{J_n})
        \ \vee\
        \sup_{u_0\in\mathcal U_q(R,B)}
        \mathbb E_{p_{u_0}}h^2(\widehat p_{J_n},p_{u_0})\\
        &\qquad\vee\
        \sup_{u_0\in\mathcal U_q(R,B)}
        \mathbb E_{p_{u_0}}
        \|\widehat p_{J_n}-p_{u_0}\|_{L^2(M)}^2
        \leq C\rho_n^{\mathrm{hom},+}(R,B).
\end{aligned}
\]
Here \(C\) may depend on the fixed sieve radius \(B_\star\), the structural
constants, and \(c_0\), but not on \(n\) or \(u_0\). In particular, for fixed
\(R\) and \(B\), the rate is
\[
        R^{\frac{2(\alpha+\gamma)}{q+\gamma+2\alpha}}
        \left(\frac{e^B\log n}{n}\right)^{
        \frac{q+\alpha}{q+\gamma+2\alpha}}.
\]
\end{corollary}

The feasibility argument, penalty calibration, and the proof of
Corollary~\ref{cor:positive-homogeneous-rate}, together with its sphere and
\(SO(3)\) specializations, are collected in Appendix~\ref{app:positive-likelihood}.


\section{Experiments}
\label{sec:experiments}

In this section, we numerically illustrate the performance of the spectral
block-shrinkage estimator for density estimation on the two-dimensional
sphere $S^2$, equipped with the normalized surface measure $ d\nu=\frac{d\sigma}{4\pi}$.
We begin by describing the experimental setup. The target density is
\[
    p_0(x)
    =
    1+\sum_{\ell\in\{2,6,12\}}
    a_\ell\sqrt{2\ell+1}\,
    P_\ell(\langle x,u_\ell\rangle),
\]
where
\[
    a_2=0.15,
    \qquad
    a_6=0.085,
    \qquad
    a_{12}=0.055,
\]
and $u_\ell\in S^2$ are fixed directions. Thus, the target density has
exactly three active spectral blocks. Independent observations were
generated by exact rejection sampling, using the uniform distribution on
$S^2$ as the proposal distribution. A heatmap of $p_0$, displayed using
the Mollweide projection, is shown in
Figure~\ref{fig:experiments}\textbf{(a)}.
\begin{figure}[H]
    \centering
    \includegraphics[width=\linewidth]{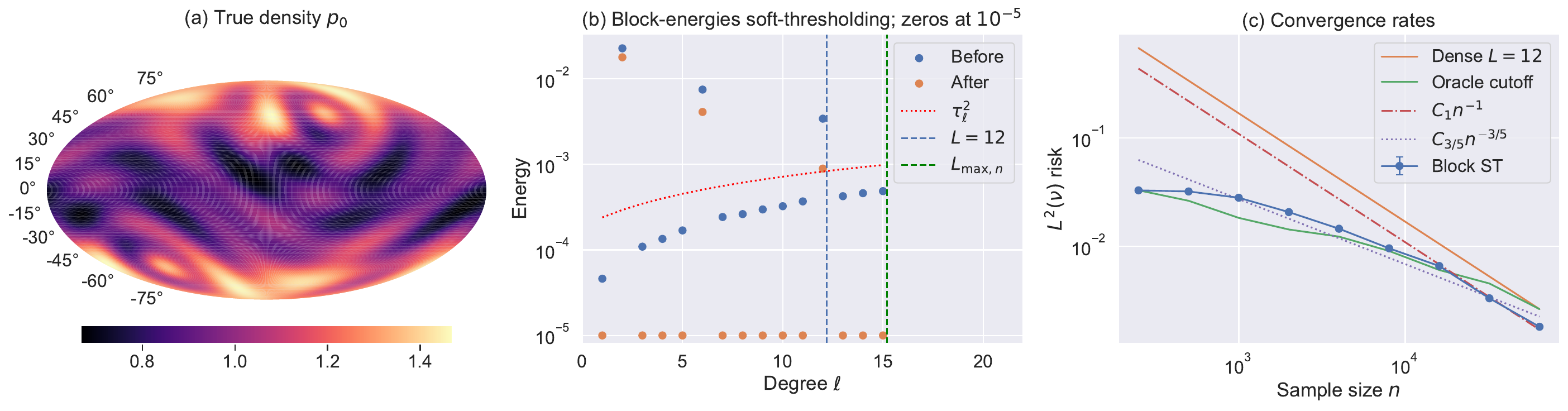}
    \caption{
        \textbf{(a)} Heatmap of the target density $p_0$ on $S^2$.
        \textbf{(b)} Energies of the empirical spectral blocks before and
        after soft thresholding for $n = 64,000$. The dashed curve represents the squared
        thresholds $\tau_\ell^2$, while the vertical dashed lines indicate
        the fixed dense cutoff $L=12$ and the sample-size-dependent maximal
        degree $L_{\max,n}$ used by the block-shrinkage estimator.
        \textbf{(c)} Convergence of the risks of the dense estimator with
        fixed cutoff $\ell=12$, the oracle dense estimator with the
        risk-minimizing cutoff $\ell_{\mathrm{or}}(n)$, which is unavailable
        in practice, and the block-shrinkage estimator, denoted by
        ``Block ST'' in the legend.
        The constants $C_1$ and $C_{3/5}$ in the reference curves
        $C_1n^{-1}$ and $C_{3/5}n^{-3/5}$, respectively, are fitted using
        the last three risk estimates of the block-shrinkage estimator.
        The curve $C_{3/5}n^{-3/5}$ represents the worst-case rate predicted
        by our bound over the class $\cP_1(R,b)$, whereas $C_1n^{-1}$
        represents the parametric benchmark attainable when the finite
        collection of active spectral blocks is known.
    }
    \label{fig:experiments}
\end{figure}

For each spherical-harmonic degree $\ell$, we computed the empirical
coefficient vector $\widehat\theta_\ell$ and applied Euclidean soft
thresholding to the entire eigenspace:
\[
    \widetilde\theta_\ell
    =
    \left(
        1-\frac{\tau_\ell}{\|\widehat\theta_\ell\|_2}
    \right)_+
    \widehat\theta_\ell.
\]

The threshold accounts for the block dimension
$d_\ell=2\ell+1$ and provides simultaneous control over all degrees in
the spectral dictionary. For $s=1$, the maximal eigenvalue
$\Lambda_{\max,n}$ was selected using the tail condition from
Theorem~\ref{thm:multiplicity-sensitive-upper}:
\[
    A_R^2(1+\Lambda_{\max,n})^{-1}
    \leq
    \rho_n,
    \qquad
    \rho_n
    =
    A_R^{4/5}
    \left(
        \frac{b\log n}{n}
    \right)^{3/5},
\]
where $A_R=R-1$ and $b\approx1.92$ is a valid upper bound on
$\|p_0\|_\infty$. The corresponding maximal spherical-harmonic degree is
\[
    L_{\max,n}
    =
    \max\bigl\{
        \ell\in\mathbb N:
        \ell(\ell+1)\leq\Lambda_{\max,n}
    \bigr\}.
\]

The thresholds were then defined by
\[
    \tau_\ell
    =
    0.9
    \left[
        \sqrt{
            \frac{b(d_\ell+\xi_n)}{n}
        }
        +
        (1+\sqrt b)
        \frac{\sqrt{d_\ell}\,\xi_n}{n}
    \right],
    \qquad
    \xi_n
    =
    \log\left(
        \frac{2L_{\max,n}}{0.05}
    \right).
\]

As a benchmark, we considered an unshrunk dense spectral estimator with
the fixed cutoff $L=12$. This estimator includes all 168 nonconstant
spherical-harmonic coefficients up to and including degree $12$,
regardless of whether the corresponding spectral blocks contain signal.
In addition, we report the risk of the dense spectral estimator with the
oracle cutoff $\ell_{\mathrm{or}}(n)$ that minimizes its expected risk.
Such a choice is infeasible in practice because it depends on the
unknown target density.

For each sample size, we used 500 independent Monte Carlo replications.
The mean $L^2(\nu)$ risk of the block-shrinkage estimator decreased from
$3.27\times10^{-2}$ at $n=250$ to $1.8\times10^{-3}$ at
$n=64{,}000$. The corresponding risks of the fixed-cutoff dense
estimator were $6.72\times10^{-1}$ and $2.62\times10^{-3}$,
respectively. The risk curves and their Monte Carlo confidence bands are
shown in Figure~\ref{fig:experiments}\textbf{(c)}.

We also studied the progressive detection of the active blocks. The
degree-$2$ block was retained with probability close to one for
$n\geq2{,}000$, the degree-$6$ block for $n\geq8{,}000$, and the
weaker degree-$12$ block for $n\geq32{,}000$. The probability of
retaining at least one inactive block remained small and did not exceed
$1.6\%$ over the considered sample-size grid.

These results support the oracle interpretation of the theoretical
bound. Unlike a dense spectral cutoff, block shrinkage does not incur
the full variance cost of estimating every eigenspace below a fixed
degree. Instead, it primarily retains the harmonic blocks whose energies
exceed the corresponding noise levels; see
Figure~\ref{fig:experiments}\textbf{(b)}. For small sample sizes, the
weaker active blocks may remain below their thresholds and are therefore
suppressed, producing an additional bias. As the sample size increases,
the thresholds decrease, the active blocks are reliably detected, and
the purely noisy degrees remain largely suppressed.

Additional sensitivity experiments are reported in Appendix~\ref{sec:additional-experiments}. Most notably, we compare block and coordinatewise soft thresholding
under rotations of the eigenbasis within a multiple eigenspace, while keeping
the target density and the observations unchanged. The coordinatewise risk
can vary by nearly an order of magnitude solely because of this arbitrary
change of basis, whereas the block estimator remains exactly invariant. We
also investigate the sensitivity to the threshold multiplier and compare the
resulting estimator with dense spectral benchmarks.

\section{Discussion and outlook}
\label{sec:discussion}

Complete Laplace eigenspaces are the canonical spectral blocks associated with
a fixed Riemannian metric. Block thresholding exploits sparsity across these
blocks while retaining the unavoidable cost of estimating an unknown direction
inside an active one. The oracle inequality and matching lower bounds identify
a single high-multiplicity eigenspace as least favorable in the rough regime
and a many-block spectral window in the smoother regime. On $S^m$, $m\geq2$,
and on $SO(3)$, summable coding attains the ideal expected-risk scale; the
positive likelihood construction transfers the same intrinsic geometry to
bounded centered log-densities while enforcing positivity and normalization.

The remaining limitations concern calibration and adaptivity. Priorities for
further work include bounded-multiplicity and nonhomogeneous manifolds,
positive-sieve or localized likelihood arguments when $q<\beta$, adaptation
to unknown smoothness, radius, and density envelope, combinations of exact
spectral invariance with spatial localization, and matching lower bounds for
the log-density classes.

\section*{Acknowledgments}
During the preparation of this manuscript, the authors used
OpenAI's ChatGPT for language editing and \LaTeX{} assistance. All AI-assisted material was
independently checked and revised by the authors, who take full
responsibility for the mathematical content, references, and
numerical results.

\bibliographystyle{plain}
\bibliography{intrinsic_spectral_block_variation_refs_revised}

@article{wainwright2008graphical,
  author  = {Wainwright, Martin J. and Jordan, Michael I.},
  title   = {Graphical Models, Exponential Families, and Variational Inference},
  journal = {Foundations and Trends in Machine Learning},
  volume  = {1},
  number  = {1--2},
  pages   = {1--305},
  year    = {2008},
  doi     = {10.1561/2200000001}
}

@article{Sogge1988SpectralClusters,
  author  = {Sogge, Christopher D.},
  title   = {Concerning the {$L^p$} norm of spectral clusters for
             second-order elliptic operators on compact manifolds},
  journal = {Journal of Functional Analysis},
  volume  = {77},
  number  = {1},
  pages   = {123--138},
  year    = {1988},
  doi     = {10.1016/0022-1236(88)90081-X}
}

@article{gine1975addition,
  author  = {Gin{\'e}, Evarist},
  title   = {The Addition Formula for the Eigenfunctions of the Laplacian},
  journal = {Advances in Mathematics},
  volume  = {18},
  number  = {1},
  pages   = {102--107},
  year    = {1975},
  doi     = {10.1016/0001-8708(75)90003-1}
}

@article{barron_sheu_1991,
  author  = {Barron, Andrew R. and Sheu, Chyong-Hwa},
  title   = {Approximation of Density Functions by Sequences of Exponential Families},
  journal = {The Annals of Statistics},
  year    = {1991},
  volume  = {19},
  number  = {3},
  pages   = {1347--1369}
}

@article{bousquet2002bennett,
  author  = {Bousquet, Olivier},
  title   = {A {Bennett} Concentration Inequality and Its Application to Suprema of Empirical Processes},
  journal = {Comptes Rendus Math\'{e}matique},
  year    = {2002},
  volume  = {334},
  number  = {6},
  pages   = {495--500}
}

@article{burq2013injections,
  author  = {Burq, Nicolas and Lebeau, Gilles},
  title   = {Injections de {Sobolev} probabilistes et applications},
  journal = {Annales scientifiques de l'\'{E}cole Normale Sup\'{e}rieure},
  year    = {2013},
  volume  = {46},
  number  = {6},
  pages   = {917--962}
}

@article{Cai1999,
  author  = {Cai, T. Tony},
  title   = {Adaptive Wavelet Estimation: A Block Thresholding and Oracle Inequality Approach},
  journal = {The Annals of Statistics},
  year    = {1999},
  volume  = {27},
  number  = {3},
  pages   = {898--924}
}

@article{cammarota_marinucci_2015,
  author  = {Cammarota, Valentina and Marinucci, Domenico},
  title   = {The Stochastic Properties of $\ell^1$-Regularized Spherical Gaussian Fields},
  journal = {Applied and Computational Harmonic Analysis},
  year    = {2015},
  volume  = {38},
  number  = {2},
  pages   = {262--283}
}

@book{chavel1984eigenvalues,
  author    = {Chavel, Isaac},
  title     = {Eigenvalues in Riemannian Geometry},
  series    = {Pure and Applied Mathematics},
  volume    = {115},
  publisher = {Academic Press},
  address   = {Orlando, FL},
  year      = {1984}
}

@article{cleanthous_et_al_2020_density,
  author  = {Cleanthous, Galatia and Georgiadis, Athanasios G. and Kerkyacharian, G\'{e}rard and Petrushev, Pencho and Picard, Dominique},
  title   = {Kernel and Wavelet Density Estimators on Manifolds and More General Metric Spaces},
  journal = {Bernoulli},
  year    = {2020},
  volume  = {26},
  number  = {3},
  pages   = {1832--1862}
}

@inproceedings{cohen_welling_2015,
  author    = {Cohen, Taco S. and Welling, Max},
  title     = {Harmonic Exponential Families on Manifolds},
  booktitle = {Proceedings of the 32nd International Conference on Machine Learning},
  editor    = {Bach, Francis and Blei, David},
  series    = {Proceedings of Machine Learning Research},
  volume    = {37},
  pages     = {1757--1765},
  publisher = {PMLR},
  year      = {2015}
}

@article{devore1998nonlinear,
  author  = {DeVore, Ronald A.},
  title   = {Nonlinear Approximation},
  journal = {Acta Numerica},
  year    = {1998},
  volume  = {7},
  pages   = {51--150}
}

@article{durastanti_2015_block,
  author  = {Durastanti, Claudio},
  title   = {Block Thresholding on the Sphere},
  journal = {Sankhya A},
  year    = {2015},
  volume  = {77},
  number  = {1},
  pages   = {153--185}
}

@book{edmunds_triebel_1996_entropy,
  author    = {Edmunds, David E. and Triebel, Hans},
  title     = {Function Spaces, Entropy Numbers, Differential Operators},
  series    = {Cambridge Tracts in Mathematics},
  volume    = {120},
  publisher = {Cambridge University Press},
  address   = {Cambridge},
  year      = {1996}
}

@article{geller_mayeli_2009_besov,
  author  = {Geller, Daryl and Mayeli, Azita},
  title   = {Besov Spaces and Frames on Compact Manifolds},
  journal = {Indiana University Mathematics Journal},
  year    = {2009},
  volume  = {58},
  number  = {5},
  pages   = {2003--2042}
}

@article{geller_mayeli_2009_frames,
  author  = {Geller, Daryl and Mayeli, Azita},
  title   = {Nearly Tight Frames and Space--Frequency Analysis on Compact Manifolds},
  journal = {Mathematische Zeitschrift},
  year    = {2009},
  volume  = {263},
  number  = {2},
  pages   = {235--264}
}

@misc{greco_marinucci_2026_sparsity,
  author       = {Greco, Giacomo and Marinucci, Domenico},
  title        = {Sparsity for Isotropic Spherical Random Fields},
  year         = {2026},
  eprint       = {2601.21535},
  archivePrefix= {arXiv:2601.21535},
  primaryClass = {math.ST}
}

@book{helgason2000groups,
  author    = {Helgason, Sigurdur},
  title     = {Groups and Geometric Analysis: Integral Geometry, Invariant Differential Operators, and Spherical Functions},
  series    = {Mathematical Surveys and Monographs},
  volume    = {83},
  publisher = {American Mathematical Society},
  address   = {Providence, RI},
  year      = {2000}
}

@article{hendriks1990nonparametric,
  author  = {Hendriks, Harrie},
  title   = {Nonparametric Estimation of a Probability Density on a {Riemannian} Manifold Using {Fourier} Expansions},
  journal = {The Annals of Statistics},
  year    = {1990},
  volume  = {18},
  number  = {2},
  pages   = {832--849},
  doi     = {10.1214/aos/1176347628}
}

@article{hormander1968spectral,
  author  = {H\"{o}rmander, Lars},
  title   = {The Spectral Function of an Elliptic Operator},
  journal = {Acta Mathematica},
  year    = {1968},
  volume  = {121},
  pages   = {193--218}
}

@article{KerkyacharianNicklPicard2012,
  author  = {Kerkyacharian, G\'{e}rard and Nickl, Richard and Picard, Dominique},
  title   = {Concentration Inequalities and Confidence Bands for Needlet Density Estimators on Compact Homogeneous Manifolds},
  journal = {Probability Theory and Related Fields},
  year    = {2012},
  volume  = {153},
  number  = {1--2},
  pages   = {363--404}
}

@article{le_gia_sloan_womersley_wang_2020,
  author  = {Le Gia, Quoc T. and Sloan, Ian H. and Womersley, Robert S. and Wang, Yu Guang},
  title   = {Isotropic Sparse Regularization for Spherical Harmonic Representations of Random Fields on the Sphere},
  journal = {Applied and Computational Harmonic Analysis},
  year    = {2020},
  volume  = {49},
  number  = {1},
  pages   = {257--278}
}

@article{li_chen_2024_group_sparse,
  author  = {Li, Chao and Chen, Xiaojun},
  title   = {Group Sparse Optimization for Inpainting of Random Fields on the Sphere},
  journal = {IMA Journal of Numerical Analysis},
  year    = {2024},
  volume  = {44},
  number  = {5},
  pages   = {3028--3058}
}

@article{liao2025spectral,
  author  = {Liao, Yulei and Ming, Pingbing},
  title   = {Spectral {Barron} Space for Deep Neural Network Approximation},
  journal = {SIAM Journal on Mathematics of Data Science},
  year    = {2025},
  volume  = {7},
  number  = {3},
  pages   = {1053--1076}
}

@inproceedings{lu2021apriori,
  author    = {Lu, Yulong and Lu, Jianfeng and Wang, Min},
  title     = {A Priori Generalization Analysis of the Deep {Ritz} Method for Solving High Dimensional Elliptic Partial Differential Equations},
  booktitle = {Proceedings of the Thirty-Fourth Conference on Learning Theory},
  editor    = {Belkin, Mikhail and Kpotufe, Samory},
  series    = {Proceedings of Machine Learning Research},
  volume    = {134},
  pages     = {3196--3241},
  publisher = {PMLR},
  year      = {2021}
}

@article{meier_vandegeer_buhlmann_2008,
  author  = {Meier, Lukas and van de Geer, Sara and B\"{u}hlmann, Peter},
  title   = {The Group Lasso for Logistic Regression},
  journal = {Journal of the Royal Statistical Society: Series B (Statistical Methodology)},
  year    = {2008},
  volume  = {70},
  number  = {1},
  pages   = {53--71}
}

@misc{mensah2025spectral,
  author       = {Mensah, Yaogan and Aremua, Isiaka},
  title        = {Spectral {Barron} Spaces of Vector-Valued Functions on Compact Groups},
  year         = {2025},
  eprint       = {2512.12382},
  archivePrefix= {arXiv:2512.12382},
  primaryClass = {math.FA}
}

@article{negahban_ravikumar_wainwright_yu_2012,
  author  = {Negahban, Sahand N. and Ravikumar, Pradeep and Wainwright, Martin J. and Yu, Bin},
  title   = {A Unified Framework for High-Dimensional Analysis of $M$-Estimators with Decomposable Regularizers},
  journal = {Statistical Science},
  year    = {2012},
  volume  = {27},
  number  = {4},
  pages   = {538--557}
}

@article{ParikhBoyd2014,
  author  = {Parikh, Neal and Boyd, Stephen},
  title   = {Proximal Algorithms},
  journal = {Foundations and Trends in Optimization},
  year    = {2014},
  volume  = {1},
  number  = {3},
  pages   = {127--239}
}

@article{pinelis1994optimum,
  author  = {Pinelis, Iosif},
  title   = {Optimum Bounds for the Distributions of Martingales in {Banach} Spaces},
  journal = {The Annals of Probability},
  year    = {1994},
  volume  = {22},
  number  = {4},
  pages   = {1679--1706}
}

@book{rosenberg1997laplacian,
  author    = {Rosenberg, Steven},
  title     = {The Laplacian on a Riemannian Manifold: An Introduction to Analysis on Manifolds},
  series    = {London Mathematical Society Student Texts},
  volume    = {31},
  publisher = {Cambridge University Press},
  address   = {Cambridge},
  year      = {1997}
}

@article{shi2010gradient,
  author  = {Shi, Yiqian and Xu, Bin},
  title   = {Gradient Estimate of an Eigenfunction on a Compact {Riemannian} Manifold without Boundary},
  journal = {Annals of Global Analysis and Geometry},
  year    = {2010},
  volume  = {38},
  number  = {1},
  pages   = {21--26}
}

@book{tsybakov2008nonparametric,
  author    = {Tsybakov, Alexandre B.},
  title     = {Introduction to Nonparametric Estimation},
  series    = {Springer Series in Statistics},
  publisher = {Springer},
  address   = {New York},
  year      = {2009}
}

@article{VanDeGeer2008,
  author  = {van de Geer, Sara A.},
  title   = {High-Dimensional Generalized Linear Models and the Lasso},
  journal = {The Annals of Statistics},
  year    = {2008},
  volume  = {36},
  number  = {2},
  pages   = {614--645}
}

@book{Vershynin2026,
  author    = {Vershynin, Roman},
  title     = {High-Dimensional Probability: An Introduction with Applications in Data Science},
  edition   = {2nd},
  series    = {Cambridge Series in Statistical and Probabilistic Mathematics},
  publisher = {Cambridge University Press},
  address   = {Cambridge},
  year      = {2026}
}

@article{yuan_lin_2006,
  author  = {Yuan, Ming and Lin, Yi},
  title   = {Model Selection and Estimation in Regression with Grouped Variables},
  journal = {Journal of the Royal Statistical Society: Series B (Statistical Methodology)},
  year    = {2006},
  volume  = {68},
  number  = {1},
  pages   = {49--67}
}

\clearpage
\section*{Appendix}
\begin{table}[!htbp]
\centering
\footnotesize
\setlength{\tabcolsep}{3.2pt}
\renewcommand{\arraystretch}{1.05}
\noindent\textbf{Notation at a glance.}\par\smallskip
\begin{tabularx}{\textwidth}{@{}>{\raggedright\arraybackslash}p{0.15\textwidth}Y@{\hspace{0.6em}}>{\raggedright\arraybackslash}p{0.15\textwidth}Y@{}}
\toprule
Notation & Meaning & Notation & Meaning \\
\midrule
\((M,g),\nu,m\)
& Manifold, normalized volume, and intrinsic dimension.
& \(L=-\Delta_M\)
& Nonnegative Laplace--Beltrami operator. \\

\(\mu_j,E_j,P_j,d_j\)
& Distinct eigenvalue, eigenspace, orthogonal projector, and multiplicity.
& \(K_j,\kappa_j\)
& Projector kernel and envelope, \(\kappa_j=\sup_xK_j(x,x)\). \\

\(E_{\leq\Lambda},\Pi_{\leq\Lambda},D(\Lambda)\)
& Low-frequency space, projector, and dimension \(D(\Lambda)=\sum_{\mu_j\leq\Lambda}d_j\).
& \(\Vspec_s(M)\)
& Spectral block-variation space with weighted \(\ell^1/\ell^2\) norm. \\

\(\Pcal_s(R,b),A_R\)
& Upper-bound density class and nonconstant budget \(A_R=R-1\).
& \(\Pcal_s(R,a,b)\)
& Two-sided density class used for lower bounds. \\

\(\alpha,\gamma\)
& \(\gamma\): individual block growth; \(\alpha+\gamma\): cumulative spectral dimension.
& \(\delta_n,L_n,\rho_{\alpha,\gamma,s}\)
& Effective squared-noise level, common coding-level inflation, and rate function. \\

\(\cJ_n,\tau_j,\eta_j\)
& Retained blocks, thresholds, and deviation remainders.
& \(A(u),p_u\)
& Log-partition functional and density \(p_u=e^{u-A(u)}\). \\

\(\Ucal_q(R,B)\)
& Centered log-density block-variation class.
& \(B_\star,b_0,\beta,\lambda_j\)
& Sieve radius, density envelope, projector-envelope exponent, and block penalties. \\

\(\mathbb P_n,\mathbb P_0\)
& Empirical and population averages in the likelihood section.
& \(\mathcal E_J\)
& Simultaneous block score-control event for the likelihood estimator. \\
\bottomrule
\end{tabularx}
\end{table}
\FloatBarrier
\begin{appendix}

\section{Proofs of structural facts}\label{app:structural-proofs}

\subsection{Proof of Proposition \ref{prop:structural}}

Let \(w_j=(1+\mu_j)^{s/2}\).  The map
\[
        f\longmapsto (w_jP_jf)_{j\geq0}
\]
is an isometric identification of \(\Vspec_s(M)\) with the
\(\ell^1\)-direct sum of the finite-dimensional Hilbert spaces \(E_j\).
Completeness and density of finite spectral sums therefore follow from the
corresponding elementary facts for \(\ell^1\).  Parseval and
\(\ell^1\subset\ell^2\) give
\[
        \|f\|_{L^2}
        \leq\sum_{j\geq0}\|P_jf\|_{L^2}
        \leq\|f\|_{\Vspec_s},
\]
and monotonicity in \(s\) follows directly from the weights.  Finally, an
isometric pullback is unitary and commutes with \(L\), hence with every
spectral projector \(P_j\); it therefore preserves each block norm and their
weighted sum.

\subsection{Proof of Proposition~\ref{prop:block-compressibility}}\label{proof_prop:block-compressibility}

Set \(a_j=(1+\mu_j)^{s/2}\|P_jf\|_{L^2(M)}\), and let \(S_K\) contain the
indices of the \(K\) largest \(a_j\)'s.  Stechkin's best-term inequality
\cite{devore1998nonlinear} gives
\[
        \left(\sum_{j\notin S_K}a_j^2\right)^{1/2}
        \leq \frac{\sum_{j\geq0}a_j}{\sqrt{K+1}}.
\]
Orthogonality identifies the left-hand side with
\(\|f-\sum_{j\in S_K}P_jf\|_{H^s(M)}\), while the numerator is
\(\|f\|_{\Vspec_s(M)}\).  The \(L^2\) statement follows from
\(\|g\|_{L^2}\leq\|g\|_{H^s}\).

\subsection{Proof of Proposition~\ref{prop:atomic-basis-envelope}}

Write \(w_j=(1+\mu_j)^{s/2}\). Suppose first that
\(\|f\|_{\Vspec_s(M)}\leq1\). For every \(j\) with \(P_jf\neq0\), set
\[
        c_j:=w_j\|P_jf\|_{L^2(M)},
        \qquad
        a_j:=w_j^{-1}\frac{P_jf}{\|P_jf\|_{L^2(M)}}\in\Acal_s.
\]
Then
\[
        f=\sum_{j:P_jf\neq0}c_ja_j
        \quad\text{in }L^2(M),
        \qquad
        \sum_jc_j=\|f\|_{\Vspec_s(M)}\leq1.
\]
Thus the unit ball is contained in the \(L^2\)-closed absolutely convex hull
of \(\Acal_s\). Conversely, every finite sum
\(g=\sum_{r=1}^Nc_ra_r\), with \(a_r\in\Acal_s\) and
\(\sum_r|c_r|\leq1\), satisfies
\[
        \|g\|_{\Vspec_s(M)}
        \leq
        \sum_{r=1}^N|c_r|\|a_r\|_{\Vspec_s(M)}
        \leq1.
\]
The block-variation norm is lower semicontinuous for \(L^2\)-convergence:
for each fixed \(J\), the finite partial sum
\(\sum_{j=0}^Jw_j\|P_jf\|_2\) is continuous in \(L^2\), and taking the
supremum over \(J\) gives lower semicontinuity. Hence the \(L^2\)-closure of
the absolutely convex hull is also contained in the unit ball.

For the basis-envelope identity, let
\(\Bcal=\{\phi_{j,\ell}\}\in\mathfrak E(L)\). In each block,
\[
        \sum_{\ell=1}^{d_j}
        |\langle f,\phi_{j,\ell}\rangle_{L^2(M)}|
        \geq
        \left(
        \sum_{\ell=1}^{d_j}
        |\langle f,\phi_{j,\ell}\rangle_{L^2(M)}|^2
        \right)^{1/2}
        =\|P_jf\|_{L^2(M)}.
\]
After multiplication by \(w_j\) and summation over \(j\), this gives
\(\|f\|_{\mathrm{SB},s;\Bcal}\geq\|f\|_{\Vspec_s(M)}\). For every block
with \(P_jf\neq0\), choose the first basis vector to be
\(P_jf/\|P_jf\|_2\) and complete it to an orthonormal basis of \(E_j\); in a
zero block choose any orthonormal basis. In the resulting eigenbasis, only the
first coordinate in each nonzero block is nonzero, so equality holds. This
proves both attainment and \eqref{eq:basis-envelope-characterization}.

\subsection{Proof of Proposition \ref{prop:besov-embedding}}

Use the compact-manifold Littlewood--Paley characterization from
\cite{geller_mayeli_2009_besov}.  If
\(Q_n=\psi_n((I+L)^{1/2})\) is a smooth dyadic spectral multiplier, then
\[
        Q_nf=\sum_{j\geq0}\psi_n((1+\mu_j)^{1/2})P_jf
\]
and hence
\[
        \|Q_nf\|_{L^2}
        \leq
        \sum_{j\geq0}
        \big|\psi_n((1+\mu_j)^{1/2})\big|\,\|P_jf\|_{L^2}.
\]
Tonelli's theorem and finite overlap of the dyadic supports yield
\[
\begin{aligned}
        \|f\|_{B^s_{2,1}(M)}
        &\lesssim
        \sum_{j\geq0}
        \left[\sum_{n\geq0}2^{ns}
        \big|\psi_n((1+\mu_j)^{1/2})\big|\right]
        \|P_jf\|_{L^2}\\
        &\lesssim
        \sum_{j\geq0}(1+\mu_j)^{s/2}\|P_jf\|_{L^2}
        =\|f\|_{\Vspec_s(M)}.
\end{aligned}
\]
The remaining embedding
\(B^s_{2,1}(M)\hookrightarrow B^s_{2,2}(M)=H^s(M)\) is the standard
\(\ell^1\hookrightarrow\ell^2\) embedding in the same Littlewood--Paley
characterization \cite{geller_mayeli_2009_besov}.
\subsection{Strictness of the Besov embedding}
\label{sec:strict-besov}

Set
\[
a_j := (1+\mu_j)^{1/2}.
\]
The following proposition gives a sufficient condition under which the
embedding
\[
V_s^{\mathrm{spec}}(M)\hookrightarrow B^s_{2,1}(M)
\]
is strict.

\begin{proposition}[A spectral-shell criterion for strict inclusion]
\label{prop:strict-besov}
Let \(s\geq 0\). Suppose that there exist constants
\(0<c_-<c_+<\infty\), an integer \(k_0\geq 1\), and pairwise disjoint
finite sets of eigenspace indices
\[
\mathcal J_k\subset \{1,2,\ldots\},
\qquad k\geq k_0,
\]
such that
\[
c_-2^k
\leq
a_j
\leq
c_+2^k,
\qquad
j\in\mathcal J_k.
\tag{A.1}
\label{eq:strict-shell-localization}
\]
Write
\[
K_k:=\#\mathcal J_k.
\]
If
\[
\sum_{k\geq k_0}
\frac{\sqrt{K_k}}{(1+k)^2}
=
\infty,
\tag{A.2}
\label{eq:strict-shell-growth}
\]
then
\[
V_s^{\mathrm{spec}}(M)
\subsetneq
B^s_{2,1}(M).
\]

Condition \eqref{eq:strict-shell-growth} holds, in particular, on every
round sphere \(S^m\), \(m\geq 1\), and every standard flat torus
\(T^d=(\mathbb R/2\pi\mathbb Z)^d\), \(d\geq 1\).
\end{proposition}

\begin{proof}
For every \(k\geq k_0\) and \(j\in\mathcal J_k\), choose a unit vector
\[
u_j\in E_j,
\qquad
\|u_j\|_{L^2(M)}=1.
\]
Define
\[
c_k
:=
\frac{2^{-ks}}{(1+k)^2\sqrt{K_k}}
\]
and consider the orthogonal series
\[
f
:=
\sum_{k\geq k_0}
c_k
\sum_{j\in\mathcal J_k}u_j.
\tag{A.3}
\label{eq:strict-besov-construction}
\]

Since distinct Laplace eigenspaces are mutually orthogonal,
\[
\|f\|_{L^2(M)}^2
=
\sum_{k\geq k_0}K_kc_k^2
=
\sum_{k\geq k_0}
\frac{2^{-2ks}}{(1+k)^4}
<
\infty.
\]
Hence the series in \eqref{eq:strict-besov-construction} converges in
\(L^2(M)\).

We first show that \(f\notin V_s^{\mathrm{spec}}(M)\). For
\(j\in\mathcal J_k\), one has
\[
P_jf=c_ku_j,
\qquad
\|P_jf\|_{L^2(M)}=c_k.
\]
Using the lower bound in
\eqref{eq:strict-shell-localization}, we obtain
\begin{align*}
\|f\|_{V_s^{\mathrm{spec}}(M)}
&=
\sum_{j\geq 0}
a_j^s\|P_jf\|_{L^2(M)}
\\
&\geq
\sum_{k\geq k_0}
\sum_{j\in\mathcal J_k}
(c_-2^k)^s c_k
\\
&=
c_-^s
\sum_{k\geq k_0}
2^{ks}K_kc_k
\\
&=
c_-^s
\sum_{k\geq k_0}
\frac{\sqrt{K_k}}{(1+k)^2}.
\end{align*}
The last series diverges by
\eqref{eq:strict-shell-growth}. Therefore,
\[
f\notin V_s^{\mathrm{spec}}(M).
\]

It remains to prove that \(f\in B^s_{2,1}(M)\). Let
\[
A=(I+L)^{1/2},
\]
and use the dyadic decomposition from the proof of
Proposition~\ref{prop:besov-embedding}:
\[
Q_0=\psi_0(A),
\qquad
Q_n=\psi(2^{-n}A),
\quad n\geq 1.
\]
Because the multiplier \(\psi\) is compactly supported and the
frequencies in \(\mathcal J_k\) satisfy
\eqref{eq:strict-shell-localization}, there exists an integer
\(N_0\geq 1\), depending only on \(c_-\), \(c_+\), and the support of
\(\psi\), such that
\[
Q_nu_j=0
\qquad
\text{whenever }
j\in\mathcal J_k
\text{ and }
|n-k|>N_0.
\tag{A.4}
\label{eq:strict-finite-overlap}
\]
Moreover,
\[
Q_nu_j=\psi(2^{-n}a_j)u_j.
\]
By orthogonality and the uniform boundedness of the multipliers,
\begin{align*}
\|Q_nf\|_{L^2(M)}^2
&=
\sum_{\substack{k\geq k_0\\ |k-n|\leq N_0}}
c_k^2
\sum_{j\in\mathcal J_k}
\left|\psi(2^{-n}a_j)\right|^2
\\
&\leq
C
\sum_{\substack{k\geq k_0\\ |k-n|\leq N_0}}
K_kc_k^2.
\end{align*}
Since \(|k-n|\leq N_0\) implies \(2^{ns}\leq C2^{ks}\), it follows
that
\begin{align*}
2^{ns}\|Q_nf\|_{L^2(M)}
&\leq
C
\left(
\sum_{\substack{k\geq k_0\\ |k-n|\leq N_0}}
2^{2ks}K_kc_k^2
\right)^{1/2}
\\
&=
C
\left(
\sum_{\substack{k\geq k_0\\ |k-n|\leq N_0}}
\frac{1}{(1+k)^4}
\right)^{1/2}
\\
&\leq
C
\sum_{\substack{k\geq k_0\\ |k-n|\leq N_0}}
\frac{1}{(1+k)^2}.
\end{align*}
Summing over \(n\) and using the finite-overlap property
\eqref{eq:strict-finite-overlap} gives
\[
\sum_{n\geq 1}
2^{ns}\|Q_nf\|_{L^2(M)}
\leq
C
\sum_{k\geq k_0}
\frac{1}{(1+k)^2}
<
\infty.
\]
The low-frequency term \(\|Q_0f\|_{L^2(M)}\) is finite because
\(f\in L^2(M)\). Thus
\[
f\in B^s_{2,1}(M).
\]
Together with \(f\notin V_s^{\mathrm{spec}}(M)\), this proves the
strict inclusion.

We finally verify the shell condition in the stated examples.

On the round sphere \(S^m\), the distinct eigenspaces are indexed by
the harmonic degree \(\ell\), with
\[
\mu_\ell=\ell(\ell+m-1),
\qquad
(1+\mu_\ell)^{1/2}\asymp 1+\ell.
\]
For each sufficiently large \(k\), choose the degrees
\[
2^k\leq \ell<2^{k+1}.
\]
They determine \(K_k\asymp 2^k\) distinct eigenspaces satisfying
\eqref{eq:strict-shell-localization}. Hence
\[
\sum_k
\frac{\sqrt{K_k}}{(1+k)^2}
\gtrsim
\sum_k
\frac{2^{k/2}}{(1+k)^2}
=
\infty.
\]

On the standard flat torus \(T^d\), the eigenvalues are
\[
|q|^2,
\qquad
q\in\mathbb Z^d.
\]
For every positive integer \(r\), the lattice vector
\[
(r,0,\ldots,0)
\]
produces the eigenvalue \(r^2\). Choose the distinct eigenspaces
corresponding to
\[
2^k\leq r<2^{k+1}.
\]
Again,
\[
(1+r^2)^{1/2}\asymp 2^k
\]
and \(K_k\asymp 2^k\), so
\eqref{eq:strict-shell-growth} follows in exactly the same way.
\end{proof}

\begin{remark}
The sequence \((1+k)^{-2}\) has no special significance. More
generally, the same argument applies whenever there exists a positive
sequence \((b_k)\in\ell^1\) such that
\[
\sum_k b_k\sqrt{K_k}=\infty.
\]
One then replaces \(c_k\) in
\eqref{eq:strict-besov-construction} by
\[
c_k=\frac{2^{-ks}b_k}{\sqrt{K_k}}.
\]
The Besov contribution of shell \(k\) is of order \(b_k\), whereas its
block-variation contribution is of order \(b_k\sqrt{K_k}\).
\end{remark}

\subsection{Proof of Proposition \ref{prop:truncation}}

By orthogonality and the triangle inequality,
\[
\begin{aligned}
        \|f-\Pi_{\leq\Lambda}f\|_{L^2}
        &\leq \sum_{\mu_j>\Lambda}\|P_jf\|_{L^2}\\
        &\leq (1+\Lambda)^{-s/2}
        \sum_{\mu_j>\Lambda}(1+\mu_j)^{s/2}\|P_jf\|_{L^2}\\
        &\leq (1+\Lambda)^{-s/2}\|f\|_{\Vspec_s}.
\end{aligned}
\]
The remaining assertions follow from the definition of \(\Lambda_\eps\) and
Weyl's bound \eqref{eq:weyl}.

\section{Approximation, kernels, needlets, and entropy}\label{app:approx-entropy}

\subsection{Projection kernels}

Let \(K_j\) be the integral kernel of \(P_j\).  For an orthonormal basis \(\{\varphi_{j,\ell}\}_{\ell=1}^{d_j}\) of \(E_j\),
\[
K_j(x,y)=\sum_{\ell=1}^{d_j}\varphi_{j,\ell}(x)\varphi_{j,\ell}(y),
\]
with complex conjugation inserted in the complex case.  The kernel is independent of the basis.  The low-frequency projection kernel is
\[
K_{\le\Lambda}(x,y)=\sum_{\mu_j\le\Lambda}K_j(x,y).
\]
Then
\[
\Pi_{\le\Lambda}f(x)=\int_MK_{\le\Lambda}(x,y)f(y)\,\dd\nu(y).
\]

\begin{proposition}[Projection-kernel realization]\label{prop:kernel-realization}
For every \(\Lambda\ge0\), there are \(D(\Lambda)\) points \(z_i\in M\)
such that every \(g\in E_{\le\Lambda}\) can be written as
\[
g(x)=\sum_{i=1}^{D(\Lambda)}a_iK_{\le\Lambda}(x,z_i).
\]
Consequently, if \(f\in\Vspec_s(R)\), then for every \(0<\eps<R\) there is such a kernel expansion \(g_\eps\) with
\[
\norm{f-g_\eps}_{L^2(M)}\le\eps,
\qquad
D(\Lambda_\eps)\le C_W(R/\eps)^{m/s}.
\]
\end{proposition}

\begin{proof}
Let \(E=E_{\le\Lambda}\).  For \(z\in M\), write \(K_z(x)=K_{\le\Lambda}(x,z)\).  The family \(\{K_z:z\in M\}\) spans \(E\).  Otherwise, a nonzero \(h\in E\) would be orthogonal to each \(K_z\), hence \(h(z)=0\) for all \(z\in M\), contradicting \(h\ne0\).  Since \(\dim E=D(\Lambda)\), one may choose \(D(\Lambda)\) such kernels forming a basis.  The approximation statement follows by taking \(g_\eps=\Pi_{\le\Lambda_\eps}f\) and applying Proposition~\ref{prop:truncation}.
\end{proof}

\subsection{Needlet control}

Spectral frames and their Besov coefficient characterizations on compact
manifolds are developed in~\cite{geller_mayeli_2009_besov,geller_mayeli_2009_frames}.
Let
$\{\psi_{n,\xi}\}_{n\geq 0,\,\xi\in\mathcal X_n}$ and
$\{\widetilde{\psi}_{n,\xi}\}_{n\geq 0,\,\xi\in\mathcal X_n}$
be a dual pair of needlet frames. The standard characterization gives
\[
\sum_{n\geq 0} 2^{ns}
\left(
  \sum_{\xi\in\mathcal X_n}
  |\langle f,\psi_{n,\xi}\rangle_{L^2(M)}|^2
\right)^{1/2}
\lesssim
\|f\|_{B^s_{2,1}(M)}
\lesssim
\|f\|_{V^{\mathrm{spec}}_s(M)}.
\]
For $J\geq 0$, define the partial needlet reconstruction through level
$J$ by
\[
S_J f
:=
\sum_{n=0}^{J}\sum_{\xi\in\mathcal X_n}
\langle f,\psi_{n,\xi}\rangle_{L^2(M)}
\,\widetilde{\psi}_{n,\xi}.
\]
In the Parseval-frame case one may take
$\widetilde{\psi}_{n,\xi}=\psi_{n,\xi}$.
The usual needlet reconstruction estimate therefore yields
\[
\|f-S_Jf\|_{L^2(M)}
\lesssim
2^{-Js}\|f\|_{V^{\mathrm{spec}}_s(M)}.
\]

If \(\#X_n\lesssim2^{nm}\), the reconstruction through level \(J\) uses
\(O(2^{Jm})\) atoms; hence \(\Vspec_s(R)\) admits localized approximations
with error \(\eps\) using \(O((R/\eps)^{m/s})\) atoms.

\subsection{Metric entropy and compactness}

\begin{proposition}[Sobolev-based entropy upper bound]\label{prop:entropy}
Let \(s>0\) and \(R>0\).  There exists \(C=C(M,g,s)>0\) such that, for every \(0<\eps<R\),
\[
\log N\left(\eps,\Vspec_s(R),\norm{\cdot}_{L^2(M)}\right)
\le
C(R/\eps)^{m/s}.
\]
Consequently, \(\Vspec_s(R)\) is compact in \(L^2(M)\).
\end{proposition}

\begin{proof}
The inclusion \(\Vspec_s(R)\subset H^s(R)\), together with the standard
entropy estimate for the compact Sobolev embedding
\(H^s(M)\hookrightarrow L^2(M)\), gives the displayed bound; see
\cite{edmunds_triebel_1996_entropy}.  Closedness follows from lower
semicontinuity of the block norm: if \(f_n\to f\) in \(L^2\), then
\(P_jf_n\to P_jf\) for each fixed \(j\), and the finite partial sums pass to
the limit.  Thus the ball is closed and totally bounded in \(L^2(M)\), hence
compact.
\end{proof}

\begin{remark}[Sharpness]
The entropy bound is obtained by embedding into the larger Sobolev ball.  It is useful as a compactness and complexity estimate, but it is not a sharp analysis of the weighted block \(\ell^1/\ell^2\) block-variation ball.  The minimax theory below does not rely on this Sobolev entropy bound as its main input.
\end{remark}

\begin{proposition}[Failure of compactness for \(s=0\)]\label{prop:noncompact-s0}
The ball \(\Vspec_0(R)\) is not totally bounded in \(L^2(M)\) for \(R>0\).
\end{proposition}

\begin{proof}
Choose unit vectors \(u_j\in E_j\).  Then \(f_j=Ru_j\) belongs to \(\Vspec_0(R)\).  For \(i\ne j\), orthogonality gives
\[
\norm{f_i-f_j}_{L^2(M)}=\sqrt2R.
\]
Thus the ball contains an infinite set separated by \(\sqrt2R\), so it is not totally bounded.
\end{proof}
\section{Auxiliary results and proofs for block-shrinkage density estimation}

\subsection{Proof of Theorem~\ref{thm:block-oracle}}

The proof is based on the following deterministic Hilbert-block estimate,
written in Euclidean coordinates.

\begin{lemma}[Soft-thresholding error on one block]
\label{lem:single-block-deterministic}
Let $\theta,y\in\R^d$, let $e=y-\theta$, and let $\tau>0$. Then:
\begin{enumerate}[label=(\roman*)]
\item On $\{\norm e_2\leq\tau/2\}$,
\[
        \norm{S_\tau(y)-\theta}_2^2
        \leq
        9\min\{\norm\theta_2^2,\tau^2\}.
\]
\item Always,
\[
        \norm{S_\tau(y)-\theta}_2^2
        \leq
        2\norm e_2^2+2\tau^2.
\]
\end{enumerate}
\end{lemma}

The formula for \(S_\tau\) is the standard proximal map of
\(\tau\|\cdot\|_2\); see
\cite[Secs.~6.5.1 and 6.5.4]{ParikhBoyd2014}.  We record the short verification
because the constants are used below.

\begin{proof}
If \(\|y\|_2\leq\tau\), then \(S_\tau(y)=0\) and
\(\|S_\tau(y)-\theta\|_2\leq\|e\|_2+\tau\).  If
\(\|y\|_2>\tau\), then
\(S_\tau(y)-\theta=e-\tau y/\|y\|_2\), so the same bound holds.  Squaring
gives part~(ii).

Assume now \(\|e\|_2\leq\tau/2\).  In the first case the error is exactly
\(\|\theta\|_2\), and also
\(\|\theta\|_2\leq\|y\|_2+\|e\|_2\leq3\tau/2\).  In the second case the
error is at most \(3\tau/2\), while
\(\|\theta\|_2\geq\|y\|_2-\|e\|_2>\tau/2\); hence the error is at most
\(3\|\theta\|_2\).  In either case its square is bounded by both
\(9\|\theta\|_2^2\) and \(9\tau^2\), proving part~(i).
\end{proof}

\begin{lemma}[Single-block risk bound]
\label{lem:single-block-risk}
Under Assumption~\ref{ass:block-deviation}, for every $j\in\cJ$,
\[
        \E_{p_0}
        \norm{\mathsf S_{\tau_j}(\widehat g_j)-g_j}_{L^2(M)}^2
        \leq
        9\min\left\{\norm{g_j}_{L^2(M)}^2,\tau_j^2\right\}
        +2\eta_j.
\]
\end{lemma}

\begin{proof}
Fix an orthonormal basis of $E_j$ and let $\widehat\theta_j$, $\theta_j$, and
$e_j$ be the coordinate vectors of $\widehat g_j$, $g_j$, and
$\widehat g_j-g_j$. The coordinate map is an isometry, and radial thresholding
corresponds to $S_{\tau_j}$. On
$G_j=\{\norm{e_j}_2\leq\tau_j/2\}$, apply
Lemma~\ref{lem:single-block-deterministic}(i). On $G_j^c$, apply part (ii),
take expectations, and use Assumption~\ref{ass:block-deviation}.
\end{proof}

\begin{proof}[Proof of Theorem~\ref{thm:block-oracle}]
The constant block is estimated exactly. Orthogonality gives
\[
        \norm{\widetilde p_{\cJ}-p_0}_{L^2(M)}^2
        =
        \sum_{j\in\cJ}
        \norm{\mathsf S_{\tau_j}(\widehat g_j)-g_j}_{L^2(M)}^2
        +
        \sum_{j\notin\cJ,\,j\geq1}
        \norm{P_jp_0}_{L^2(M)}^2.
\]
Apply Lemma~\ref{lem:single-block-risk} block by block and take expectations.
The projected estimator satisfies the same bound because Hilbert projection
onto the closed convex set $\cD$ cannot increase distance to $p_0\in\cD$.
\end{proof}

\subsection{Proof of Proposition~\ref{prop:bernstein-threshold}}
The following elementary lemma records the tail-integration step used below.

\begin{lemma}[Tail integration]
\label{lem:block-tail-integration}
Let $Z\geq0$ and suppose that, for $x\geq1$,
\[
        \PP\{Z>u(x)\}\leq e^{-x},
\]
where $u$ is nondecreasing and $u(2x)\leq C_u u(x)$ for a constant $C_u$.
Then there are constants $C,c>0$, depending only on $C_u$, such that for every
$\xi\geq1$ and $\tau=2u(\xi)$,
\[
        \E\!\left[Z^2\bbone\{Z>\tau/2\}\right]
        +\tau^2\PP\{Z>\tau/2\}
        \leq
        C\tau^2e^{-c\xi}.
\]
\end{lemma}

\begin{proof}
Decompose $\{Z>u(\xi)\}$ into the dyadic layers
$\{u(2^k\xi)<Z\leq u(2^{k+1}\xi)\}$ and use the tail bound at
$2^k\xi$. Iterating the doubling condition gives
$u(2^{k+1}\xi)^2\leq C_u^{2(k+1)}u(\xi)^2$, while
$e^{-2^k\xi}$ dominates this geometric growth. Summing over $k\geq0$ proves
the expectation bound; the probability term follows directly from the tail
inequality at $x=\xi$.
\end{proof}

\begin{proof}[Proof of Proposition~\ref{prop:bernstein-threshold}]
Fix an orthonormal basis of $E_j$ and let $e_j\in\R^{d_j}$ be the coordinate
vector of $\widehat g_j-g_j$. Then
\[
        e_j
        =
        \frac1n\sum_{i=1}^nY_i,
        \qquad
        Y_i=\Phi_j(X_i)-\E_{p_0}\Phi_j(X_i).
\]
The $Y_i$ are independent, mean-zero Hilbert-space-valued random variables.
First,
\[
\begin{aligned}
        \E_{p_0}\norm{\Phi_j(X)}_2^2
        &=\int_MK_j(x,x)p_0(x)\,d\nu(x)\\
        &\leq b\int_MK_j(x,x)\,d\nu(x)
        =bd_j.
\end{aligned}
\]
Hence $\E\norm{e_j}_2^2\leq bd_j/n$. Moreover, for every unit vector
$u\in\R^{d_j}$,
\[
        \E_{p_0}\ip{u}{\Phi_j(X)}^2
        \leq
        b\int_M\left(\sum_{\ell=1}^{d_j}u_\ell\phi_{j,\ell}(x)\right)^2
        d\nu(x)
        =b,
\]
so the weak variance parameter of $\sum_iY_i$ is at most $nb$.

Finally,
\[
        \norm{\Phi_j(x)}_2\leq\sqrt{\kappa_j},
        \qquad
        \norm{\E_{p_0}\Phi_j(X)}_2
        \leq\sqrt{bd_j}
        \leq\sqrt b\sqrt{\kappa_j},
\]
because $\kappa_j\geq\int_MK_j(x,x)d\nu(x)=d_j$. Thus
\[
        \norm{Y_i}_2\leq(1+\sqrt b)\sqrt{\kappa_j}=:B_j.
\]
Lemma~\ref{lem:bousquet-pinelis} therefore gives
\[
        \PP\left(
        \left\|\sum_{i=1}^nY_i\right\|_2
        >C\left[\sqrt{nbd_j}+\sqrt{nbx}+B_jx\right]
        \right)
        \leq e^{-x}.
\]
After division by $n$ and absorption of constants, this is the stated tail
bound. The coordinate norm of $e_j$ equals
$\norm{\widehat g_j-g_j}_{L^2(M)}$.

For the remainder assertion, let
\[
        u_j(x)
        :=
        C\left[
        \sqrt{\frac{b(d_j+x)}n}
        +(1+\sqrt b)\frac{\sqrt{\kappa_j}\,x}{n}
        \right].
\]
This function is nondecreasing and satisfies $u_j(2x)\leq2u_j(x)$ for
$x\geq1$. Lemma~\ref{lem:block-tail-integration}, with
$\tau_j=2u_j(\xi_j)$, gives
\[
        \E\!\left[
        \norm{\widehat g_j-g_j}_{L^2(M)}^2
        \bbone\!\left\{\norm{\widehat g_j-g_j}_{L^2(M)}>\tau_j/2\right\}
        \right]
        +
        \tau_j^2\PP\!\left(
        \norm{\widehat g_j-g_j}_{L^2(M)}>\tau_j/2
        \right)
        \leq
        C'\tau_j^2e^{-c\xi_j}.
\]
\end{proof}

\subsection{Proof of Theorem~\ref{thm:multiplicity-sensitive-upper}}

\begin{lemma}[Block-variation-ball oracle optimization]
\label{lem:block-variation-oracle-optimization}
Assume Assumption~\ref{ass:poly-growth}. 
Let \(A,\delta>0\), set \(w_j=(1+\mu_j)^{s/2}\), and let
\((v_j)_{j\ge1}\) be a nonnegative sequence satisfying
\[
    0\le v_j\le C_v\delta d_j,\qquad j\ge1.
\]
Define
\[
        \Psi(A,\delta)
        :=
        \sup_{\{r_j\geq0:\,\sum_{j\geq1}w_jr_j\leq A\}}
        \sum_{j\geq1}\min\{r_j^2,v_j\}.
\]
Then
\[
        \Psi(A,\delta)
        \leq
        C\rho_{\alpha,\gamma,s}(A,\delta),
\]
where $\rho_{\alpha,\gamma,s}$ is defined in
\eqref{eq:multiplicity-sensitive-rate}. The constant depends only on
$s,\alpha,\gamma,C_d,C_D,C_v$.
\end{lemma}

\begin{proof}
For every sequence in the supremum defining \(\Psi(A,\delta)\),
\(\sum_{j\ge1} w_j r_j\le A\). Since \(w_j\ge1\) and \(r_j\ge0\),
\[
    \sum_{j\ge1}r_j^2
    \le \left(\sum_{j\ge1}r_j\right)^2
    \le \left(\sum_{j\ge1}w_jr_j\right)^2
    \le A^2.
\]
Since $v_j\lesssim\delta d_j$ and
$d_j\lesssim(1+\mu_j)^\gamma$,
\[
        \frac{\sqrt{v_j}}{w_j}
        \lesssim
        \sqrt\delta(1+\mu_j)^{(\gamma-s)/2}.
\]
We repeatedly use $\min\{r_j^2,v_j\}\leq r_j\sqrt{v_j}$.

\medskip
\noindent\textbf{Case 1: $0<s<\gamma$.}
If $A^2\leq\delta$, then
$\sum_j\min\{r_j^2,v_j\}\leq\sum_jr_j^2\leq A^2\leq\rho_{\alpha,\gamma,s}(A,\delta)$.
Assume therefore that $A^2>\delta$ and set
\[
        1+\Lambda_*
        =
        \left(\frac{A^2}{\delta}\right)^{1/(s+\gamma)}.
\]
For $\mu_j\leq\Lambda_*$,
\[
        \min\{r_j^2,v_j\}
        \lesssim
        \sqrt\delta(1+\Lambda_*)^{(\gamma-s)/2}w_jr_j,
\]
so the low-frequency contribution is at most
\[
        A\sqrt\delta(1+\Lambda_*)^{(\gamma-s)/2}
        =
        A^{\frac{2\gamma}{s+\gamma}}
        \delta^{\frac{s}{s+\gamma}}.
\]

For the high-frequency contribution, using
\(\min\{r_j^2,v_j\}\le r_j^2\), \(w_j^2=(1+\mu_j)^s\),
and the constraint \(\sum_{j\ge1}w_jr_j\le A\), we obtain
\[
\begin{aligned}
\sum_{\mu_j>\Lambda_*}\min\{r_j^2,v_j\}
&\le \sum_{\mu_j>\Lambda_*}r_j^2
=\sum_{\mu_j>\Lambda_*}(1+\mu_j)^{-s}(w_jr_j)^2\\
&\le (1+\Lambda_*)^{-s}
    \sum_{\mu_j>\Lambda_*}(w_jr_j)^2\\
&\le (1+\Lambda_*)^{-s}
    \left(\sum_{\mu_j>\Lambda_*}w_jr_j\right)^2\\
&\le A^2(1+\Lambda_*)^{-s}=
A^{\frac{2\gamma}{s+\gamma}}
\delta^{\frac{s}{s+\gamma}}.
\end{aligned}
\]

\medskip
\noindent\textbf{Case 2: $s\geq\gamma$.}
Again, if $A^2\leq\delta$, the conclusion follows from the trivial bound
$\sum_jr_j^2\leq A^2$. Otherwise set
\[
        1+\Lambda_*
        =
        \left(\frac{A}{\sqrt\delta}\right)^{1/(\alpha+(s+\gamma)/2)}.
\]
The low-frequency contribution is bounded by
\[
        \sum_{\mu_j\leq\Lambda_*}v_j
        \lesssim
        \delta(1+\Lambda_*)^{\alpha+\gamma}.
\]
Since $s\geq\gamma$, the high-frequency contribution is bounded by
\[
\begin{aligned}
        \sum_{\mu_j>\Lambda_*}r_j\sqrt{v_j}
        &\lesssim
        \sqrt\delta(1+\Lambda_*)^{-(s-\gamma)/2}
        \sum_{\mu_j>\Lambda_*}w_jr_j\\
        &\leq
        A\sqrt\delta(1+\Lambda_*)^{-(s-\gamma)/2}.
\end{aligned}
\]
The two terms are balanced by the definition of $\Lambda_*$, and substitution
gives
\[
        A^{\frac{2(\alpha+\gamma)}{s+\gamma+2\alpha}}
        \delta^{\frac{s+\alpha}{s+\gamma+2\alpha}}.
\]
\end{proof}

\begin{proof}[Proof of Theorem~\ref{thm:multiplicity-sensitive-upper}]
Apply Theorem~\ref{thm:block-oracle}. By
\eqref{eq:nonconstant-budget-constraint},
\[
\begin{aligned}
        \sum_{j\notin\cJ_n,\,j\geq1}
        \norm{P_jp_0}_{L^2(M)}^2
        &\leq
        (1+\Lambda_{\max,n})^{-s}
        \left(
        \sum_{j\notin\cJ_n}
        (1+\mu_j)^{s/2}\norm{P_jp_0}_{L^2(M)}
        \right)^2\\
        &\leq
        A_R^2(1+\Lambda_{\max,n})^{-s}.
\end{aligned}
\]
For the oracle term, set $r_j=\norm{P_jp_0}_{L^2(M)}$,
$v_j=\tau_j^2$ on $\cJ_n$, and $v_j=0$ outside $\cJ_n$. Then
$\sum_jw_jr_j\leq A_R$ and $v_j\leq C_v\delta_nd_j$, so
Lemma~\ref{lem:block-variation-oracle-optimization} gives
\[
        \sum_{j\in\cJ_n}\min\{r_j^2,\tau_j^2\}
        \leq
        C\rho_{\alpha,\gamma,s}(A_R,\delta_n).
\]
The cutoff condition makes the omitted tail no larger than the same target
rate. Adding the deviation remainders proves
\eqref{eq:main-upper-bound}; the projection step cannot increase the risk.
\end{proof}

\subsection{Proofs of the calibration corollaries}
\label{app:calibration-corollary-proofs}

\begin{proof}[Proof of Corollary~\ref{cor:homogeneous-threshold}]
This is the classical addition formula for Laplace eigenfunctions on a
compact homogeneous manifold; see \cite{gine1975addition}. Indeed, if
\(T\) is an isometry, then its unitary pullback commutes with
\(-\Delta_M\), and hence with the spectral projector \(P_j\). Therefore
\[
    K_j(Tx,Ty)=K_j(x,y).
\]
By transitivity, \(x\mapsto K_j(x,x)\) is constant. Since \(\nu(M)=1\),
\[
    K_j(x,x)
    =
    \int_M K_j(y,y)\,d\nu(y)
    =
    \operatorname{tr}(P_j)
    =
    d_j.
\]
Thus \(\kappa_j=d_j\). Substitution into \eqref{eq:bernstein-threshold} gives
\eqref{eq:homogeneous-threshold-square}.
\end{proof}

\begin{proof}[Proof of Corollary~\ref{cor:summable-coding}]
Under \eqref{eq:summable-code-local},
$d_j+\xi_j\lesssim d_j$ and
$d_j\xi_j^2/n^2\leq d_j/n$. The threshold conclusion follows from
\eqref{eq:homogeneous-threshold-square}, and
\eqref{eq:summable-code-remainder} follows from
\eqref{eq:bernstein-remainder}.
\end{proof}

\begin{proof}[Proof of Corollary~\ref{cor:explicit-no-log-examples}]
On $S^m$, $d_\ell\asymp \ell^{m-1}$, while on $SO(3)$,
$d_\ell\asymp\ell^2$. Hence $\xi_\ell\lesssim d_\ell$. Moreover, the stated
choice of $A_0$ gives
\[
        \sum_{\ell\geq1}d_\ell e^{-c\xi_\ell}<\infty,
\]
because the summand is of order $\ell^{m-1-cA_0}$ on $S^m$ and of order
$\ell^{2-cA_0}$ on $SO(3)$.

For fixed $s$, $R$, and $b$, the quantity
$\rho_{\alpha,\gamma,s}(A_R,b/n)$ is a positive power of $n^{-1}$.
Therefore the cutoff in \eqref{eq:explicit-cutoff} grows at most
polynomially in $n$. Since $\mu_\ell\asymp\ell^2$, every retained index
satisfies $\ell\leq n^q$ for some fixed $q>0$, and hence
$\xi_\ell^2=O((\log n)^2)\leq n$ for all sufficiently large $n$.
Corollary~\ref{cor:homogeneous-no-log-general} proves
\eqref{eq:no-log-examples-upper}. Finally, in both branches of
\eqref{eq:multiplicity-sensitive-rate}, the exponent of $b/n$ lies strictly
between zero and one for these manifolds. Thus the term $b/n$ is absorbed by
the leading rate for fixed $R>1$ and $b$.
\end{proof}

\subsection{Spectral growth in standard examples}
\label{sec:specializations}

On every compact homogeneous manifold with normalized volume,
$K_j(x,x)=d_j$. Consequently, the common-code calibration
$\xi_j\leq L_n$, $L_n^2\leq n$, gives
$\tau_j^2\lesssim(bL_n/n)d_j$. The summable calibration of
Corollary~\ref{cor:summable-coding} gives the sharper scale $bd_j/n$ whenever
the code conditions hold.

For the round sphere \(S^m\), the standard spherical-harmonic spectrum
\cite{chavel1984eigenvalues} is
\[
        \mu_\ell=\ell(\ell+m-1),
        \qquad
        d_\ell\asymp\ell^{m-1}.
\]
Thus
\[
        d_\ell\asymp(1+\mu_\ell)^{(m-1)/2},
        \qquad
        \sum_{\mu_\ell\leq\Lambda}d_\ell
        \asymp(1+\Lambda)^{m/2},
\]
so
\[
        \gamma=\frac{m-1}{2},
        \qquad
        \alpha=\frac12.
\]
For $m\geq2$, the choice
$\xi_\ell=A_0\log(e+\ell)$ satisfies the summable-code conditions for
$A_0$ sufficiently large.

More generally, let $G$ be a compact connected Lie group of dimension $m$ and
rank $r$, equipped with normalized Haar measure and a bi-invariant metric.
If the dimension of the full Casimir eigenspace at eigenvalue $\lambda$ obeys
\[
        d_\lambda
        \lesssim
        (1+\lambda)^{(m-r)/2+\zeta}
\]
for a collision exponent $\zeta\in[0,r/2]$, then Weyl's law gives an admissible
pair
\[
        \gamma_G=\frac{m-r}{2}+\zeta,
        \qquad
        \alpha_G=\frac r2-\zeta.
\]

For \(SO(3)\), Peter--Weyl theory and the Casimir spectrum
\cite{helgason2000groups} give irreducible representations indexed by
\(\ell\geq0\),
\[
        \mu_\ell=\ell(\ell+1),
        \qquad
        d_\ell=(2\ell+1)^2\asymp\ell^2.
\]
Hence
\[
        \gamma=1,
        \qquad
        \alpha=\frac12,
\]
and $\xi_\ell=A_0\log(e+\ell)$ again yields summable coding at the ideal
variance scale.

\section{Auxiliary results and proofs for the minimax lower bounds}

\subsection{Proof of Corollary~\ref{cor:minimax-standard-homogeneous}}
\label{app:minimax-short-proofs}

\begin{proof}
The lower bounds follow from Theorems~\ref{thm:minimax-lower-dense} and
\ref{thm:minimax-lower-rough}. The inclusion $\Pcal_s(R,1/2,3/2)\subseteq\Pcal_s(R,3/2)$ and
Theorem~\ref{thm:multiplicity-sensitive-upper} give the common-code upper
bound with $\delta_n\asymp L_n/n$.
Indeed, for the common choice
\(\xi_j=L_n=C_0\log n\), homogeneity gives, for all sufficiently
large \(n\),
\[
    \tau_j^2 \lesssim \frac{L_n}{n}d_j,
    \qquad
    \eta_j \lesssim \tau_j^2 e^{-cL_n},
    \qquad j\in J_n.
\]
Consequently,
\[
    \sum_{j\in J_n}\eta_j
    \lesssim
    \frac{L_n}{n}e^{-cL_n}
    \sum_{j\in J_n}d_j
    \le
    \frac{L_n}{n}
    D(\Lambda_{\max,n})e^{-cL_n}
    =
    o\!\left(
        \rho_{\alpha,\gamma,s}
        \!\left(A_R,\frac{L_n}{n}\right)
    \right).
\]
For the last relation, note that, for fixed \(M,s,R\), the cutoff
\(\Lambda_{\max,n}\) may be chosen to grow at most polynomially in
\(n\). Hence both \(D(\Lambda_{\max,n})\) and
\(\rho_{\alpha,\gamma,s}(A_R,L_n/n)^{-1}\) grow at most polynomially
in \(n\), whereas
\[
    e^{-cL_n}=n^{-cC_0}.
\]
Thus the displayed \(o(\,\cdot\,)\) relation holds when \(C_0\) is
chosen sufficiently large.
The summable-code upper bound follows from
Corollary~\ref{cor:homogeneous-no-log-general}.

On \(S^m\),
\[
        \mu_\ell=\ell(\ell+m-1),
        \qquad
        d_\ell\asymp\ell^{m-1},
\]
which gives \(\alpha=1/2\) and \(\gamma=(m-1)/2\). For \(m\geq2\), the code
\(\xi_\ell=A_0\log(e+\ell)\) satisfies the summability condition after
multiplication by \(d_\ell\). On \(SO(3)\),
\[
        \mu_\ell=\ell(\ell+1),
        \qquad
        d_\ell=(2\ell+1)^2\asymp\ell^2,
\]
so \(\alpha=1/2\), \(\gamma=1\), and the same coding argument applies. In
each displayed example the exponent of \(n^{-1}\) in
\(\rho_{\alpha,\gamma,s}(A_R,n^{-1})\) is strictly less than one, so the
additional \(n^{-1}\) term in \eqref{eq:minimax-summable-calibration} is
negligible.
\end{proof}

\label{app:minimax-lower-proofs}

We first isolate the three ingredients used in both lower-bound constructions:
a random spectral-window sup-norm estimate, a bounded spherical packing, and a
corresponding sup-norm estimate inside a single eigenspace.

\begin{lemma}[Random spectral-window sup-norm bound]
\label{lem:random-window-supnorm}
Fix \(c_0\in(0,1)\), and for \(\Lambda\geq1\) set
\[
        \Hcal_\Lambda
        :=
        \bigoplus_{j:\,c_0\Lambda<\mu_j\leq\Lambda}E_j.
\]
Let \(U_\Lambda\) be uniformly distributed on the unit sphere of
\(\Hcal_\Lambda\). There exist constants \(C>0\) and \(\Lambda_0\geq1\),
depending only on \((M,g)\) and \(c_0\), such that, whenever
\(\Lambda\geq\Lambda_0\) and \(\dim(\Hcal_\Lambda)\geq2\),
\begin{equation}
\label{eq:random-window-supnorm}
        \PP\!\left(
        \norm{U_\Lambda}_{L^\infty(M)}
        \leq C\sqrt{\log(2+\Lambda)}
        \right)
        \geq\frac12.
\end{equation}
\end{lemma}

\begin{proof}
Apply Theorem~2, equation~(2.9), of Burq and
Lebeau~\cite{burq2013injections} to the real unit sphere in their spectral
window
\[
        h\sqrt{\mu_j}\in(a_h,b_h].
\]
With
\[
        h=\Lambda^{-1/2},
        \qquad
        a_h=\sqrt{c_0},
        \qquad
        b_h=1,
\]
that window is exactly
\(c_0\Lambda<\mu_j\leq\Lambda\), hence its real spectral subspace is
\(\Hcal_\Lambda\). Their estimate gives, for every \(t\geq1\),
\[
        \PP\!\left(\|U_\Lambda\|_{L^\infty(M)}>t\right)
        \leq
        C h^{-c_1}e^{-c_2t^2}
        =C\Lambda^{c_1/2}e^{-c_2t^2},
        \qquad
        c_1=m\left(1+\frac m2\right),
\]
where their dimension \(d\) is denoted by \(m=\dim M\) here. Taking
\(t=C_0\sqrt{\log(2+\Lambda)}\), with \(C_0\) sufficiently large, makes
the right-hand side at most \(1/2\) for all sufficiently large \(\Lambda\).
\end{proof}

\begin{proposition}[Bounded spherical packing]
\label{prop:bounded-spherical-packing}
Let \(H\subset L^2(M)\) be an \(m\)-dimensional real Hilbert subspace, and let
\(U\) be uniformly distributed on its unit sphere. Suppose that
\[
        \PP\!\left(\norm{U}_{L^\infty(M)}\leq B\right)\geq\frac12
\]
for some \(B>0\). There exist universal constants \(c_0>0\) and
\(m_0\geq1\) such that, for every \(m\geq m_0\) and every
\(0<\varepsilon\leq(2B)^{-1}\), there is a finite set
\(\mathcal F\subset H\) satisfying
\begin{equation}
\label{eq:bounded-spherical-packing}
        \log|\mathcal F|\geq c_0m,
        \qquad
        \norm{f}_{L^2(M)}=\varepsilon,
        \qquad
        \norm{f}_{L^\infty(M)}\leq\frac12
\end{equation}
for every \(f\in\mathcal F\), and
\[
        \norm{f-g}_{L^2(M)}\geq\varepsilon
        \qquad
        \text{for all distinct }f,g\in\mathcal F.
\]
\end{proposition}

\begin{proof}
Choose an orthonormal basis of \(H\), identify its unit sphere with
\(\mathbb S^{m-1}\), and let \(U_1,\ldots,U_N\) be independent copies of
\(U\). For independent uniform vectors \(A_i,A_j\in\mathbb S^{m-1}\), the
standard concentration inequality on the sphere gives
\[
        \PP\!\left(\langle A_i,A_j\rangle>\frac12\right)
        \leq 2e^{-c_1m}
\]
for a universal \(c_1>0\); see, for example,
\cite[Theorem 3.4.5]{Vershynin2026}.
For \(1\leq i<j\leq N\), let
\[
    \mathcal B_{ij}
    :=
    \left\{
        \|U_i-U_j\|_{L^2(M)}<1
    \right\}.
\]
Writing \(U_i=\sum_{k=1}^m A_{i,k}e_k\), orthonormality gives
\[
    \|U_i-U_j\|_{L^2(M)}^2
    =
    2-2\langle A_i,A_j\rangle.
\]
Thus
\[
    \mathcal B_{ij}
    =
    \left\{
        \langle A_i,A_j\rangle>\frac12
    \right\},
\]
and hence
\[
    \mathbb P(\mathcal B_{ij})
    \leq 2e^{-c_1m}.
\]
Therefore, by the union bound over the \(\binom N2\) unordered
pairs,
\[
\begin{aligned}
    \mathbb P\!\left(
        \exists\, i<j:
        \|U_i-U_j\|_{L^2(M)}<1
    \right)
    &\leq
    2\binom N2 e^{-c_1m} \\
    &\leq
    N^2e^{-c_1m}.
\end{aligned}
\]
Choose
\[
    N=\lfloor e^{c_2m}\rfloor
\]
with \(0<c_2<c_1/2\). Then
\[
    N^2e^{-c_1m}
    \leq
    e^{-(c_1-2c_2)m}.
\]
After increasing the universal lower-dimensional threshold
\(m_0\), the last quantity is at most \(1/4\). Consequently,
\[
    \mathbb P\!\left(
        \|U_i-U_j\|_{L^2(M)}\geq1
        \ \text{for all } i\neq j
    \right)
    \geq \frac34.
\]
Set
\[
    Z_i
    :=
    \mathbf 1_{\{\|U_i\|_{L^\infty(M)}\leq B\}},
    \qquad 1\leq i\leq N.
\]
The variables \(Z_1,\dots,Z_N\) are independent Bernoulli random
variables and, by assumption,
\[
    \mathbb E Z_i
    =
    \mathbb P\!\left(\|U_i\|_{L^\infty(M)}\leq B\right)
    \geq \frac12.
\]
Hence
\[
    \mathbb E\!\left[\sum_{i=1}^N Z_i\right]\geq \frac N2.
\]
A lower-tail Chernoff bound therefore gives, for a universal
constant \(c_3>0\),
\[
    \mathbb P\!\left(\sum_{i=1}^N Z_i<\frac N4\right)
    \leq e^{-c_3N}.
\]

Let
\[
    \mathcal E_{\rm sep}
    :=
    \left\{
        \|U_i-U_j\|_{L^2(M)}\geq1
        \text{ for all }i\neq j
    \right\}
\]
and
\[
    \mathcal E_{\rm bd}
    :=
    \left\{
        \sum_{i=1}^N Z_i\geq\frac N4
    \right\}.
\]
The preceding estimates imply
\[
    \mathbb P(\mathcal E_{\rm sep})\geq\frac34,
    \qquad
    \mathbb P(\mathcal E_{\rm bd})\geq1-e^{-c_3N}.
\]
Consequently,
\[
\begin{aligned}
    \mathbb P(\mathcal E_{\rm sep}\cap\mathcal E_{\rm bd})
    &\geq
    1-\mathbb P(\mathcal E_{\rm sep}^c)
      -\mathbb P(\mathcal E_{\rm bd}^c)\\
    &\geq
    \frac34-e^{-c_3N}>0
\end{aligned}
\]
for all sufficiently large \(m\), since
\(N=\lfloor e^{c_2m}\rfloor\to\infty\).
Thus there exists a realization for which all \(U_i\)'s are
pairwise \(L^2\)-separated and at least \(N/4\) of them satisfy
\(\|U_i\|_{L^\infty(M)}\leq B\).
\end{proof}

\begin{lemma}[Random eigenspace sup-norm bound]
\label{lem:random-eigenspace-supnorm}
Let \(j\geq1\), and suppose
\[
        \kappa_j:=\sup_{x\in M}K_j(x,x)\leq C_Kd_j.
\]
If \(U_j\) is uniformly distributed on the unit sphere of \(E_j\), then there
exist constants \(C>0\) and \(\mu_\star\geq1\), depending only on
\((M,g)\) and \(C_K\), such that, whenever \(\mu_j\geq\mu_\star\),
\begin{equation}
\label{eq:random-eigenspace-supnorm}
        \PP\!\left(
        \norm{U_j}_{L^\infty(M)}
        \leq C\sqrt{\log(2+\mu_j)}
        \right)
        \geq\frac12.
\end{equation}
\end{lemma}

\begin{proof}
Let \(e_1,\ldots,e_{d_j}\) be an orthonormal basis of \(E_j\), and write
\[
        U_j=\sum_{\ell=1}^{d_j}a_\ell e_\ell,
        \qquad
        a=(a_1,\ldots,a_{d_j})\sim\operatorname{Unif}(\mathbb S^{d_j-1}).
\]
For fixed \(x\in M\), the coefficient vector
\(v(x)=(e_1(x),\ldots,e_{d_j}(x))\) satisfies
\[
        \norm{v(x)}_2^2=K_j(x,x)\leq C_Kd_j.
\]
Concentration of a linear functional on the sphere therefore gives
\begin{equation}
\label{eq:fixed-point-spherical-tail}
        \PP\!\left(|U_j(x)|>t\right)
        \leq 2e^{-c t^2/C_K},
        \qquad t\geq0,
\end{equation}
for a universal constant \(c>0\).

We next obtain a deterministic Lipschitz bound directly for \(U_j\), which
is itself an \(L^2\)-normalized eigenfunction in \(E_j\). The main theorem of
Shi and Xu~\cite{shi2010gradient} gives
\[
        \|\nabla U_j\|_{L^\infty(M)}
        \leq
        C(1+\mu_j)^{1/2}\|U_j\|_{L^\infty(M)},
\]
while the standard \(L^\infty\) eigenfunction estimate, namely the
\(p=\infty\) case of Sogge's spectral-cluster bounds
\cite{Sogge1988SpectralClusters}, gives
\[
        \|U_j\|_{L^\infty(M)}
        \leq C(1+\mu_j)^{(m-1)/4}\|U_j\|_{L^2(M)},
        \qquad m=\dim M.
\]
Since \(\|U_j\|_{L^2(M)}=1\),
\begin{equation}
\label{eq:random-eigenfunction-lipschitz}
        \norm{\nabla U_j}_{L^\infty(M)}
        \leq
        C(1+\mu_j)^{(m+1)/4}.
\end{equation}

Let
\[
        r_j=(1+\mu_j)^{-(m+1)/4},
\]
and let \(\mathcal N_j\) be a maximal \(r_j\)-separated subset of
\(M\). Then \(\mathcal N_j\) is an \(r_j\)-net, while the geodesic
balls
\[
        \left\{
        B_g(x,r_j/2):x\in\mathcal N_j
        \right\}
\]
are pairwise disjoint. Since, uniformly in \(x\in M\),
\[
        \nu(B_g(x,r))\ge c r^m
\]
for all sufficiently small \(r\), and since \(\nu(M)=1\), it follows
that
\[
        |\mathcal N_j|
        \le C r_j^{-m}
        =
        C(1+\mu_j)^{m(m+1)/4}.
\]
By \eqref{eq:fixed-point-spherical-tail} and a union bound,
\[
        \PP\!\left(
        \max_{x\in\mathcal N_j}|U_j(x)|>t
        \right)
        \leq
        C(1+\mu_j)^{m(m+1)/4}e^{-ct^2/C_K}.
\]

Set
\[
        t_j:=C_0\sqrt{\log(2+\mu_j)}.
\]
Choosing \(C_0\) sufficiently large, and then increasing
\(\mu_\star\) if necessary, the preceding union bound gives
\[
        \PP\!\left(
        \max_{y\in\mathcal N_j}|U_j(y)|\leq t_j
        \right)
        \geq \frac12,
        \qquad \mu_j\geq\mu_\star.
\]

It remains to pass from the net to all of \(M\). Let
\(C_{\mathrm{Lip}}\) denote the constant in
\eqref{eq:random-eigenfunction-lipschitz}, so that
\[
        \|\nabla U_j\|_{L^\infty(M)}
        \leq
        C_{\mathrm{Lip}}(1+\mu_j)^{(m+1)/4}.
\]
For every \(x\in M\), choose \(y_x\in\mathcal N_j\) such that
\(d_g(x,y_x)\leq r_j\). Integrating the gradient of \(U_j\)
along a minimizing geodesic joining \(y_x\) to \(x\), we obtain
\[
\begin{aligned}
        |U_j(x)-U_j(y_x)|
        &\leq
        d_g(x,y_x)\,
        \|\nabla U_j\|_{L^\infty(M)}\\
        &\leq
        r_j C_{\mathrm{Lip}}
        (1+\mu_j)^{(m+1)/4}\\
        &=
        C_{\mathrm{Lip}},
\end{aligned}
\]
where the last equality follows from
\(r_j=(1+\mu_j)^{-(m+1)/4}\). Consequently,
\[
        \|U_j\|_{L^\infty(M)}
        \leq
        \max_{y\in\mathcal N_j}|U_j(y)|
        +C_{\mathrm{Lip}}.
\]
Hence, on the preceding event,
\[
        \|U_j\|_{L^\infty(M)}
        \leq
        C_0\sqrt{\log(2+\mu_j)}
        +C_{\mathrm{Lip}}.
\]
Since
\[
        \sqrt{\log(2+\mu_j)}\geq\sqrt{\log 2},
\]
the additive constant can be absorbed into the multiplicative
constant:
\[
        C_0\sqrt{\log(2+\mu_j)}+C_{\mathrm{Lip}}
        \leq
        \left(
        C_0+\frac{C_{\mathrm{Lip}}}{\sqrt{\log 2}}
        \right)
        \sqrt{\log(2+\mu_j)}.
\]
Thus, with
\[
        C:=
        C_0+\frac{C_{\mathrm{Lip}}}{\sqrt{\log 2}},
\]
we conclude that
\[
        \PP\!\left(
        \|U_j\|_{L^\infty(M)}
        \leq C\sqrt{\log(2+\mu_j)}
        \right)
        \geq\frac12,
\]
which proves \eqref{eq:random-eigenspace-supnorm}.
\end{proof}

 \paragraph{How the auxiliary results fit together}
Both lower bounds are obtained by perturbing the uniform density as
\(p_f=1+f\), where \(f\) belongs to a high-dimensional spectral
subspace orthogonal to the constants.  The main difficulty is to find
exponentially many perturbations that are well separated in
\(L^2(M)\), while remaining uniformly small enough that
\(1/2\leq p_f\leq 3/2\).

Lemmas~\ref{lem:random-window-supnorm} and~ \ref{lem:random-eigenspace-supnorm} provide the required pointwise control in the two
spectral geometries relevant below.  Lemma~\ref{lem:random-window-supnorm} shows that a random
\(L^2\)-unit vector in a fixed-width spectral window has
\(L^\infty\)-norm of order \(\sqrt{\log(2+\Lambda)}\) with probability
bounded away from zero.  Lemma~\ref{lem:random-eigenspace-supnorm} gives the analogous conclusion for
a random unit vector in a single eigenspace, under the
projector-diagonal condition
\(\sup_x K_j(x,x)\lesssim d_j\).

Proposition~\ref{prop:bounded-spherical-packing} converts either probabilistic statement into a
deterministic testing family.  If a positive fraction of the unit
sphere of an \(m\)-dimensional subspace has \(L^\infty\)-norm at most
\(B\), then it extracts an \(L^2\)-separated packing of cardinality
exponential in \(m\); after scaling by
\(\varepsilon\leq (2B)^{-1}\), every packing element has
\(L^\infty\)-norm at most \(1/2\).  Theorem~ \ref{thm:minimax-lower-dense} applies this mechanism
to a window containing many eigenspaces, whereas Theorem~\ref{thm:minimax-lower-rough} applies
it inside one high-multiplicity eigenspace.  The remainder of each
proof chooses the frequency \(\Lambda\) and amplitude \(\varepsilon\)
so that the block-variation constraint and the Fano
Kullback--Leibler condition are simultaneously balanced.

\subsection{Proof of Theorem~\ref{thm:minimax-lower-dense}}

\begin{proof}
Let
\[
        D(\Lambda)
        :=
        \sum_{\mu_j\leq\Lambda}d_j.
\]
Since \(D(\Lambda)\to\infty\), the upper growth assumption implies
\(\alpha+\gamma>0\). Combining the upper and lower cumulative growth bounds,
we may choose a constant \(c_0\in(0,1)\) and constants
\(c_{\mathrm w},\Lambda_0>0\) such that
\begin{equation}
\label{eq:dense-window-dimension}
        \sum_{j:\,c_0\Lambda<\mu_j\leq\Lambda}d_j
        \geq
        c_{\mathrm w}(1+\Lambda)^{\alpha+\gamma},
        \qquad \Lambda\geq\Lambda_0.
\end{equation}
Indeed, choose \(c_0\) sufficiently small that the upper bound for
\(D(c_0\Lambda)\) is at most one half of the lower bound for \(D(\Lambda)\)
for all sufficiently large \(\Lambda\). Increase \(\Lambda_0\), if necessary,
so that \(c_0\Lambda\) is above the high-frequency threshold in
Assumption~\ref{ass:poly-lower-growth} whenever \(\Lambda\geq\Lambda_0\).

Let
\[
        \cJ_\Lambda
        :=
        \{j\geq1:c_0\Lambda<\mu_j\leq\Lambda\},
        \qquad
        \Hcal_\Lambda
        :=
        \bigoplus_{j\in\cJ_\Lambda}E_j,
        \qquad
        m_\Lambda:=\dim\Hcal_\Lambda.
\]
By \eqref{eq:dense-window-dimension},
\begin{equation}
\label{eq:dense-window-total-dimension}
        m_\Lambda\geq c(1+\Lambda)^{\alpha+\gamma}.
\end{equation}
Moreover, the individual lower bound in
Assumption~\ref{ass:poly-lower-growth} and the cumulative upper bound in
Assumption~\ref{ass:poly-growth} imply
\begin{equation}
\label{eq:dense-window-block-count}
        |\cJ_\Lambda|
        \leq C(1+\Lambda)^\alpha.
\end{equation}

Let \(U_\Lambda\) be uniform on the unit sphere of \(\Hcal_\Lambda\). By
Lemma~\ref{lem:random-window-supnorm},
\[
        \PP\!\left(
        \norm{U_\Lambda}_{L^\infty(M)}
        \leq B_\Lambda
        \right)
        \geq\frac12,
        \qquad
        B_\Lambda:=C\sqrt{\log(2+\Lambda)}.
\]
For \(m_\Lambda\) sufficiently large, Proposition~\ref{prop:bounded-spherical-packing}
therefore supplies, for every
\(0<\varepsilon\leq(2B_\Lambda)^{-1}\), a set
\(\mathcal F_\Lambda\subset\Hcal_\Lambda\) such that
\begin{equation}
\label{eq:dense-packing-size}
        \log|\mathcal F_\Lambda|
        \geq c(1+\Lambda)^{\alpha+\gamma},
\end{equation}
all \(f\in\mathcal F_\Lambda\) satisfy
\[
        \norm{f}_{L^2(M)}=\varepsilon,
        \qquad
        \norm{f}_{L^\infty(M)}\leq\frac12,
\]
and distinct packing elements are separated by at least \(\varepsilon\) in
\(L^2(M)\).

For \(f\in\mathcal F_\Lambda\), set \(p_f=1+f\), and also set \(p_0=1\).
Since $\Hcal_\Lambda$ is contained in the spectral subspace associated
with positive eigenvalues, it is orthogonal in $L^2(M)$ to the constant
functions and satisfies
\[
    \int_M f\,d\nu=\langle f,1\rangle_{L^2(M)}=0.
\]
Since every $f\in\Hcal_\Lambda$ has mean zero and satisfies
$\|f\|_\infty\leq 1/2$, the function $p_f:=1+f$ obeys
\[
    \int_M p_f\,d\nu=1
    \qquad\text{and}\qquad
    \frac12\leq p_f\leq\frac32.
\]
Hence $p_f$ is a probability density.

 By orthogonality,
\begin{align}
        \norm{p_f}_{\Vspec_s(M)}
        &=
        1+
        \sum_{j\in\cJ_\Lambda}
        (1+\mu_j)^{s/2}\norm{P_jf}_{L^2(M)}
        \notag\\
        &\leq
        1+(1+\Lambda)^{s/2}|\cJ_\Lambda|^{1/2}
        \left(
        \sum_{j\in\cJ_\Lambda}\norm{P_jf}_{L^2(M)}^2
        \right)^{1/2}
        \notag\\
        &\leq
        1+C\varepsilon(1+\Lambda)^{(s+\alpha)/2}.
\label{eq:dense-packing-variation}
\end{align}
Furthermore,
\begin{equation}
\label{eq:dense-packing-kl}
        \KL\!\left(p_f^{\otimes n},p_0^{\otimes n}\right)
        \leq
        n\chi^2(p_f,p_0)
        =n\varepsilon^2.
\end{equation}

Set
\[
        q_{\mathrm d}:=s+\gamma+2\alpha,
        \qquad
        1+\Lambda_n\asymp(nA_R^2)^{1/q_{\mathrm d}},
        \qquad
        \varepsilon_n
        :=
        c_\varepsilon A_R(1+\Lambda_n)^{-(s+\alpha)/2},
\]
where \(c_\varepsilon>0\) is a sufficiently small constant. For all
sufficiently large \(n\), the window lies above \(\Lambda_0\), its dimension
exceeds the universal threshold in Proposition~\ref{prop:bounded-spherical-packing},
and
\[
        \varepsilon_nB_{\Lambda_n}\longrightarrow0.
\]
Thus the packing construction applies. By
\eqref{eq:dense-packing-variation}, choosing \(c_\varepsilon\) small makes all
alternatives belong to \(\Pcal_s(R,1/2,3/2)\). Moreover,
\begin{align*}
        \frac{n\varepsilon_n^2}{\log|\mathcal F_{\Lambda_n}|}
        &\leq
        Cc_\varepsilon^2
        nA_R^2(1+\Lambda_n)^{-(s+2\alpha+\gamma)}
        \leq C'c_\varepsilon^2.
\end{align*}
Fix \(\eta_0=1/16\). After reducing \(c_\varepsilon\) once more, the
Kullback--Leibler condition in Theorem~\ref{thm:fano-kl} holds with
\(\eta=\eta_0\). Applying that theorem with separation radius
\(\varepsilon_n/2\) yields
\[
        \mathfrak R_n\bigl(\Pcal_s(R,1/2,3/2)\bigr)
        \geq c\varepsilon_n^2.
\]
Finally,
\begin{align*}
        \varepsilon_n^2
        &\asymp
        A_R^2(nA_R^2)^{-(s+\alpha)/(s+\gamma+2\alpha)}\\
        &=
        A_R^{\frac{2(\alpha+\gamma)}{s+\gamma+2\alpha}}
        n^{-\frac{s+\alpha}{s+\gamma+2\alpha}},
\end{align*}
which proves \eqref{eq:dense-minimax-lower}. 

The class constraint and the Kullback--Leibler condition require only
\[
        Cc_\varepsilon\leq 1
        \qquad\text{and}\qquad
        C'c_\varepsilon^2\leq\eta_0,
\]
respectively. Hence \(c_\varepsilon\), and therefore the multiplicative
constant in the lower bound, can be chosen independently of \(R\).
By contrast, the eventual conditions that the spectral window lie above
\(\Lambda_0\), have sufficiently large dimension, and satisfy
\(\varepsilon_nB_{\Lambda_n}\leq1/2\) may require
\(n\geq n_0(R)\).
\end{proof}

\subsection{Proof of Theorem~\ref{thm:minimax-lower-rough}}

\begin{proof}
Let \(m=\dim M\), and let
\[
        D(\Lambda)=\sum_{\mu_j\leq\Lambda}d_j
\]
denote the eigenvalue counting function, with multiplicities.
 By H\"ormander's local Weyl law
\cite[Theorem~5.1, in particular equation~(5.3)]{hormander1968spectral}, applied to
$I-\Delta_g$ and integrated over $M$, the eigenvalue counting
function satisfies
\[
    D(\Lambda)
    =
    \frac{\omega_m\Vol_g(M)}{(2\pi)^m}\Lambda^{m/2}
    +O\!\left(\Lambda^{(m-1)/2}\right),
    \qquad \Lambda\to\infty.
\]
In particular,
\begin{equation}
\label{eq:full-weyl-asymptotic}
        D(\Lambda)
        =
        \frac{\omega_m\operatorname{Vol}_g(M)}{(2\pi)^m}
        \Lambda^{m/2}
        +o\!\left(\Lambda^{m/2}\right),
        \qquad \Lambda\to\infty,
\end{equation}
where \(\omega_m\) is the Euclidean volume of the unit ball in
\(\mathbb R^m\), and \(\operatorname{Vol}_g(M)\) is the unnormalized
Riemannian volume of \(M\). Consequently,
\[
        D(\Lambda)-D(\Lambda/2)
        =
        \frac{\omega_m\operatorname{Vol}_g(M)}{(2\pi)^m}
        \left(1-2^{-m/2}\right)\Lambda^{m/2}
        +o\!\left(\Lambda^{m/2}\right)>0
\]
for every sufficiently large \(\Lambda\). Hence the interval
\((\Lambda/2,\Lambda]\) contains at least one nonconstant eigenvalue; choose
an index \(j_\Lambda\) such that
\[
        \frac\Lambda2<\mu_{j_\Lambda}\leq\Lambda.
\]
For all sufficiently large \(\Lambda\), one has
\(\mu_{j_\Lambda}>\Lambda/2\) above the high-frequency threshold in
Assumption~\ref{ass:single-block-lower}. That assumption then gives
\begin{equation}
\label{eq:rough-block-dimension}
        d_{j_\Lambda}\geq c(1+\Lambda)^\gamma,
        \qquad
        \kappa_{j_\Lambda}\leq C_Kd_{j_\Lambda}.
\end{equation}
Let \(U_\Lambda\) be uniformly distributed on the unit sphere of
\(E_{j_\Lambda}\). Lemma~\ref{lem:random-eigenspace-supnorm} gives
\[
        \PP\!\left(
        \norm{U_\Lambda}_{L^\infty(M)}
        \leq B_\Lambda
        \right)
        \geq\frac12,
        \qquad
        B_\Lambda:=C\sqrt{\log(2+\Lambda)}.
\]
For all sufficiently large \(\Lambda\), Proposition~\ref{prop:bounded-spherical-packing}
therefore yields, for each
\(0<\varepsilon\leq(2B_\Lambda)^{-1}\), a set
\(\mathcal F_\Lambda\subset E_{j_\Lambda}\) with
\begin{equation}
\label{eq:rough-packing-size}
        \log|\mathcal F_\Lambda|
        \geq c(1+\Lambda)^\gamma,
\end{equation}
whose elements have \(L^2\)-norm \(\varepsilon\), are uniformly bounded by
\(1/2\), and are pairwise \(\varepsilon\)-separated in \(L^2(M)\).

For \(f\in\mathcal F_\Lambda\), put \(p_f=1+f\), and take \(p_0=1\).
These functions are densities in \([1/2,3/2]\), and
\begin{equation}
\label{eq:rough-packing-variation}
        \norm{p_f}_{\Vspec_s(M)}
        =
        1+(1+\mu_{j_\Lambda})^{s/2}\varepsilon
        \leq
        1+C(1+\Lambda)^{s/2}\varepsilon.
\end{equation}
As in \eqref{eq:dense-packing-kl},
\begin{equation}
\label{eq:rough-packing-kl}
        \KL\!\left(p_f^{\otimes n},p_0^{\otimes n}\right)
        \leq n\varepsilon^2.
\end{equation}

Set
\[
        q_{\mathrm r}:=s+\gamma,
        \qquad
        1+\Lambda_n\asymp(nA_R^2)^{1/q_{\mathrm r}},
        \qquad
        \varepsilon_n
        :=
        c_\varepsilon A_R(1+\Lambda_n)^{-s/2},
\]
with \(c_\varepsilon>0\) sufficiently small. Since \(s>0\),
\(\varepsilon_nB_{\Lambda_n}\to0\), so the packing is available for all
sufficiently large \(n\). Equation~\eqref{eq:rough-packing-variation} places
all alternatives in \(\Pcal_s(R,1/2,3/2)\), and
\[
        \frac{n\varepsilon_n^2}{\log|\mathcal F_{\Lambda_n}|}
        \leq
        Cc_\varepsilon^2
        nA_R^2(1+\Lambda_n)^{-(s+\gamma)}
        \leq C'c_\varepsilon^2.
\]
Fix \(\eta_0=1/16\). Choosing \(c_\varepsilon\) small enough makes the
Kullback--Leibler condition in Theorem~\ref{thm:fano-kl} hold with
\(\eta=\eta_0\). Applying that theorem with separation radius
\(\varepsilon_n/2\) gives
\[
        \mathfrak R_n\bigl(\Pcal_s(R,1/2,3/2)\bigr)
        \geq c\varepsilon_n^2.
\]
Finally,
\begin{align*}
        \varepsilon_n^2
        &\asymp
        A_R^2(nA_R^2)^{-s/(s+\gamma)}\\
        &=
        A_R^{\frac{2\gamma}{s+\gamma}}
        n^{-\frac{s}{s+\gamma}},
\end{align*}
which proves \eqref{eq:rough-minimax-lower}. As in the dense-window
argument, the constant in the lower bound is independent of \(R\), while the
threshold \(n_0(R)\) may depend on \(R\).
\end{proof}

\section{Technical details for the positive likelihood extension}
\label{app:positive-likelihood}

\subsection{KL--\texorpdfstring{\(L^2\)}{L2} equivalence on bounded log-density paths}

The following standard consequence of the Bregman representation of
Kullback--Leibler divergence for exponential families is recorded with
explicit constants because it is used in both directions below; see, for
example, \cite[Proposition~3.1 and Section~5.2.2]{wainwright2008graphical},
and \cite{barron_sheu_1991} for related bounded-log-density
inequalities.

\begin{lemma}[KL--\(L^2\) equivalence on bounded sieves]
\label{lem:kl-l2-equivalence}
Assume that \(\nu(M)=1\).  Let \(u,v\) be mean-zero functions with
\(\|u\|_\infty\le B\) and \(\|v\|_\infty\le B\).  Then
\[
        \frac{e^{-2B}}2\|u-v\|_{L^2(M)}^2
        \le \KL(p_v,p_u)
        \le \frac{e^B}2\|u-v\|_{L^2(M)}^2.
\]
More generally, if along \(v_t=v+t(u-v)\) the densities satisfy
\(0<a_B\le p_{v_t}\le b_B<\infty\) for \(0\le t\le1\), then
\[
        \frac{a_B}2\|u-v\|_{L^2(M)}^2
        \le \KL(p_v,p_u)
        \le \frac{b_B}2\|u-v\|_{L^2(M)}^2.
\]
\end{lemma}

\begin{proof}
Let \(h=u-v\). Since \(u\) and \(v\) are mean-zero, \(h\) is mean-zero.
For \(t\in[0,1]\), set
\[
        v_t=v+th,
        \qquad
        \psi(t)=A(v_t).
\]
Because \(u\) and \(v\) are bounded, \(h\), \(e^{v_t}\), and the first two
\(t\)-derivatives of \(e^{v_t}\) are dominated uniformly on \([0,1]\) by
integrable constants.  Differentiation under the integral sign is therefore
justified, and
\[
        \psi'(t)=\int_M h p_{v_t}\,d\nu,
        \qquad
        \psi''(t)=\operatorname{Var}_{p_{v_t}}(h).
\]
Since
\[
        {\rm KL}(p_v,p_u)
        =
        A(u)-A(v)-\int_M (u-v)p_v\,d\nu,
\]
Taylor's formula gives
\[
        {\rm KL}(p_v,p_u)
        =
        \int_0^1 (1-t)\operatorname{Var}_{p_{v_t}}(h)\,dt .
\]
Now
\[
        \operatorname{Var}_{p_{v_t}}(h)
        =
        \inf_{c\in\mathbb R}
        \int_M (h-c)^2p_{v_t}\,d\nu .
\]
If \(a_B\leq p_{v_t}\leq b_B\), then
\[
        \operatorname{Var}_{p_{v_t}}(h)
        \geq
        a_B\inf_{c\in\mathbb R}\int_M(h-c)^2\,d\nu
        =
        a_B\|h\|_{L^2(M)}^2,
\]
because \(h\) has \(\nu\)-mean zero. Similarly,
\[
        \operatorname{Var}_{p_{v_t}}(h)
        \leq
        \int_M h^2p_{v_t}\,d\nu
        \leq
        b_B\|h\|_{L^2(M)}^2 .
\]
Integrating \(1-t\) over \([0,1]\) gives the stated general bound.

Finally, under \(\|v_t\|_\infty\leq B\) and \(\int_M v_t\,d\nu=0\), Jensen's inequality gives
\[
        0\leq A(v_t)\leq B.
\]
Hence
\[
        e^{-2B}\leq p_{v_t}(x)=e^{v_t(x)-A(v_t)}\leq e^B,
\]
so one may take \(a_B=e^{-2B}\) and \(b_B=e^B\).
\end{proof}

\subsection{Feasibility of block-sparse comparators}
\label{app:positive-feasibility}

Three constants enter the positive-sieve analysis.  The true centered
log-density satisfies \(\|u_0\|_\infty\le B\).  The estimator may be optimized
over a larger sieve radius \(B_\star\), because spectral projectors need not be
contractions on \(L^\infty(M)\).  Finally,
\[
        b_0:=\sup_{u_0\in\mathcal U_q(R,B)}
        \|p_{u_0}\|_{L^\infty(M)}
        \le e^B,
\]
where the inequality follows from \(0\le A(u_0)\le B\).

For completeness, write the growth and projector-envelope assumptions as
\begin{equation}\label{eq:positive-block-growth}
        d_j\le C_d(1+\mu_j)^\gamma,
        \qquad
        \sum_{1\le j:\mu_j\le\Lambda}d_j
        \le C_D(1+\Lambda)^{\alpha+\gamma},
        \qquad \Lambda\ge1,
\end{equation}
and
\begin{equation}\label{eq:positive-envelope}
        \kappa_j:=\sup_{x\in M}K_j(x,x)
        \le C_\kappa(1+\mu_j)^\beta,
        \qquad j\ge1.
\end{equation}
If \(g_j\in E_j\), the reproducing-kernel identity gives
\begin{equation}\label{eq:block-linfty-envelope}
        \|g_j\|_\infty
        \le \kappa_j^{1/2}\|g_j\|_{L^2(M)}.
\end{equation}
Consequently, for \(u_0\in\mathcal U_q(R,B)\), \(q\ge\beta\), and every
\(S\subset\{j\ge1\}\),
\begin{equation}\label{eq:hard-block-feasibility}
        \left\|\sum_{j\in S}P_ju_0\right\|_\infty
        \le C_\kappa^{1/2}
        \sum_{j\in S}(1+\mu_j)^{\beta/2}
             \|P_ju_0\|_{L^2(M)}
        \le C_\kappa^{1/2}R.
\end{equation}
Thus every hard-block comparator is feasible in \(F_{J_n}(B_\star)\) when
\begin{equation}\label{eq:Bstar-choice}
        B_\star\ge\max\{B,C_\kappa^{1/2}R\}.
\end{equation}
On a homogeneous manifold, \(\kappa_j=d_j\), so one may take
\(\beta=\gamma\).  The present hard-block comparator argument therefore yields
the explicit homogeneous rate only for \(q\ge\gamma\).  For
\(0<q<\gamma\), the feasible oracle inequality in
Theorem~\ref{thm:positive-sieve-oracle} remains valid, but this uniform
block-variation comparison is not automatic.  The curvature and risk constants
also depend on the sieve radius \(B_\star\); through
\eqref{eq:Bstar-choice}, this may introduce additional dependence on \(R\).

\subsection{Proof of Theorem~\ref{thm:positive-sieve-oracle}}

\begin{proof}
The compactness of \(F_J(B)\), continuity of the objective, and strict
convexity of \(A\) on the mean-zero space \(F_J\) give a unique measurable
minimizer.  The remaining argument is the standard
basic-inequality/decomposability proof for group penalties and regularized
\(M\)-estimators
\cite{yuan_lin_2006,meier_vandegeer_buhlmann_2008,VanDeGeer2008,negahban_ravikumar_wainwright_yu_2012};
we retain the steps specific to the spectral exponential family.  Put
\[
        \ell_n(v):=-\mathbb P_nv+A(v),
        \qquad
        \ell(v):=-\mathbb P_0v+A(v),
        \qquad
        \mathcal P_\lambda(v):=
        \sum_{j\in J}\lambda_j\|P_jv\|_{L^2}.
\]
For any comparator \(u\in F_J(B)\), optimality and
\(\ell(v)-\ell(u_0)=\KL(p_0,p_v)\) give
\[
        \KL(p_0,\widehat p_J)+4\mathcal P_\lambda(\widehat u_J)
        \leq
        \KL(p_0,p_u)+4\mathcal P_\lambda(u)
        +(\mathbb P_n-\mathbb P_0)(\widehat u_J-u).
\]
Let \(\delta=\widehat u_J-u\) and \(S=S(u)\).  Blockwise
Cauchy--Schwarz and the event \(\mathcal E_J\) imply
\[
        |(\mathbb P_n-\mathbb P_0)\delta|
        \leq\sum_{j\in J}\|Z_j\|_2\|P_j\delta\|_{L^2}
        \leq\mathcal P_\lambda(\delta).
\]
Decomposability gives
\[
        \mathcal P_\lambda(u)-\mathcal P_\lambda(\widehat u_J)
        \leq
        \mathcal P_{\lambda,S}(\delta)
        -\mathcal P_{\lambda,S^c}(\delta),
\]
and therefore
\begin{equation}\label{eq:positive-basic-bound}
        \KL(p_0,\widehat p_J)
        \leq
        \KL(p_0,p_u)+5\mathcal P_{\lambda,S}(\delta).
\end{equation}
Writing \(L_S^2=\sum_{j\in S}\lambda_j^2\), orthogonality yields
\[
        \mathcal P_{\lambda,S}(\delta)
        \leq L_S\|\delta\|_{L^2}.
\]
Set \(K=\KL(p_0,\widehat p_J)\) and
\(a=\|u-u_0\|_{L^2}\).  Lemma~\ref{lem:kl-l2-equivalence} and the triangle
inequality give
\[
        \|\delta\|_{L^2}\leq\sqrt2e^B K^{1/2}+a,
        \qquad
        \KL(p_0,p_u)\leq\frac{e^B}{2}a^2.
\]
Substitution into \eqref{eq:positive-basic-bound} and two applications of
Young's inequality give
\[
        K\leq(e^B+1)a^2+(50e^{2B}+25)L_S^2.
\]
This is \eqref{eq:positive-kl-oracle}, for example with
\(C_B=\max\{e^B+1,50e^{2B}+25\}\).  Finally,
\(h^2(\widehat p_J,p_0)\leq\KL(p_0,\widehat p_J)\); taking the infimum over
\(u\in F_J(B)\) proves \eqref{eq:positive-hellinger-oracle}.
\end{proof}

\subsection{Proof of Corollary~\ref{cor:positive-sieve-rate}}

\begin{proof}
Fix \(u_0\in\mathcal U_q(R,B)\) and write \(p_0=p_{u_0}\).  Set
\[
        r_j:=\|P_ju_0\|_{L^2(M)},
        \qquad
        w_j:=(1+\mu_j)^{q/2}.
\]
Since \(u_0\in\Vspec_q(R)\) and \(P_0u_0=0\),
\begin{equation}\label{eq:positive-block-variation-constraint}
        \sum_{j\ge1}w_jr_j\le R.
\end{equation}

On \(\mathcal E_{J_n}\), Theorem~\ref{thm:positive-sieve-oracle}, applied
with radius \(B_\star\), gives
\[
        \KL(p_{u_0},\widehat p_{J_n})
        \le
        C_{B_\star}
        \inf_{u\in F_{J_n}(B_\star)}
        \left[
        \|u-u_0\|_{L^2(M)}^2
        +\sum_{j\in S(u)}\lambda_j^2
        \right],
\]
where \(S(u):=\{j\in J_n:P_ju\ne0\}\).

For any subset \(S\subset J_n\), define the hard-block comparator
\[
        u_S:=\sum_{j\in S}P_ju_0.
\]
It is mean-zero, and \eqref{eq:hard-block-feasibility} together with the choice
of \(B_\star\) gives \(\|u_S\|_\infty\le B_\star\).  Hence
\(u_S\in F_{J_n}(B_\star)\).  Orthogonality and optimization over all
\(S\subset J_n\) yield
\begin{equation}\label{eq:positive-hard-block-oracle}
        \inf_{u\in F_{J_n}(B_\star)}
        \left[
        \|u-u_0\|_{L^2(M)}^2
        +\sum_{j\in S(u)}\lambda_j^2
        \right]
        \le
        \sum_{j\in J_n}\min\{r_j^2,\lambda_j^2\}
        +\sum_{j\notin J_n}r_j^2.
\end{equation}
The tail is bounded by
\[
        \sum_{j\notin J_n}r_j^2
        \le
        (1+\Lambda_{\max,n})^{-q}
        \left(\sum_{j\notin J_n}w_jr_j\right)^2
        \le
        R^2(1+\Lambda_{\max,n})^{-q}
        \le\rho_n^{\log}(R,\delta_n).
\]
Moreover, by \eqref{eq:lambda-delta-condition},
\[
        \sum_{j\in J_n}\min\{r_j^2,\lambda_j^2\}
        \le
        C
        \sup_{\{s_j\ge0:\,\sum_{j\ge1}w_js_j\le R\}}
        \sum_{j\ge1}\min\{s_j^2,\delta_nd_j\}.
\]
Lemma~\ref{lem:block-variation-oracle-optimization}, applied with smoothness
parameter \(q\), gives
\[
        \sup_{\{s_j\ge0:\,\sum_{j\ge1}w_js_j\le R\}}
        \sum_{j\ge1}\min\{s_j^2,\delta_nd_j\}
        \le C\rho_n^{\log}(R,\delta_n),
\]
where \(\rho_n^{\log}\) is defined in
\eqref{eq:positive-log-rate}.  Combining the preceding displays, on
\(\mathcal E_{J_n}\),
\begin{equation}\label{eq:positive-kl-on-event}
        \KL(p_{u_0},\widehat p_{J_n})
        \le C\rho_n^{\log}(R,\delta_n).
\end{equation}

On the complement event, both \(u_0\) and \(\widehat u_{J_n}\) are mean-zero
and bounded by \(B_\star\).  Lemma~\ref{lem:kl-l2-equivalence} therefore gives
the deterministic bound
\[
        \KL(p_{u_0},\widehat p_{J_n})
        \le
        \frac{e^{B_\star}}2
        \|u_0-\widehat u_{J_n}\|_{L^2(M)}^2
        \le 2e^{B_\star}B_\star^2.
\]
Taking expectations in \eqref{eq:positive-kl-on-event} and using
\eqref{eq:event-complement-rate} proves, uniformly over
\(u_0\in\mathcal U_q(R,B)\),
\[
        \mathbb E_{p_{u_0}}\KL(p_{u_0},\widehat p_{J_n})
        \le C\rho_n^{\log}(R,\delta_n).
\]
for all sufficiently large \(n\).  Since
\(h^2(\widehat p_{J_n},p_{u_0})\le
\KL(p_{u_0},\widehat p_{J_n})\), the same bound holds for the Hellinger risk.

Finally, \(0\le A(v)\) for every mean-zero \(v\), so
\(p_v\le e^{B_\star}\) whenever \(\|v\|_\infty\le B_\star\).  Hence
\[
        \|p-q\|_{L^2(M)}^2
        =\int_M(\sqrt p-\sqrt q)^2(\sqrt p+\sqrt q)^2\,d\nu
        \le C_{B_\star}h^2(p,q)
\]
for \(p=p_v\) and \(q=p_w\) with \(\|v\|_\infty,\|w\|_\infty\le B_\star\).
Applying this to \(\widehat p_{J_n}\) and \(p_{u_0}\) proves the
\(L^2\)-risk bound.  The dependence of the final constant is precisely of the
form stated in Corollary~\ref{cor:positive-sieve-rate}.
\end{proof}

\subsection{Penalty calibration and homogeneous specializations}
\label{app:positive-calibration}

The quantity \(\delta_n\) in Corollary~\ref{cor:positive-sieve-rate} is an
upper bound, per spectral degree of freedom, for the squared block penalties.
By Proposition~\ref{prop:bernstein-threshold}, applied to
\(Z_j=(\mathbb P_n-\mathbb P_0)\Phi_j\), one may choose, for coding levels
\(\xi_j\ge1\),
\[
        \lambda_j
        =K\left[
        \sqrt{\frac{b_0(d_j+\xi_j)}n}
        +(1+\sqrt{b_0})\frac{\sqrt{\kappa_j}\,\xi_j}{n}
        \right],
        \qquad j\in J_n,
\]
with \(K\) sufficiently large.  Uniformly over
\(u_0\in\mathcal U_q(R,B)\),
\[
        \mathbb P_{p_{u_0}}(\mathcal E_{J_n}^c)
        \le\sum_{j\in J_n}e^{-\xi_j},
\]
and
\[
        \lambda_j^2
        \le C\left[
        \frac{b_0(d_j+\xi_j)}n
        +\frac{(1+b_0)\kappa_j\xi_j^2}{n^2}
        \right].
\]
Hence \eqref{eq:lambda-delta-condition} holds whenever
\[
        \delta_n
        \ge C\max_{j\in J_n}
        \left\{
        \frac{b_0(d_j+\xi_j)}{nd_j}
        +\frac{(1+b_0)\kappa_j\xi_j^2}{n^2d_j}
        \right\}.
\]
If \(M\) is homogeneous, \(\kappa_j=d_j\).  If additionally
\[
        \max_{j\in J_n}\xi_j\le L_n,
        \qquad L_n\ge1,
        \qquad L_n^2\le n,
\]
then
\[
        \lambda_j^2\le C\frac{b_0L_n}{n}d_j,
        \qquad
        \delta_n=\frac{b_0L_n}{n}\le\frac{e^BL_n}{n},
\]
up to universal numerical factors.  The event condition
\eqref{eq:event-complement-rate} follows whenever
\[
        \sum_{j\in J_n}e^{-\xi_j}
        =o\!\left(\rho_n^{\log}(R,\delta_n)\right).
\]

\begin{proof}[Proof of Corollary~\ref{cor:positive-homogeneous-rate}]
On a homogeneous manifold, \(\kappa_j=d_j\). Assumption~\ref{ass:poly-growth}
therefore gives the projector-envelope condition with \(\beta=\gamma\) and
\(C_\kappa=C_d\), while the stipulated choice of \(B_\star\) makes every
hard-block comparator feasible. Set \(b_0=e^B\) and use the common code
\(\xi_j=L_n\). For all sufficiently large \(n\), one has \(L_n^2\leq n\),
and the penalty calibration above, together with
\eqref{eq:positive-homogeneous-penalty}, gives
\[
        \lambda_j^2
        \leq C\frac{e^BL_n}{n}d_j,
        \qquad j\in J_n.
\]
Thus \eqref{eq:lambda-delta-condition} holds with
\(\delta_n=e^BL_n/n\). Since \(q\geq\gamma\), the second branch of
\(\rho_{\alpha,\gamma,q}\) applies and equals
\(\rho_n^{\mathrm{hom},+}(R,B)\) up to a structural constant. The cutoff
\eqref{eq:positive-homogeneous-cutoff} gives
\[
        R^2(1+\Lambda_{\max,n})^{-q}
        \leq \rho_n^{\mathrm{hom},+}(R,B),
\]
so the tail condition \eqref{eq:positive-tail-choice} holds.

It remains to verify the event-complement condition. Uniformly over
\(u_0\in\mathcal U_q(R,B)\),
\[
        \mathbb P_{p_{u_0}}(\mathcal E_{J_n}^c)
        \leq \#J_n e^{-L_n}
        \leq C_D(1+\Lambda_{\max,n})^{\alpha+\gamma}n^{-c_0}.
\]
Write \(\vartheta_q=(q+\alpha)/(q+\gamma+2\alpha)\). For fixed
\(R,B\), equations \eqref{eq:positive-homogeneous-rate} and
\eqref{eq:positive-homogeneous-cutoff} imply
\[
        1+\Lambda_{\max,n}
        \lesssim
        n^{\vartheta_q/q}(\log n)^{-\vartheta_q/q}
\]
for all sufficiently large \(n\). Consequently,
\[
        \mathbb P_{p_{u_0}}(\mathcal E_{J_n}^c)
        \lesssim
        n^{-c_0+\vartheta_q(\alpha+\gamma)/q}
        (\log n)^{-\vartheta_q(\alpha+\gamma)/q}
        =o\!\left(\rho_n^{\mathrm{hom},+}(R,B)\right),
\]
where the last step is exactly the stated lower bound on \(c_0\). All
conditions of Corollary~\ref{cor:positive-sieve-rate} are therefore
satisfied, and that corollary proves the three asserted risk bounds.
\end{proof}

No rough-regime branch is asserted, because the present hard-block comparator
argument uses \(q\geq\beta=\gamma\).

For the round sphere \(S^m\), \(\alpha=1/2\) and
\(\gamma=(m-1)/2\), so for \(q\ge(m-1)/2\),
\[
\begin{aligned}
        &\sup_{u_0\in\mathcal U_q(R,B)}
        \mathbb E_{p_{u_0}}\KL(p_{u_0},\widehat p_{J_n})\\
        &\quad\vee
        \sup_{u_0\in\mathcal U_q(R,B)}
        \mathbb E_{p_{u_0}}h^2(\widehat p_{J_n},p_{u_0})\\
        &\quad\vee
        \sup_{u_0\in\mathcal U_q(R,B)}
        \mathbb E_{p_{u_0}}
        \|\widehat p_{J_n}-p_{u_0}\|_{L^2(S^m)}^2\\
        &\qquad\lesssim
        R^{\frac{2m}{2q+m+1}}
        \left(\frac{e^BL_n}{n}\right)^{
        \frac{2q+1}{2q+m+1}}.
\end{aligned}
\]
In particular, on \(S^2\) the exponents are
\(4/(2q+3)\) and \((2q+1)/(2q+3)\).  On \(SO(3)\),
\(\alpha=1/2\) and \(\gamma=1\), so for \(q\ge1\),
\[
\begin{aligned}
        &\sup_{u_0\in\mathcal U_q(R,B)}
        \mathbb E_{p_{u_0}}\KL(p_{u_0},\widehat p_{J_n})\\
        &\quad\vee
        \sup_{u_0\in\mathcal U_q(R,B)}
        \mathbb E_{p_{u_0}}h^2(\widehat p_{J_n},p_{u_0})\\
        &\quad\vee
        \sup_{u_0\in\mathcal U_q(R,B)}
        \mathbb E_{p_{u_0}}
        \|\widehat p_{J_n}-p_{u_0}\|_{L^2(SO(3))}^2\\
        &\qquad\lesssim
        R^{\frac{3}{q+2}}
        \left(\frac{e^BL_n}{n}\right)^{
        \frac{q+1/2}{q+2}}.
\end{aligned}
\]


\section{Hilbert-space Bernstein inequality}

Bousquet's inequality for suprema of empirical processes yields the
following standard form when it is applied to the symmetric class of linear
functionals on a Hilbert space~\cite{bousquet2002bennett}. Related
Banach-space martingale inequalities appear in~\cite{pinelis1994optimum}.

\begin{lemma}[Hilbert-space Bernstein inequality]\label{lem:bousquet-pinelis}
	Let \(H\) be a separable Hilbert space, and let
	\(Y_1,\ldots,Y_n\) be independent mean-zero \(H\)-valued random
	variables. Assume that
	\[
	\|Y_i\|_H \le L \qquad \text{almost surely for all } i .
	\]
	Set
	\[
	S_n=\sum_{i=1}^n Y_i,
	\qquad
	\sigma^2
	=
	\sup_{\|u\|_H\le 1}
	\sum_{i=1}^n
	\mathbb E \langle u,Y_i\rangle_H^2 .
	\]
	Then, for every \(x>0\),
	\[
	\mathbb P\left\{
	\|S_n\|_H
	\ge
	\mathbb E\|S_n\|_H
	+
	\sqrt{2(\sigma^2+2L\mathbb E\|S_n\|_H)x}
	+
	\frac{Lx}{3}
	\right\}
	\le e^{-x}.
	\]
	Consequently, for a universal constant \(C>0\),
	\[
	\mathbb P\left\{
	\|S_n\|_H
	\ge
	C\left[
	\mathbb E\|S_n\|_H
	+
	\sqrt{\sigma^2 x}
	+
	Lx
	\right]
	\right\}
	\le e^{-x}.
	\]
\end{lemma}
\section{Fano lower bound}
\label{app:fano-lower-bound}

\begin{theorem}[Fano-type lower bound]
\label{thm:fano-kl}
For convenient reference, we restate the form of Fano's lemma given in
\cite[Theorem~2.5]{tsybakov2008nonparametric}.
Let \(N\geq2\), and suppose that a parameter space \(\Theta\), endowed with a
semimetric \(d\), contains elements \(\theta_0,\ldots,\theta_N\) such that
\begin{enumerate}
\item
\(d(\theta_i,\theta_j)\geq2r\) for every pair of distinct indices
\(i,j\in\{0,\ldots,N\}\);
\item
for \(j=1,\ldots,N\), the corresponding probability measures satisfy
\(P_j\ll P_0\) and
\[
        \frac1N\sum_{j=1}^N\KL(P_j,P_0)
        \leq\eta\log N
\]
for some \(0<\eta<1/8\).
\end{enumerate}
Then, for every nondecreasing function \(\psi:\R_+\to\R_+\),
\[
        \inf_{\widehat\theta}
        \sup_{\theta\in\Theta}
        \E_\theta\psi\!\left(d(\theta,\widehat\theta)\right)
        \geq
        \psi(r)
        \frac{\sqrt N}{1+\sqrt N}
        \left(
        1-2\eta-\sqrt{\frac{2\eta}{\log N}}
        \right).
\]
\end{theorem}

\section{Additional Experiments}
\label{sec:additional-experiments}

We next illustrate why thresholding complete eigenspaces, rather than
individual coefficients in an arbitrarily chosen eigenbasis, is statistically
consequential. We focus on the degree-$12$ eigenspace $E_{12}$ on
$\mathbb S^2$, whose dimension is $d_{12}=25$. Throughout this experiment,
the target density, the observations, and the degree-$12$ spectral component
are kept fixed. Only the orthonormal basis used to represent this component
is changed.

Let $\theta_{12}\in\mathbb R^{25}$ and
$\widehat\theta_{12}\in\mathbb R^{25}$ denote respectively the true and
empirical coefficient vectors in a fixed real orthonormal basis of $E_{12}$.
For every $k\in\{1,\ldots,25\}$, we construct an orthogonal matrix
$Q_k\in O(25)$ such that
\[
    Q_k^\top\theta_{12}
    =
    \frac{\|\theta_{12}\|_2}{\sqrt{k}}
    \bigl(
        \underbrace{1,\ldots,1}_{k\ \mathrm{entries}},
        0,\ldots,0
    \bigr)^\top.
\]
Hence $k=1$ corresponds to a signal-aligned eigenbasis, in which the whole
degree-$12$ component is supported on a single coordinate, whereas $k=25$
corresponds to a maximally spread representation of the same component.
Importantly, this construction changes only the coefficient representation:
the underlying function and the observations remain unchanged.

As a basis-dependent benchmark, we consider coordinatewise soft thresholding
in the rotated basis,
\[
    \widetilde\theta_{12,Q_k}^{\mathrm{coord}}
    =
    Q_k
    \eta_{\lambda_k}
    \bigl(Q_k^\top\widehat\theta_{12}\bigr),
\]
where
\[
    \eta_\lambda(z)_r
    =
    \operatorname{sign}(z_r)
    \bigl(|z_r|-\lambda\bigr)_+.
\]
We compare this estimator with radial block soft thresholding,
\[
    \widetilde\theta_{12}^{\mathrm{block}}
    =
    \left(
        1-
        \frac{\tau_{12}}{\|\widehat\theta_{12}\|_2}
    \right)_+
    \widehat\theta_{12}.
\]
The latter is unchanged under every $Q_k$, since both the shrinkage factor
and the loss are invariant under orthogonal transformations.

To separate the effect of the eigenbasis from the effect of threshold
calibration, we calibrate the two procedures to the same blockwise null
false-activation probability $\alpha=0.05$. In this diagnostic experiment
only, the block threshold is defined by the empirical quantile as follows
\begin{equation*}
    \tau_{12}^{\mathrm{null}}
    =
    \widehat q_{1-\alpha}
    \left(
        \|\widehat\theta_{12}\|_2
        \,\middle|\,
        p_0\equiv1
    \right),
\end{equation*} 
whereas the coordinatewise threshold for the basis $Q_k$ is
\[
    \lambda_k^{\mathrm{null}}
    =
    \widehat q_{1-\alpha}
    \left(
        \|Q_k^\top\widehat\theta_{12}\|_\infty
        \,\middle|\,
        p_0\equiv1
    \right).
\]
Here $\widehat q_{1-\alpha}$ denotes the empirical $(1-\alpha)$-quantile
computed from independent samples under the uniform density. 

The thresholds are estimated from $1{,}000$ independent null samples, and
their false-activation probabilities are evaluated on a separate collection
of $2{,}000$ null samples. The resulting blockwise false-activation
probability is $0.046$ for block thresholding and ranges from $0.044$ to
$0.066$ for coordinatewise thresholding across the considered bases, close
to the nominal level $0.05$. Risks are evaluated at $n=32{,}000$ using
$500$ independent samples from the target density. The same empirical
coefficient vectors are used for all $Q_k$, making the comparison paired.

Figure~\ref{fig:sensitivity-experiments}\textbf{(a)} shows that the risk of
coordinatewise soft thresholding depends strongly on the selected
eigenbasis. Its mean degree-$12$ risk increases monotonically from
approximately $3.22\times10^{-4}$
in the signal-aligned basis to $2.86\times10^{-3}$
in the maximally spread basis. By contrast, the mean block-thresholding risk
is approximately
$1.11\times10^{-3}$
and is constant across all basis rotations; the largest observed numerical
discrepancy was below $2.2\times10^{-18}$. The coordinatewise risk becomes
larger than the block risk between $k=3$ and $k=4$.

The experiment is not intended to show that block thresholding uniformly
dominates every coordinatewise procedure. Indeed, coordinatewise thresholding
can be substantially more accurate in the oracle-aligned basis $k=1$.
Rather, it shows that this advantage relies on an unknown and geometrically
arbitrary choice of eigenbasis. The block estimator removes this dependence:
its performance is unchanged when the same spectral component is represented
sparsely or diffusely within the eigenspace.

In the next experiment, we vary the multiplicative constant
$c_\tau$ in the threshold $\tau_\ell$:
\[
    \tau_\ell
    =
    c_{\tau}
    \left[
        \sqrt{
            \frac{b(d_\ell+\xi_n)}{n}
        }
        +
        (1+\sqrt b)
        \frac{\sqrt{d_\ell}\,\xi_n}{n}
    \right],
\]
over
\[
    c_\tau\in\{0.75,0.90,1.05\},
\]
while keeping all other parameters unchanged. To reduce Monte Carlo
variability in this comparison, all three estimators are evaluated using the
same empirical spectral coefficients. The results are shown in
Figure~\ref{fig:sensitivity-experiments}\textbf{(b)}, together with the exact expected
risk of the dense estimator with fixed cutoff $L=12$ and the exact oracle
dense risk obtained by minimizing over the nested family of spectral cutoffs.

\begin{figure}[H]
    \centering
    \includegraphics[width=\linewidth]{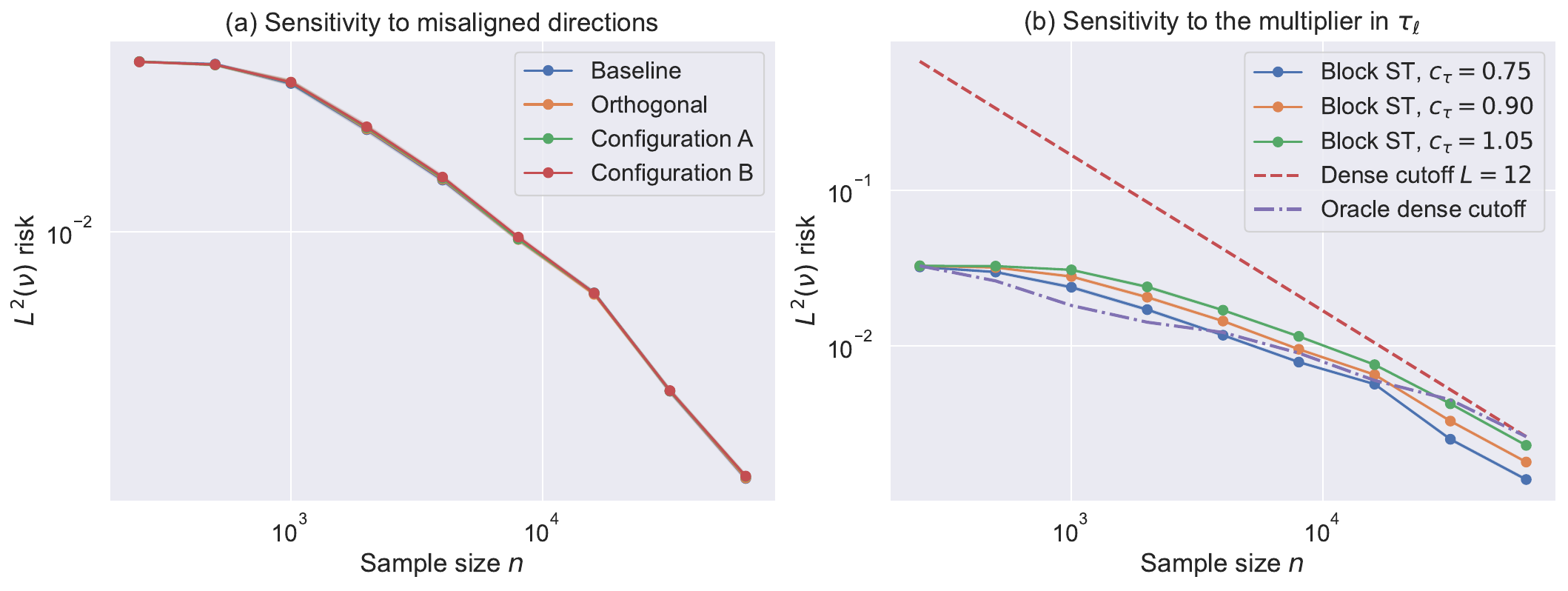}
    \caption{
Sensitivity experiments.
\textbf{(a)} Dependence of the degree-$12$ risk on the orthonormal
eigenbasis of $E_{12}$ at $n=32{,}000$. The target density, the observations,
and the degree-$12$ spectral component are identical for all values of $k$;
only the basis is changed. The basis $Q_k$ is chosen so that the true
coefficient vector has $k$ equal nonzero coordinates, ranging from the
signal-aligned representation $k=1$ to the maximally spread representation
$k=25$. Coordinatewise soft thresholding is applied in each rotated basis,
whereas block soft thresholding is radial and therefore basis invariant.
For each procedure, the threshold is calibrated under the uniform density
to a blockwise false-activation probability $\alpha=0.05$. Calibration uses
$1{,}000$ null samples, and risks are estimated from $500$ independent
samples from the target density. The shaded regions are pointwise $95\%$
Student-$t$ confidence intervals.
\textbf{(b)} Mean $L^2(\nu)$ risks of the block-shrinkage estimator for
threshold multipliers $c_\tau\in\{0.75,0.90,1.05\}$, together with the
exact expected risks of the dense estimator with fixed cutoff $L=12$ and
the oracle dense estimator using the risk-minimizing cutoff. The three
block-shrinkage procedures are evaluated using the same Monte Carlo samples.
}
\label{fig:sensitivity-experiments}
\end{figure}

In this finite-block example, the smaller multiplier $c_\tau=0.75$ yields
the lowest risk throughout the considered sample-size grid. For instance,
at $n=64{,}000$, the risks corresponding to
$c_\tau=0.75$, $0.90$, and $1.05$ are approximately
$1.38\times10^{-3}$, $1.79\times10^{-3}$, and
$2.29\times10^{-3}$, respectively, whereas the risk of the dense estimator
with cutoff $L=12$ is $2.62\times10^{-3}$. A smaller multiplier allows the
weaker active blocks to be retained earlier, while a larger multiplier
produces more aggressive shrinkage and consequently a larger bias.


\end{appendix}
\end{document}